\documentclass[11pt,letterpaper]{article}
\usepackage[letterpaper,margin=1in]{geometry}
\usepackage[utf8]{inputenc}
\usepackage[T1]{fontenc}
\usepackage[dvipsnames]{xcolor}

\usepackage{amsmath,amssymb,amsthm,mathtools,mathrsfs}
\usepackage{enumitem,graphicx,booktabs,float}
\usepackage{tikz}
\usetikzlibrary{arrows.meta,calc,positioning,decorations.pathreplacing}
\usepackage{microtype}
\microtypesetup{expansion=false}
\usepackage[numbers,square]{natbib}
\setcitestyle{numbers,open={[},close={]},citesep={,}}
\usepackage{mathptmx}
\usepackage{courier}
\usepackage{aliascnt}
\usepackage[unicode=true,colorlinks=true,breaklinks=true]{hyperref}
\definecolor{selectedpurple}{HTML}{79538A}
\definecolor{curvepurple}{HTML}{6F3F8F}
\hypersetup{linkcolor=selectedpurple!85!black,citecolor=YellowOrange!85!black,
  urlcolor=Aquamarine!85!black,
  pdftitle={Silver Rate Is (Almost) Optimal for Gradient Descent: The Strongly Convex Case},
  pdfauthor={Kaizhao Liu, Yuhan Ye}}
\newtheorem{theorem}{Theorem}[section]
\newaliascnt{lemma}{theorem}
\newtheorem{lemma}[lemma]{Lemma}
\aliascntresetthe{lemma}
\newaliascnt{proposition}{theorem}
\newtheorem{proposition}[proposition]{Proposition}
\aliascntresetthe{proposition}
\newaliascnt{corollary}{theorem}
\newtheorem{corollary}[corollary]{Corollary}
\aliascntresetthe{corollary}
\newaliascnt{fact}{theorem}
\newtheorem{fact}[fact]{Fact}
\aliascntresetthe{fact}
\theoremstyle{definition}
\newaliascnt{definition}{theorem}

\aliascntresetthe{definition}
\newaliascnt{remark}{theorem}

\aliascntresetthe{remark}
\usepackage[nameinlink,capitalize]{cleveref}
\AddToHook{env/theorem/begin}{\crefalias{section}{theorem}}
\AddToHook{env/lemma/begin}{\crefalias{section}{lemma}}
\AddToHook{env/proposition/begin}{\crefalias{section}{proposition}}
\AddToHook{env/corollary/begin}{\crefalias{section}{corollary}}
\AddToHook{env/fact/begin}{\crefalias{section}{fact}}
\AddToHook{env/definition/begin}{\crefalias{section}{definition}}
\AddToHook{env/remark/begin}{\crefalias{section}{remark}}
\crefname{fact}{Fact}{Facts}
\Crefname{fact}{Fact}{Facts}
\crefname{appendix}{Appendix}{Appendices}
\Crefname{appendix}{Appendix}{Appendices}
\allowdisplaybreaks
\newcommand{\psil}{p_{\mathrm{sil}}}
\renewcommand{\d}{\mathrm{d}}
\newcommand{\R}{\mathbb R}
\newcommand{\N}{\mathbb N}
\newcommand{\Fclass}{\mathcal F}
\newcommand{\Derr}{D}
\newcommand{\Eerr}{E}
\DeclareMathOperator{\conv}{conv}
\DeclareMathOperator{\env}{env}

\newcommand{\subsectionnotoc}[1]{%
  \let\scsavedaddcontentsline\addcontentsline
  \renewcommand{\addcontentsline}[3]{}%
  \subsection{#1}%
  \let\addcontentsline\scsavedaddcontentsline
}

\title{Silver Rate Is (Almost) Optimal for Gradient Descent:\\ The Strongly Convex Case}
\author{Kaizhao Liu\footnotemark[1]\\MIT\\\texttt{mrzt@mit.edu}
  \and Yuhan Ye\thanks{Authors are listed in random order.}\\MIT\\\texttt{yyh03@mit.edu}}
\date{September 22, 2026}

\begin{document}
\maketitle
\begin{abstract}
Understanding the limitations of gradient descent (GD) has long been a central problem in optimization.
On smooth strongly convex objectives, a recent breakthrough by Altschuler and Parrilo~\citep{altschuler2025hedging} showed that GD can be accelerated to $O(\kappa^{1/p_{\mathrm{sil}}}\log(1/\varepsilon))$ iterations using the \textit{silver stepsize schedule}, where $\kappa$ is the condition number, $\varepsilon$ is the target relative accuracy, and $p_{\mathrm{sil}}=\log_2(1+\sqrt2)$, improving the polynomial dependence on $\kappa$ compared to the classical $O(\kappa\log(1/\varepsilon))$ bound for constant-step GD.
However, the best-known lower bound has remained $\Omega(\sqrt{\kappa}\log(1/\varepsilon))$, established by Nemirovsky and Yudin~\citep{NemirovskyYudin1979} over four decades ago.

In this paper, we prove that GD with any predetermined stepsize schedule requires
\[
 \Omega\!\left(
 \kappa^{\frac{1}{p_{\mathrm{sil}}}-o(1)}
 \log\frac1\varepsilon\right)
\]
iterations to reduce either the squared distance to the minimizer or function suboptimality by a factor $\varepsilon$.
This matches the polynomial exponent of $\kappa$ achieved by the silver stepsize schedule.

As in our previous work on smooth convex objectives~\citep{ye2026silver}, we construct the hard instance from a chain of \emph{local hard functions}, where the next coordinate in the chain becomes active at selected steps called \textit{checkpoints}. 
The main difficulty arises from the additional quadratic term required by strong convexity, whose step-by-step contraction of the iterates toward the minimizer can destroy the lower bound if the total stepsize between consecutive checkpoints is too large.
We overcome this by sequentially scanning the stepsize schedule \emph{block by block} to choose the checkpoints, while inserting \emph{repair components} at carefully chosen steps whenever the accumulated unselected stepsize grows too large.
\end{abstract}

\clearpage
\tableofcontents
\clearpage
\section{Introduction}\label{sec:introduction}
\textbf{Gradient descent} (GD) is one of the simplest and most fundamental algorithms in optimization.
To minimize a differentiable function $F:\R^d\to\R$, GD proceeds via the iteration
\begin{equation*}
 x_t=x_{t-1}-h_t\nabla F(x_{t-1}),\qquad 1\le t\le n,
\end{equation*}
where the stepsize $h_t$ controls how far the iterate moves along the negative gradient direction.

In this paper, we study the fundamental limits of GD for smooth strongly convex minimization when its \emph{stepsize schedule} $H=(h_1,\ldots,h_n)\in[0,\infty)^n$ is predetermined.
Throughout, we assume $F$ is $\mu$-strongly convex and $L$-smooth, with unique minimizer $x_*$.
Let $\kappa:=L/\mu$ denote the condition number.

\paragraph{Classical Upper Bounds.}
The convergence analysis of GD has a long history.
For quadratic objectives, Young~\citep{Young1953} showed that Chebyshev stepsizes reduce the squared distance to the minimizer by a factor of $0<\varepsilon\le1/4$ in $O(\sqrt\kappa\log(1/\varepsilon))$ iterations.
\footnote{These stepsizes are the reciprocals of the roots of a degree-$n$ Chebyshev polynomial of the first kind, shifted and scaled to $[\mu,L]$.}

For general smooth strongly convex objectives, classical guarantees for constant-step GD are established by bounding the progress of each iteration separately~\citep{Polyak1963}.
This yields an $O(\kappa\log(1/\varepsilon))$ iteration complexity for reducing either the squared distance to the minimizer or function suboptimality by a factor $0<\varepsilon\le1/4$; see also the textbooks~\citep{Polyak1987,Nesterov2004,Bubeck2015}.

Further improving this rate, particularly its dependence on $\kappa$, has long motivated the development of accelerated first-order methods.
Standard approaches often add momentum or auxiliary sequences, as in Polyak's heavy-ball method~\citep{Polyak1964} and Nesterov's accelerated method~\citep{Nesterov1983}.
The latter achieves an $O(\sqrt\kappa\log(1/\varepsilon))$ iteration complexity.
These methods, however, modify the original GD update.

\paragraph{Acceleration via Stepsize Schedules.}
Compared with general smooth strongly convex objectives, quadratic objectives have additional structure~\citep[Chapter~8]{Altschuler2018}, allowing Chebyshev stepsizes to attain the accelerated $O(\sqrt\kappa\log(1/\varepsilon))$ iteration complexity discussed above.
Beyond quadratics, however, whether predetermined stepsizes alone could accelerate GD remained open for decades.

Altschuler observed that analyzing steps individually fails to capture the benefits from coordinating short and long steps in his thesis~\citep[Chapter~8]{Altschuler2018}.
A joint analysis of multiple steps is therefore necessary to establish any improvement over the best constant stepsize.
Building on this observation, Altschuler and Parrilo~\citep{altschuler2025hedging} achieved a major breakthrough by introducing the recursive \emph{silver stepsize schedule}.
This schedule improves the iteration complexity from $O(\kappa\log(1/\varepsilon))$ to $O(\kappa^{1/\psil}\log(1/\varepsilon))$, where
\begin{equation*}
 \psil:=\log_2(1+\sqrt2)\approx1.27155.
\end{equation*}
In the same work, they also conjectured its optimality.
The classical first-order lower bound $\Omega(\sqrt\kappa\log(1/\varepsilon))$ \cite{NemirovskyYudin1979,Nesterov2004} leaves a polynomial gap with respect to the silver upper bound.
Therefore, the central question is whether we can close this gap by proving a stronger lower bound for GD.

\paragraph{Progress on Lower Bounds.}
We briefly review the history of lower bound constructions and the foundational ideas that motivate our approach.
For a quadratic objective $F(x)=\tfrac12(x-x_*)^\top A(x-x_*)$, the component of the error along an eigenvector of the Hessian $A$ with eigenvalue $\lambda$ is scaled by $\prod_{t=1}^n(1-h_t\lambda)$ after $n$ GD steps.
Among predetermined stepsize schedules, the Chebyshev schedule minimizes the maximum absolute value of this product over $\lambda\in[\mu,L]$, yielding the tight $\Omega(\sqrt\kappa\log(1/\varepsilon))$ lower bound~\citep{Gavurin1950}.

Two one-dimensional quadratics are sufficient to yield a lower bound for constant-step GD for the general class of smooth strongly convex functions.
Indeed, any constant stepsize $h$ satisfies~\citep{CrockettChernoff1955}
\[
\max\{|1-h\mu|,|1-hL|\}\ge\frac{\kappa-1}{\kappa+1}.
\]
Thus, for at least one of the functions $F(x)=\mu x^2/2$ and $F(x)=Lx^2/2$, the squared distance after $n$ steps remains at least $((\kappa-1)/(\kappa+1))^{2n}$ times the initial squared error.
These examples yield the $\Omega(\kappa\log(1/\varepsilon))$ lower bound for constant-step GD on the full function class. See~\citep[Section~2.1]{Altschuler2018} for details.

For the broader class of \emph{deterministic first-order methods}, which access the objective through function values and gradients at arbitrary points, Nemirovsky and Yudin~\citep{NemirovskyYudin1979} established the classical $\Omega(\sqrt\kappa\log(1/\varepsilon))$ lower bound.
Their \emph{resisting oracle} chooses the unrevealed coordinate directions as new queries arrive, even when those queries depend adaptively and nonlinearly on previous function values and gradients.
These choices keep the unrevealed part of the minimizer orthogonal to the queried points and returned gradients, without changing the function values or gradients at any previously queried point.
For methods satisfying the \emph{linear-span assumption}\footnote{\label{fn:linear-span}The assumption is $x_t\in x_0+\operatorname{span}\{\nabla F(x_0),\ldots,\nabla F(x_{t-1})\}$~\citep[Assumption~2.1.4]{Nesterov2004}.}, Nesterov~\citep{Nesterov2004} provides a simpler quadratic-chain construction.
By initializing at the origin, each iteration discovers at most one new coordinate direction, while the minimizer has nonzero coordinates further along the chain.
The gap between the optimal objective value over the first $n$ coordinates and that over the full space establishes the desired lower bound.

Subsequent work established finite-horizon lower bounds with exact constants for squared distance to the minimizer~\citep{DroriTaylor2022} and function suboptimality~\citep{Drori2026}.
These works also employ a resisting-oracle argument similar to that of Nemirovsky and Yudin~\citep{NemirovskyYudin1979} to establish their lower bounds for arbitrary deterministic first-order methods.
To obtain sharp constants,~\citep{DroriTaylor2022} constructs smooth strongly convex functions from finitely many points alongside their corresponding function values and gradients.
This approach is intimately connected to the interpolation techniques used in \emph{performance estimation problems} (PEPs), discussed in~\cref{sec:performance-estimation}.


Our recent work~\citep{ye2026silver} proved an $\Omega(n^{-\psil-o(1)})$ lower bound in the smooth convex setting, matching the polynomial exponent of the corresponding silver upper bound.
Although upper bounds can be transferred between smooth convex and strongly convex optimization via standard black-box reductions~\citep{AllenZhuHazan2016}, these techniques do not readily yield an analogous lower bound in the strongly convex regime.
Instead, one should manually extend the lower-bound construction, where the primary difficulty becomes accommodating the quadratic term required to ensure strong convexity.

\subsection{Our Contributions}
We prove a lower bound of $\Omega(\kappa^{1/\psil-o(1)}\log(1/\varepsilon))$ for the iteration complexity of GD with predetermined nonnegative stepsizes.
This matches the silver exponent in the upper bound of Altschuler and Parrilo~\citep{altschuler2025hedging}.

To state our results, let $\Fclass_{\mu,L}(\R^d)$ denote the class of differentiable $\mu$-strongly convex functions on $\R^d$ with $L$-Lipschitz gradients.
For a normalized\footnote{\label{fn:normalization}Dividing the objective by $L$ and multiplying each stepsize by $L$ preserves the iterates. We therefore normalize $L=1$ and $\mu=1/\kappa$ for our strongly convex results.} schedule $H$, which may depend on $\kappa$ and the horizon $n$, define the worst-case ratios for squared distance to the minimizer and function suboptimality, each relative to its initial value, by
\begin{align}
 \Derr_{n,\kappa}(H)
 &:=\sup_{d\in\N}\ \sup_{F\in\Fclass_{1/\kappa,1}(\R^d)}\ \sup_{x_0\ne x_*}
 \frac{\|x_n-x_*\|^2}{\|x_0-x_*\|^2},\label{eq:distance-risk}\\*
 \Eerr_{n,\kappa}(H)
 &:=\sup_{d\in\N}\ \sup_{F\in\Fclass_{1/\kappa,1}(\R^d)}\ \sup_{x_0\ne x_*}
 \frac{F(x_n)-F(x_*)}{F(x_0)-F(x_*)}.\notag
\end{align}

Both ratios satisfy the following lower bound.
\begin{theorem}[Convergence lower bound]\label{thm:main}
There are absolute constants $C_1,C_2,\kappa_0>0$ such that, for every $\kappa\ge\kappa_0$, every integer $n\ge1$, and every predetermined nonnegative schedule $H\in[0,\infty)^n$,
\begin{equation}\label{eq:main}
 \min\{\Derr_{n,\kappa}(H),\Eerr_{n,\kappa}(H)\}
 \ge \exp\!\left[-C_1 n\,
       \kappa^{-1/\psil}\log\kappa\,e^{C_2\sqrt{\log\kappa}}\right].
\end{equation}
\end{theorem}

For a target reduction factor $0<\varepsilon<1$, define the iteration complexities by
\begin{align*}
 T_{\Derr}(\kappa,\varepsilon)
 &:=\inf\{n\in\N_{\ge1}:\ \exists H\in[0,\infty)^n\text{ with }
                   \Derr_{n,\kappa}(H)\le\varepsilon\},\\
 T_{\Eerr}(\kappa,\varepsilon)
 &:=\inf\{n\in\N_{\ge1}:\ \exists H\in[0,\infty)^n\text{ with }
                   \Eerr_{n,\kappa}(H)\le\varepsilon\}.
\end{align*}
\begin{corollary}[Iteration complexity]\label{cor:complexity}
There are absolute $c,C>0$ such that, for sufficiently large $\kappa$ and every $0<\varepsilon<1$,
\begin{equation}\label{eq:main-complexity}
 \min\{T_{\Derr}(\kappa,\varepsilon),T_{\Eerr}(\kappa,\varepsilon)\}
 \ge c\,\frac{\kappa^{1/\psil}}{e^{C\sqrt{\log\kappa}}\log\kappa}
                \log\frac1\varepsilon.
\end{equation}
\end{corollary}
The $\Omega(\kappa^{1/\psil-o(1)}\log(1/\varepsilon))$ iteration lower bound in \cref{cor:complexity} matches the $O(\kappa^{1/\psil}\log(1/\varepsilon))$ upper bound of~\citep{altschuler2025hedging} up to a subpolynomial factor in $\kappa$.\footnote{The gap to the silver upper bound is the factor $e^{C\sqrt{\log\kappa}}\log\kappa=\kappa^{o(1)}$ in \eqref{eq:main-complexity}.}
The lower and upper bounds contain the same multiplicative factor $\log(1/\varepsilon)$.
\Cref{tab:complexity-comparison} compares our iteration lower bound with the best known asymptotic upper and lower bounds across different objective and algorithm classes.

\begin{table}[!htbp]
\centering
\small
\setlength{\tabcolsep}{5pt}
\renewcommand{\arraystretch}{1.1}
\newcommand{\boundcell}[2]{%
  \begin{tabular}[c]{@{}c@{}}#1\\[2pt]{\footnotesize #2}\end{tabular}}
\begin{tabular*}{\textwidth}{@{\extracolsep{\fill}}llcc@{}}
\toprule
Objective class & Algorithm class & Lower bound & Upper bound \\
\midrule
General & Constant-step GD
  & \boundcell{$\Omega\!\left(\kappa\log(1/\varepsilon)\right)$}{\citep{CrockettChernoff1955}}
  & \boundcell{$O\!\left(\kappa\log(1/\varepsilon)\right)$}{\citep{Polyak1963}} \\
\addlinespace[7pt]
Quadratic & Predetermined-step GD
  & \boundcell{$\Omega\!\left(\sqrt\kappa\log(1/\varepsilon)\right)$}{\citep{Gavurin1950}}
  & \boundcell{$O\!\left(\sqrt\kappa\log(1/\varepsilon)\right)$}{\citep{Young1953}} \\
\addlinespace[7pt]
General & Deterministic first-order methods
  & \boundcell{$\Omega\!\left(\sqrt\kappa\log(1/\varepsilon)\right)$}{\citep{NemirovskyYudin1979}}
  & \boundcell{$O\!\left(\sqrt\kappa\log(1/\varepsilon)\right)$}{\citep{Nesterov1983}} \\
\addlinespace[7pt]
General & Predetermined-step GD
  & \boundcell{{\color{curvepurple}$\displaystyle\Omega\!\left(\frac{\kappa^{1/\psil}\log(1/\varepsilon)}{e^{C\sqrt{\log\kappa}}\log\kappa}\right)$}}{\color{curvepurple}\textbf{This work}}
  & \boundcell{$O\!\left(\kappa^{1/\psil}\log(1/\varepsilon)\right)\vphantom{\displaystyle\Omega\!\left(\frac{\kappa^{1/\psil}\log(1/\varepsilon)}{e^{C\sqrt{\log\kappa}}\log\kappa}\right)}$}{\citep{altschuler2025hedging}} \\
\bottomrule
\end{tabular*}
\caption[Iteration complexity on smooth strongly convex objectives.]{Iteration complexity on smooth strongly convex objectives.\protect\footnotemark}
\label{tab:complexity-comparison}
\end{table}
\footnotetext{All bounds hold for sufficiently large $\kappa$ and concern reducing squared distance to the minimizer by a factor $0<\varepsilon\le1/4$. Some classical bounds are folklore. With limited access to early literature, we cite the earliest relevant sources we could verify.}

\subsection{Related Work}\label{sec:related-work}

\subsubsection{The special case of quadratic optimization.}
On a quadratic objective with Hessian $A=\nabla^2 F$, the error $x_n-x_*$ after $n$ GD steps is a polynomial in $A$ applied to the initial error $x_0-x_*$.
Young~\citep{Young1953} achieved the optimal $O(\sqrt\kappa\log(1/\varepsilon))$ complexity by choosing stepsizes as reciprocals of the roots of a shifted and scaled Chebyshev polynomial.
Early formulations of the associated Chebyshev minimax problem appear in~\citep{Gavurin1950,GolubVarga1961}, with a modern treatment in~\citep{Altschuler2018}.
Conjugate gradients use a three-term recurrence, with coefficients that adapt to the particular problem instance~\citep{HestenesStiefel1952}.
Polyak's heavy-ball method adds momentum to GD and, with suitable fixed parameters, attains the accelerated asymptotic distance factor $(\sqrt\kappa-1)/(\sqrt\kappa+1)$ on strongly convex quadratics~\citep{Polyak1964}.
A three-term momentum recurrence with suitably chosen iteration-dependent coefficients can also generate the normalized Chebyshev polynomials.

Another approach uses the arcsine distribution, the limiting distribution of Chebyshev roots, to design randomized stepsize schedules~\citep{PronzatoZhigljavsky2011,ZhigljavskyPronzatoBukina2013,Kalousek2017}.
In exact arithmetic, permuting a fixed set of stepsizes does not change GD's final iterate on a quadratic objective.
However, the order still affects numerical stability, motivating the study of stable Chebyshev orderings~\citep{LebedevFinogenov1971,AgarwalGoelZhang2021}.
For general nonquadratic objectives, the order can affect convergence itself.
This is one reason why the quadratic analysis does not directly carry over to general smooth strongly convex objectives; see~\citep[Chapter~8]{Altschuler2018} for a comparison.

\subsubsection{General smooth (strongly) convex optimization.}
With a suitable constant stepsize, GD reaches normalized function suboptimality at most $\varepsilon$ in $O(1/\varepsilon)$ iterations for smooth convex objectives.\footnote{For $L$-smooth convex objectives, this accuracy target is $F(x_n)-F(x_*)\le\varepsilon(L/2)\|x_0-x_*\|^2$.}
For smooth strongly convex objectives, the iteration bound is $O(\kappa\log(1/\varepsilon))$ for reducing function suboptimality by a factor $\varepsilon$~\citep{Polyak1987,Nesterov2004,Bubeck2015}.
Early analyses cover local convergence~\citep{CrockettChernoff1955}, global convergence for strongly convex objectives~\citep{Polyak1963}, and projected gradient methods on bounded convex sets~\citep{LevitinPolyak1966}.
An adaptive alternative is exact line search, which chooses the stepsize that minimizes the objective along the current descent direction.
However, its tight worst-case iteration complexity on smooth strongly convex objectives remains $\Theta(\kappa\log(1/\varepsilon))$ for reducing function suboptimality by a factor $\varepsilon$~\citep{deKlerkGlineurTaylor2017}.

The first provable gains from predetermined stepsize schedules beyond quadratic objectives were established in Altschuler's thesis~\citep[Chapter~8]{Altschuler2018}, which analyzed two- and three-step schedules in the strongly convex setting.
Repeating these short schedules improves constants in the classical $O(\kappa\log(1/\varepsilon))$ iteration complexity.
The remaining challenge was to improve the exponent of $\kappa$.
A breakthrough came when Altschuler and Parrilo~\citep{altschuler2025hedging} introduced the silver stepsize schedule, improving this complexity from $O(\kappa\log(1/\varepsilon))$ to $O(\kappa^{1/\psil}\log(1/\varepsilon))$.
They conjectured that the exponent $1/\psil$ is optimal
among all predetermined stepsize schedules.
For separable smooth strongly convex objectives, they further showed that GD with i.i.d.\ arcsine-distributed inverse stepsizes satisfies $\bigl(\|x_n-x_*\|/\|x_0-x_*\|\bigr)^{1/n}\to(\sqrt\kappa-1)/(\sqrt\kappa+1)$ almost surely~\citep{AltschulerParrilo2024RandomStepsizes}.
We refer to their survey~\citep{AltschulerParrilo2026Survey} for a broader overview.

In the general smooth convex setting, Grimmer~\citep{Grimmer2024} proved that periodic long steps improve the function suboptimality bound by a constant factor when the initial level set is bounded.
Under additional assumptions\footnote{This result applies to the smallest function suboptimality among the iterates, $\min_{0\le t\le n}\{f(x_t)-f(x_*)\}$.
It also assumes that the set $\{x:f(x)\le f(x_0)\}$ is bounded.}, Grimmer, Shu, and Wang~\citep{GrimmerShuWang2023LongSteps} proved an $O(n^{-1.0564})$ rate for function suboptimality, improving over the classical $O(n^{-1})$ rate.
Concurrently, Altschuler and Parrilo~\citep{AltschulerParrilo2025} extended their silver schedule from strongly convex to general smooth convex objectives, achieving the $O(n^{-\psil})$ rate.
\citep{GrimmerShuWang2025ObjectiveGradient} subsequently improved the constant in the bound on function suboptimality and extended the silver rate to squared gradient norm.
\citep{WangMaYangZhou2026} determined the exact worst-case bound on function suboptimality of the original silver schedule at horizons $n=2^k-1$.
\citep{GrimmerShuWang2025Composing} and~\citep{ZhangJiang2024Concatenation} gave rules for combining shorter schedules to obtain the silver rate for arbitrary prescribed horizons.
Under these symmetric composition rules, \citep{LiuChenJiangWang2026} proved that recursively joining schedules whose lengths differ by at most one minimizes the worst-case bound on function suboptimality at each horizon.
Related developments include anytime acceleration~\citep{ZhangLeeDuChen2025} and accelerated proximal gradient methods~\citep{BokAltschuler2025}.

Acceleration through stepsize choice has also been studied in learning problems and min-max optimization.
For linear regression, periodic schedules combining short and long steps exploit clustered Hessian spectra to accelerate GD~\citep{Oymak2021}.
For linearly separable logistic regression, a large constant stepsize chosen for the iteration horizon yields an $\widetilde O(n^{-2})$ loss bound~\citep{WuBartlettTelgarskyYu2024}.
Related guarantees hold for certain two-layer neural networks trained with logistic loss~\citep{CaiWuMeiLindseyBartlett2024}.
Large constant stepsizes also accelerate strongly convex regularized logistic regression on linearly separable data under suitable regularization conditions~\citep{WuMarionBartlett2025}.
More recently, prescribed increasing stepsize schedules yield stretched exponential loss decay for separable logistic regression~\citep{KaleManupriyaMarionBachRaj2026}.
Beyond minimization, power-law stepsize schedules accelerate the last-iterate convergence of extragradient for unconstrained biaffine min-max problems~\citep{WuZhengCaiLuo2026}.

\subsubsection{Performance estimation.}\label{sec:performance-estimation}
Performance estimation problems (PEPs) formulate the worst-case performance of a given first-order method after a prescribed number of iterations as an optimization problem~\citep{DroriTeboulle2014}.
For \emph{fixed-step methods}\footnote{These methods use predetermined coefficients to combine past gradients and may use momentum.}, \emph{exact interpolation conditions} characterize which function values and gradients at finitely many points can come from a single function in the specified class~\citep{TaylorHendrickxGlineur2017}.
These conditions allow the PEP to be reformulated as a semidefinite program.
Solving the resulting program numerically or analytically can yield a tight convergence guarantee.
Optimizing this guarantee over the method's coefficients then provides a systematic way to design first-order methods with optimal convergence guarantees within a chosen class~\citep{DroriTaylor2020}.

The PEP framework has led to sharper upper bounds and, in several settings, methods that attain the corresponding lower bounds.
In the general smooth convex setting, \citep{KimFessler2016} used this approach to develop the optimized gradient method (OGM), whose bound on function suboptimality was later proved optimal in both rate and constant~\citep{Drori2017}.
\citep{GrimmerJoPark2026} subsequently established OGM's uniqueness in the corresponding minimax problem over fixed-step methods.
Branch-and-bound has also been used to numerically solve the minimax problem for GD with $n\le25$ steps, yielding predetermined stepsizes that minimize the worst-case function suboptimality~\citep[Section~6.1]{DasGuptaVanParysRyu2024}.
On smooth strongly convex objectives, the PEP framework also yielded a method attaining the exact lower bound on squared distance to the minimizer for deterministic first-order methods~\citep{TaylorDrori2023,DroriTaylor2022}.
For function suboptimality in the same setting, \citep{Drori2026} derived exact lower bounds and methods attaining them.
For a broader perspective on minimax optimality, see the survey~\citep{GrimmerShuWang2026Survey}.

\subsubsection{Asymptotic lower bounds for GD.}
GD with predetermined stepsizes forms a strict subclass of deterministic first-order methods, so classical first-order lower bounds do not determine GD's optimal rate.
For smooth convex objectives, \citep{GrimmerShuWang2025Composing} proved that the exponent $\psil$ is optimal for function suboptimality within their recursively generated class of schedules.
More recent lower bounds apply to arbitrary predetermined nonnegative schedules.
\citep{MaChen2026} proved an $\Omega(n^{-1.9319})$ lower bound on function suboptimality for predetermined nonnegative stepsizes, giving the first polynomial improvement over the classical $\Omega(n^{-2})$ bound.
\citep{Tsai2026NonAnytime} sharpened this to $\Omega(n^{-\sqrt3})$ using the same construction.
\citep{YeLiu2026} improved this lower bound to $\Omega(n^{-1.6342})$, even for possibly negative stepsizes.
\citep{jung2026stronger} subsequently introduced a family of hard functions built from chains of one-sided Huber components that gives an $\Omega(n^{-\log_2(1+\sqrt3)})$ lower bound for positive schedules.
Our previous work~\citep{ye2026silver} further refined this construction to obtain an $\Omega(n^{-\psil-o(1)})$ lower bound.
This shows that the polynomial exponent $\psil$ in the $O(n^{-\psil})$ upper bound is optimal.
\Cref{sec:existing-techniques} reviews these construction ideas, which form the starting point for our strongly convex hard function.

\section{Overview of Proof Idea}\label{sec:overview}
In this section, we first review the lower bound constructions for smooth convex objectives.
We then explain how to adapt our construction~\citep{ye2026silver} to the strongly convex setting.
The main difficulty comes from the quadratic term needed to make our hard function strongly convex.

\subsection{A Review of Existing Techniques}\label{sec:existing-techniques}
\paragraph{Nesterov's quadratic zero-chain.}
The classical $\Omega(n^{-2})$ lower bound for smooth convex optimization is due to Nemirovsky and Yudin~\citep{NemirovskyYudin1979}.
Here, we recall the chain construction for methods satisfying the linear-span assumption\footref{fn:linear-span}, as presented in~\citep[Section~2.1.2]{Nesterov2004}.
For example, with $d=2n+1$, consider the convex, $1$-smooth function
\[
 f(x)=\frac18\left[(1-x^{(1)})^2+
 \sum_{j=1}^{d-1}(x^{(j)}-x^{(j+1)})^2+(x^{(d)})^2\right].
\]
For $0\le r<d$, if $x^{(r+1)}=\cdots=x^{(d)}=0$, then $(\nabla f(x))^{(j)}=0$ for $j\ge r+2$.\footnote{This is the \emph{first-order zero-chain} property~\citep[Definition~1]{DroriTaylor2022}: the gradient can introduce at most one new coordinate.}
Consequently, GD started at $x_0=0$ satisfies $x_t^{(j)}=0$ for $j>t$, regardless of the stepsizes.
The minimizer has $x_*^{(j)}=1-j/(d+1)>0$ in every coordinate.
Minimizing $f$ over vectors whose coordinates after $n$ are zero gives
\[
 \frac{2[f(x_n)-f(x_*)]}{\|x_0-x_*\|^2}
 \ge\frac{3}{16(n+1)^2}.
\]
The strongly convex counterpart yields the classical $\Omega(\sqrt\kappa\log(1/\varepsilon))$ iteration lower bound~\citep[Section~2.1.4]{Nesterov2004}.

\paragraph{The Checkpoint Method \citep{MaChen2026}.}
For GD with predetermined nonnegative stepsizes, the classical $\Omega(n^{-2})$ lower bound also applies.
Whether this bound could be improved for such methods was a longstanding question.
Ma and Chen~\citep{MaChen2026} introduced a \textit{checkpoint method} that improved it to $\Omega(n^{-1.9319})$.
This method delays the activation of a new coordinate until a selected step, allowing the lower bound to depend on the given stepsizes.
We briefly recall the method.
Given a stepsize schedule $H=(h_1,\ldots,h_n)$, we select a set of stepsize indices $T=\{t_1<\cdots<t_k\}$ called \emph{checkpoints}. The sequence of updates between successive checkpoints forms a \emph{gap}. For each chosen checkpoint set, the construction gives a hard function and an associated \emph{score}; larger scores yield stronger lower bounds on function suboptimality.
To obtain the best possible lower bound, the checkpoint method seeks to maximize this score over all possible checkpoint sets $T$.

To make this maximization tractable, these hard functions are constructed so that the score depends solely on the stepsize at each checkpoint and the sum of the stepsizes within each gap. To formalize this, let $t_0=0$ and $t_{k+1}=n+1$, and define the \emph{checkpoint stepsizes} $b_i$ and \emph{sums of gap stepsizes} $s_i$ by
\begin{equation*}b_i:=h_{t_i}\quad(1\le i\le k),\qquad s_i:=\sum_{t=t_{i-1}+1}^{t_i-1}h_t\quad(1\le i\le k+1).\end{equation*}

\begin{figure}[htbp]
\centering
\begingroup
\definecolor{gapgray}{HTML}{DCE5ED}
\definecolor{gapedge}{HTML}{8B9DAE}
\begin{tikzpicture}[x=1cm,y=1cm,font=\small,inner sep=1pt,
  checkpoint/.style={fill=selectedpurple!85,draw=selectedpurple,line width=.4pt},
  gapstep/.style={fill=gapgray,draw=gapedge,line width=.4pt},
  gapbrace/.style={draw=black!65,line width=.45pt,
    decorate,decoration={brace,mirror,amplitude=3pt}}]
\path[use as bounding box] (-.10,-1.16) rectangle (16.05,1.96);

\foreach \x/\h in {.18/.43,.82/.88,2.04/.58,
                     3.35/.68,4.00/1.00,5.32/.42,
                     9.19/.47,9.89/.77,11.25/1.08,
                     12.70/.64,13.39/.93,15.42/.51}{
  \draw[gapstep] (\x,0) rectangle ({\x+.20},\h);
}
\foreach \x in {1.47,4.68,7.21,10.56,14.48}{
  \node at (\x,.18) {$\cdots$};
}

\foreach \x/\h/\index in {2.69/1.30/1,5.98/1.64/2,
                          8.49/.92/{k-1},11.93/1.47/k}{
  \draw[checkpoint] (\x,0) rectangle ({\x+.22},\h);
  \node[anchor=south] at ({\x+.11},{\h+.08}) {$t_{\index}$};
  \node[anchor=north] at ({\x+.11},-.12)
    {$b_{\index}=h_{t_{\index}}$};
}
\node[anchor=north] at (.28,-.12) {$h_1$};
\node[anchor=north] at (15.52,-.12) {$h_n$};

\draw[gapbrace] (.11,-.71)--(2.30,-.71)
  node[midway,below=5pt] {$s_1$};
\draw[gapbrace] (3.25,-.71)--(5.61,-.71)
  node[midway,below=5pt] {$s_2$};
\node at (7.20,-.89) {$\cdots$};
\draw[gapbrace] (9.08,-.71)--(11.55,-.71)
  node[midway,below=5pt] {$s_k$};
\draw[gapbrace] (12.59,-.71)--(15.72,-.71)
  node[midway,below=5pt] {$s_{k+1}$};
\end{tikzpicture}
\endgroup
\caption{The checkpoint stepsizes are $b_i=h_{t_i}$ (purple). Each brace marks the sum of gap stepsizes $s_i$.}
\label{fig:checkpoint-notation}
\end{figure}

Ma and Chen~\citep{MaChen2026} construct a smooth convex hard function whose gradient remains constant within each gap. Consequently, the displacements within a gap simply accumulate and depend only on the sum of gap stepsizes $s_i$. At checkpoint $t_i$, the step of size $b_i$ forces GD to explore a new dimension. This strategy is reminiscent of Nesterov's lower bound construction: to sustain a strong lower bound, GD cannot be allowed to make continuous progress in the same dimension throughout the run. For a geometric explanation of Ma and Chen's construction, see~\citep[Appendix~A]{YeLiu2026}.
When identifying a checkpoint set that maximizes the score, this intuition also provides a natural heuristic: updates with larger stepsizes are strong candidates. If an update features a large stepsize while the gradient continues in the same direction, it allows GD to travel too far along that dimension, decreasing the function value too rapidly.

However, while this heuristic is conceptually straightforward, rigorously identifying the checkpoint set that maximizes the score is a highly nontrivial combinatorial problem. Our previous work~\citep{YeLiu2026} gives a finer analysis of the factors coupling consecutive checkpoints in the same construction, improving the lower bound from $\Omega(n^{-1.9319})$ to $\Omega(n^{-1.6342})$.

\paragraph{Transfer via Local Hard Functions \citep{jung2026stronger}.}
\citep{jung2026stronger} further improved the lower bound of \citep{MaChen2026} to $\Omega(n^{-\log_2(1+\sqrt3)})$, where $\log_2(1+\sqrt3)\approx1.4500$.
Note that the hard function in~\citep{MaChen2026} is constructed globally, which introduces complex dependencies between neighboring gaps into the resulting score.
To simplify these dependencies, \citep{jung2026stronger} construct a chain of local hard functions that couple only adjacent coordinates, much like Nesterov's quadratic zero-chain above.
Our construction uses the same chain structure, so we first describe the general framework and notation here.
In the smooth convex setting, the objective is the sum of these two-dimensional functions and a final one-dimensional component:
\begin{equation}\label{eq:overview-convex-function}
 F(x):=\frac12\sum_{i=1}^k\Phi_i(x^{(i)},x^{(i+1)})
       +\frac12\Phi_{k+1}(x^{(k+1)}),\qquad x_0=e_1,\qquad x_*=0.
\end{equation}
Here $k=|T|\le n$ is the number of checkpoints.
The initial point $e_1$ is the first coordinate vector.
We call each $\Phi_i$ a \emph{local hard function}, or simply a \emph{component} of $F$.
Each should be nonnegative, convex, and $1$-smooth.
For $1\le i\le k$, $\Phi_i:\mathbb R^2\to\mathbb R$ acts on the adjacent coordinates $(X,Y)=(x^{(i)},x^{(i+1)})$, while $\Phi_{k+1}:\mathbb R\to\mathbb R$ acts only on $x^{(k+1)}$.
We call $X$ and $Y$ the \textit{input} and \textit{output} coordinates respectively, and write $\partial_X\Phi_i$ and $\partial_Y\Phi_i$ for its derivatives with respect to $X$ and $Y$.
The key idea is to make these components active successively, where a component is \emph{active} when its gradient is nonzero.
During gap $i$, $\Phi_i$ is active and all later components have zero gradient.
At the end of the gap, checkpoint $t_i$ activates $\Phi_{i+1}$.

To activate the components in this order, we choose an \emph{activation threshold} $\ell_i\ge0$ for $\Phi_i$ so that
\[
 \Phi_i(X,Y)=0,\qquad \nabla\Phi_i(X,Y)=0
 \quad\text{whenever }X\le\ell_i,\ Y\ge0.
\]
Thus, as long as $Y\ge0$, the component cannot be active until its input $X$ exceeds $\ell_i$.
The final component satisfies $\Phi_{k+1}(Y)=\Phi'_{k+1}(Y)=0$ whenever $Y\le\ell_{k+1}$.
At the beginning of gap $i$, write
\begin{equation*}
 (x_{t_{i-1}}^{(i)},x_{t_{i-1}}^{(i+1)})=(\ell_i+D_i,0),
\end{equation*}
where $D_i>0$ is the input's \emph{excess} above the threshold $\ell_i$ at the start of the gap.
Initially, $\ell_1=0$ and $D_1=1$.
During gap $i$, the construction keeps $x^{(i+1)}\le\ell_{i+1}$ and all further coordinates at zero.
At checkpoint $t_i$, coordinate $i+1$ rises to $\ell_{i+1}+D_{i+1}$, crossing the threshold.
Given $(\ell_i,D_i)$ and the prescribed gap, we construct $\Phi_i$ and choose $\ell_{i+1}$.
The checkpoint update determines $D_{i+1}$, and the pair $(\ell_{i+1},D_{i+1})$ becomes the starting data for the next component.
The smooth convex construction in~\citep{ye2026silver}, also require $\partial_{x^{(i+1)}}\Phi_i=0$ whenever $x^{(i+1)}\ge\ell_{i+1}$, independently of $x^{(i)}$.
The shared coordinate $x^{(i+1)}$ stays above $\ell_{i+1}$ during the next gap, so $\Phi_i$ no longer affects its update.
\Cref{fig:local-transfer} illustrates this construction.

\begin{figure}[!htbp]
\centering
\begingroup
  \definecolor{transferwork}{HTML}{6F3F8F}
  \definecolor{transfercheckpoint}{HTML}{C47A18}
  \definecolor{transfermuted}{HTML}{656565}
  \definecolor{transferzero}{HTML}{EEF3F6}
\begin{minipage}[t]{.49\linewidth}
\centering
$\Phi_i(x^{(i)},x^{(i+1)})$\par\vspace{3pt}
\resizebox{\linewidth}{!}{%
\begin{tikzpicture}[x=1cm,y=.95cm,font=\small,
        line cap=round,line join=round]
      \path[use as bounding box] (-.75,-1.55) rectangle (7.55,4.15);
      \fill[transferzero] (0,1.65) rectangle (6.35,3.65);
      \node at (1.35,2.55) {$\partial_{x^{(i+1)}}\Phi_i=0$};
      \draw[transfermuted!65,densely dashed,line width=.499pt] (0,1.65)--(6.35,1.65);
      \node[left] at (-.05,1.65) {$\ell_{i+1}$};
      \draw[->,black!78,line width=.665pt] (0,0)--(6.55,0)
        node[right] {$x^{(i)}$};
      \draw[->,black!78,line width=.665pt] (0,0)--(0,3.85)
        node[above] {$x^{(i+1)}$};
      \node[below left] at (0,0) {$0$};
      \draw[black!65,line width=.457pt] (1,-.08)--(1,.08);
      \node[below] at (1,-.1) {$\ell_i$};

      \draw[transferwork,line width=1.123pt]
        (6,0)--(5.45,.15) (4.10,.92)--(3.42,1.14);
      \draw[->,transferwork,line width=1.040pt] (5.91,.025)--(5.59,.112);
      \draw[->,transferwork,line width=1.040pt] (4.02,.946)--(3.67,1.059);
      \foreach \u/\v in {6/0,5.45/.15,4.10/.92}
        \fill[transferwork] (\u,\v) circle[radius=1.4pt];
      \node[text=transferwork,rotate=-28,inner sep=0pt] at (4.78,.56) {$\cdots$};

      \draw[->,transfercheckpoint,line width=1.247pt]
        (3.42,1.14)--(2.97,3.12);
      \fill[transfercheckpoint] (3.42,1.14) circle[radius=1.8pt];
      \fill[transfercheckpoint] (2.97,3.12) circle[radius=1.8pt];
      \node[right,text=transfercheckpoint!85!black] at (3.35,2.55) {$b_i$};
      \node[below left] at (3.42,1.14) {$x_{t_i-1}$};
      \node[below left] at (2.97,3.12) {$x_{t_i}$};

      \draw[transferwork!45,dotted,line width=.416pt] (1,-1.15)--(1,-.4) (6,-1.15)--(6,-.12);
      \draw[<->,draw=transferwork!85,line width=.540pt] (1,-1.15)--(6,-1.15)
        node[midway,fill=white,inner sep=1pt,text=transferwork] {$D_i$};
      \draw[transfercheckpoint!48,dotted,line width=.416pt] (2.97,3.12)--(6.4,3.12);
      \draw[<->,draw=transfercheckpoint!90,line width=.540pt] (6.4,1.65)--(6.4,3.12)
        node[midway,right,text=transfercheckpoint!85!black] {$D_{i+1}$};

      \node[anchor=north,inner sep=0pt,fill=white,yshift=-3pt] at (6,0) {$x_{t_{i-1}}$};
    \end{tikzpicture}
}
\end{minipage}\hfill
\begin{minipage}[t]{.49\linewidth}
\centering
$\Phi_{i+1}(x^{(i+1)},x^{(i+2)})$\par\vspace{3pt}
\resizebox{\linewidth}{!}{%
\begin{tikzpicture}[x=1.5cm,y=.91cm,font=\small,
        every node/.style={scale=.904},
        line cap=round,line join=round]
      \path[use as bounding box] (-.5,-1.55) rectangle (4.5,3.8);
      \fill[transferzero] (0,0) rectangle (1.65,3.3);
      \draw[transfermuted!65,densely dashed,line width=.452pt] (1.65,0)--(1.65,3.3);
      \draw[->,black!78,line width=.602pt] (0,0)--(3.9,0)
        node[right] {$x^{(i+1)}$};
      \draw[->,black!78,line width=.602pt] (0,0)--(0,3.5)
        node[above] {$x^{(i+2)}$};
      \node[anchor=east,inner sep=0pt,xshift=-8pt] at (0,0) {$0$};
      \node[below] at (1.65,-.1) {$\ell_{i+1}$};
      \node[align=center,text=black!65] at (.82,2.35)
        {$\Phi_{i+1}=0$};
      \node[align=center,text=black!75] at (2.67,2.35)
        {$\Phi_{i+1}$ activates\\at $t_i$};

      \draw[->,transferwork,line width=1.016pt] (0,0)--(1.14,0);
      \foreach \v in {0,.15,.92}
        \fill[transferwork] (\v,0) circle[radius=1.1pt];
      \draw[->,transfercheckpoint,line width=1.129pt] (1.14,0)--(3.12,0);
      \fill[transfercheckpoint] (1.14,0) circle[radius=1.8pt];
      \fill[transfercheckpoint] (3.12,0) circle[radius=1.8pt];

      \draw[transfercheckpoint!48,dotted,line width=.376pt] (1.65,-1.15)--(1.65,-.4)
        (3.12,-1.15)--(3.12,-.12);
      \draw[<->,draw=transfercheckpoint!90,line width=.489pt] (1.65,-1.15)--(3.12,-1.15)
        node[midway,fill=white,inner sep=1pt,text=transfercheckpoint!85!black] {$D_{i+1}$};

      \node[anchor=north,inner sep=0pt,fill=white,yshift=-6pt] at (0,0) {$x_{t_{i-1}}$};
      \node[anchor=north,inner sep=0pt,fill=white,yshift=-6pt] at (1.14,0) {$x_{t_i-1}$};
      \node[anchor=north,inner sep=0pt,fill=white,yshift=-6pt] at (3.12,0) {$x_{t_i}$};
    \end{tikzpicture}
}
\end{minipage}
\endgroup
\caption{Activation of $\Phi_{i+1}$ at checkpoint $t_i$. Gap steps are purple and the checkpoint step is orange.}
\label{fig:local-transfer}
\end{figure}

For a fixed checkpoint set, a larger final excess $D_{k+1}$ gives a larger lower bound on function suboptimality.
We therefore seek a large \emph{transfer ratio} $D_{i+1}/D_i$ when constructing each local hard function.

In \citep{jung2026stronger}, the two-dimensional local hard functions $\Phi_i$ are constructed as scaled one-sided Huber functions:
\[
 \Phi_i(X,Y)=\frac12H_{2D_i/(2+s_i)}(X-Y-\ell_i).
\]
Here $H_\delta$ is the one-sided Huber function defined in \cref{sec:local-huber}.
With the threshold $\ell_{i+1}=s_iD_i/[2(2+s_i)]$, the transfer ratio $D_{i+1}/D_i=b_i/[2(2+s_i)]$ depends only on $b_i$ and $s_i$.
A suitable choice of $\Phi_{k+1}$ gives
\begin{equation}\label{eq:overview-huber-score}
 \frac{F(x_n)-F(x_*)}{\|x_0-x_*\|^2}
 \ge\frac{1}{4(1+s_{k+1})}\prod_{i=1}^k\left(\frac{D_{i+1}}{D_i}\right)^2
 =\frac{1}{4(1+s_{k+1})}\prod_{i=1}^k\left[\frac{b_i}{2(2+s_i)}\right]^2.
\end{equation}
The expression on the right is the score for this construction.
By maximizing this score over all checkpoint sets, \citep[Section~3]{jung2026stronger} obtain the lower bound $\Omega(n^{-\log_2(1+\sqrt3)})$.\footnote{In general, the argument in~\citep[Section~3 and Appendix~C]{jung2026stronger} suggests a lower bound of order $n^{-\log_2(1+\sqrt{1+c})}$ when the coefficient of $s_i$ in the denominator of the transfer ratio $D_{i+1}/D_i$ is $c$. The factor $b_i/[2(2+s_i)]$ in \eqref{eq:overview-huber-score} corresponds to $c=2$.}

\paragraph{Bending the Local Trajectory \citep{ye2026silver}.}
The local hard functions in~\citep{jung2026stronger} have constant gradients during their respective gaps.
These gradients decrease $X$ and increase $Y$ by equal amounts, accounting for the coefficient $2$ of $s_i$ in the denominator $2(2+s_i)$ in \eqref{eq:overview-huber-score}~\citep[Section~2]{ye2026silver}.
To reduce this loss, our previous work introduces a construction that bends the GD trajectory within each gap~\citep[Section~3.1]{ye2026silver}.
For a checkpoint set with $b_i>2$, the construction gives the score
\[
 \frac{1}{4(1+s_{k+1})}
 \prod_{i=1}^k\left[\frac{b_i-2}{c(a+s_i)}\right]^2.
\]
Here $a$ and $c>1$ are constants determined by the construction, where $c$ can be made arbitrarily close to $1$.
A recursive analysis of this score then yields the lower bound $\Omega(n^{-\psil-o(1)})$~\citep[Section~3.2]{ye2026silver}.

\subsection[Our Construction: Handling Quadratic Contraction]{Our Construction: Handling Quadratic Contraction}\label{sec:sc-ingredients}

Using the notation above, we replace the final component in \eqref{eq:overview-convex-function} by a quadratic term and scale the remaining sum by $1-1/\kappa$:
\begin{equation}\label{eq:global-function}
 F(x):=\frac1{2\kappa}\|x\|^2+
       \frac{1-1/\kappa}{2}\sum_{i=1}^k\Phi_i(x^{(i)},x^{(i+1)}),
 \qquad x_0=e_1,\quad x_*=0.
\end{equation}
The resulting $F$ is $1$-smooth and $1/\kappa$-strongly convex.
As in the smooth convex setting, we seek a large final excess $D_{k+1}=\prod_{i=1}^k(D_{i+1}/D_i)$ and use it to bound the final error.

\paragraph{Quadratic Contraction.}
However, preserving the excess is harder in the strongly convex setting, since the added quadratic term can reduce it during a gap. For $s_i<\kappa$, the quadratic term alone scales the input coordinate $x^{(i)}$ by $\chi_i$, while the threshold $\ell_i$ stays fixed. The excess therefore changes from $D_i$ to $D_i\eta_i$, where
\begin{equation}\label{eq:overview-contraction}
 \chi_i:=\prod_{t_{i-1}<t<t_i}\left(1-\frac{h_t}{\kappa}\right),\qquad
 \chi_i(\ell_i+D_i)-\ell_i
 =D_i\Bigl[\underbrace{\chi_i-(1-\chi_i)\frac{\ell_i}{D_i}}_{:=\eta_i}\Bigr].
\end{equation}
Since $\chi_i\ge1-s_i/\kappa$, the fraction of the excess lost to contraction satisfies $1-\eta_i\le(s_i/\kappa)(1+\ell_i/D_i)$. We therefore need to keep $s_i$ small relative to $\kappa$ and control the \emph{relative threshold} $\ell_i/D_i$.

\paragraph{Choosing Checkpoints within Blocks.}
As in our smooth convex construction, we use \emph{bending components}, adapted to the strongly convex setting (constructed in \cref{sec:bendin-local}). The adapted construction depends on a bending parameter $\beta$.\footnote{Our smooth convex construction~\citep[Section~3.1]{ye2026silver} prescribes the gradients of each local hard function along its gap trajectory by scaling vectors on a circular arc centered at $(0,\gamma)$ for a parameter $\gamma>0$.
We adapt this construction to the strongly convex setting with the center $(\beta,\beta)$ (see \cref{app:local-bending} for details).} For each bending component, \cref{lem:bending-local} gives
\[
 \frac{D_{i+1}}{D_i}\ge\frac{b_i-8}{c_\beta(a_\beta+s_i+8)}\,\eta_i\left(1-\frac{s_i}{\kappa}\right).
\]
The constants $a_\beta\ge8$ and $1<c_\beta<2$ depend only on $\beta$, with $c_\beta=1+O(\beta)$ and $\log a_\beta=O(1/\beta)$.
The first factor has the same form compared with the smooth convex score $(b_i-2)/(c(a+s_i))$ in \cite{ye2026silver}, with $c_\beta$ in place of $c$ and different additive constants. Taking $\beta$ smaller brings $c_\beta$ closer to one, which improves the exponent of $\kappa$ in the iteration lower bound, but increases $a_\beta$. We choose $\beta$ at the end of the proof to balance these effects.

The factor $\eta_i(1-s_i/\kappa)$ in this bound accounts for quadratic contraction.
Notice that $\eta_i$ depends on the relative threshold $\ell_i/D_i$ passed from the preceding component $\Phi_{i-1}$.
Keeping $\eta_i$ in the score would make checkpoint selection depend on the local functions constructed in earlier gaps.
To simplify checkpoint selection, we omit the factor $\eta_i(1-s_i/\kappa)$ and use the score
\[
 \prod_{i=1}^k\frac{b_i-8}{c_\beta(a_\beta+s_i+8)},\qquad b_i>8.
\]
However, maximizing this score over the whole schedule can leave gaps with very large sums of stepsizes $s_i$. The resulting quadratic contraction can become negative, destroying the lower bound. We therefore divide the schedule into blocks of at most $m$ steps and maximize the score within each block, as formalized in \cref{sec:fixed-c-block}.

As in our smooth convex analysis~\citep[Section~3.2]{ye2026silver}, maximality of the chosen checkpoint set imposes constraints on the checkpoint stepsizes $b_i$ and the sums of gap stepsizes $s_i$.
Specifically, the bounds on the gap sums $s_i$ depend on $a_\beta m^{\psil+O(\beta)}$ and the sum of unselected stepsizes carried over from earlier blocks (see the first inequality in \cref{lem:fixed-c-block}~(ii)).
Recall that we need $s_i$ small relative to $\kappa$ to limit the loss of excess from quadratic contraction.
Therefore, $m$ should be chosen such that $a_\beta m^{\psil+O(\beta)}$ is a sufficiently small fraction of $\kappa$.
We then use these constraints to account for the omitted contraction factors after checkpoint selection (see \cref{lem:mixed-block}).

\paragraph{Repairing Accumulated Unselected Steps.}
If the chosen maximizing checkpoint set is empty, we select no checkpoint in that block.
When this happens in several consecutive blocks, their stepsizes all contribute to the same $s_i$.
Quadratic contraction can then exhaust the excess $D_i$ before the next checkpoint.
We therefore fix a \emph{repair limit}, a small constant multiple of $\kappa$.
At the end of such a block, once the sum of stepsizes since the latest checkpoint reaches this limit, we select one or two additional checkpoints, called \emph{repair checkpoints}.
At these checkpoints, we use \emph{Huber repair components}, built in \cref{sec:local-huber} from the one-sided Huber functions of~\citep{jung2026stronger}.
To control the quadratic contraction discussed above, we choose the repair checkpoints to bound the relative threshold after the repair and keep the remaining stepsize sum below the repair limit.
\Cref{sec:checkpoint-repair} describes the detailed selection rule, and \cref{lem:repair} bounds the product of transfer ratios within each repair from below.
\Cref{lem:repair-cost} then shows that each stepsize index is examined in at most two repairs.
Since there are at most $1+n/m$ blocks and at most one repair per block, these bounds imply that the product of transfer ratios over all repairs is at least $(m+1)^{-O(1+n/m)}$, as shown in \cref{sec:all-horizons}.

\paragraph{Bridging Back to Bending.}
When constructing the global hard function, we sometimes need to return to bending components after a repair.
However, a bending component $\Phi_{i+1}$ requires that the preceding component $\Phi_i$ contribute nothing to the update of $x^{(i+1)}$. This means $\partial_{x^{(i+1)}}\Phi_i=0$ during gap $i+1$ and just before its checkpoint update.
A Huber repair component $\Phi_i$ guarantees only $\partial_{x^{(i+1)}}\Phi_i\le0$, so its contribution to this update may be positive.
We therefore use a \emph{bridging component}, or bridge, for $\Phi_{i+1}$, constructed in \cref{sec:local-bridge}.
The bridge works under the weaker condition $\partial_{x^{(i+1)}}\Phi_i\le0$.
It also satisfies $\partial_{x^{(i+2)}}\Phi_{i+1}=0$ whenever $x^{(i+2)}\ge\ell_{i+2}$, so a bending component $\Phi_{i+2}$ can follow it.
Since a repair may precede any nonempty maximizing set, we use a bridge at its first checkpoint and bending components at the remaining checkpoints.
\Cref{lem:mixed-block} controls the loss from quadratic contraction and shows that the bridge and bending components at every nonempty maximizing checkpoint set retain at least $1/8$ of the incoming excess.

\begin{figure}[!t]
\centering
\begingroup
\definecolor{archblue}{HTML}{345D9D}
\definecolor{archorange}{HTML}{BF620C}
\definecolor{archink}{HTML}{242731}
\definecolor{archgray}{HTML}{717783}
\definecolor{archlight}{HTML}{CAD0D7}
\newcommand{\archcheckpoint}[3]{%
  \draw[#3,line width=.8pt] (#1,{#2-.13})--(#1,{#2+.17});%
  \filldraw[fill=#3,draw=white,line width=.35pt] (#1,#2) circle (2.1pt);%
}
\begin{tikzpicture}[x=1cm,y=1cm,>=Latex,font=\footnotesize,
  text=archink,inner sep=1pt,
  component/.style={draw,rounded corners=1.5pt,minimum width=16mm,
                    minimum height=6mm,align=center},
  case arrow/.style={->,archgray,line width=.65pt}]
\path[use as bounding box] (0,0) rectangle (16.2,9.65);

\node[anchor=west] at (0,9.48) {(a)\quad Scan the blocks from left to right};
\foreach \i/\name in {0/{block 1},1/{$\cdots$},2/{block $j-1$},
                         3/{block $j$},4/{block $j+1$},5/{$\cdots$},6/{last block}}{
  \pgfmathsetmacro{\leftedge}{.75+2.1*\i}
  \pgfmathsetmacro{\rightedge}{\leftedge+2.1}
  \filldraw[fill=archlight!12,draw=archlight,line width=.45pt]
    (\leftedge,8.55) rectangle (\rightedge,9.05);
  \node at ({(\leftedge+\rightedge)/2},8.84) {\name};
  \foreach \t in {0,...,7}{
    \fill[archgray] ({\leftedge+.18+.248*\t},8.65) circle (.65pt);
  }
}
\draw[archgray,line width=.8pt] (7.05,8.55) rectangle (9.15,9.05);
\draw[case arrow] (7.05,9.20)--(9.15,9.20);
\draw[archgray,line width=.4pt] (.75,8.44)--(.75,8.34)--(2.85,8.34)--(2.85,8.44);
\node[text=archgray,font=\scriptsize] at (1.8,8.13) {$m$ steps};
\node[text=archgray,font=\scriptsize] at (8.1,8.33) {current block};

\draw[archgray,line width=.65pt] (8.1,8.14)--(8.1,7.53);
\draw[case arrow] (8.1,7.53)--(4,7.53)--(4,7.40);
\draw[case arrow] (8.1,7.53)--(12.2,7.53)--(12.2,7.40);
\node[anchor=west] at (0,7.85) {(b)\quad Choose a maximizing checkpoint set in block $j$};
\node[text=archblue] at (4,7.18) {Chosen set is nonempty};
\node[text=archorange] at (12.2,7.18) {Chosen set is empty};
\draw[archlight,dashed,line width=.4pt] (8.1,6.95)--(8.1,1.10);

\foreach \shift in {0,8.2}{
  \begin{scope}[xshift=\shift cm]
    \fill[archlight!15] (2.15,6.00) rectangle (7.2,6.83);
    \node[text=archgray,font=\scriptsize] at (4.675,6.62) {block $j$};
    \node[text=archgray,font=\scriptsize,align=center] at (1.475,6.62)
      {incoming\\tail};
    \foreach \x in {2.15,7.2}{
      \draw[archlight,dashed,line width=.4pt] (\x,5.98)--(\x,6.88);
    }
    \draw[archlight,line width=.5pt] (.8,6.20)--(7.2,6.20);
    \foreach \i in {0,...,28}{\fill[archlight] ({.8+.225*\i},6.20) circle (.8pt);}
    \archcheckpoint{.8}{6.20}{archink}
    \node[anchor=north,font=\scriptsize,align=center] at (.8,5.96) {latest chosen\\checkpoint};
  \end{scope}
}
\foreach \x in {2.375,4.175,6.875}{\archcheckpoint{\x}{6.20}{archblue}}
\node[text=archblue,align=center] at (4,5.58) {Keep the selected checkpoints};
\node[text=archorange,align=center] at (12.8,5.58) {Add repair checkpoints if the stepsize sum\\reaches the repair limit};
\draw[->,archblue,line width=.65pt] (4,5.35)--(4,4.75);
\draw[->,archorange,line width=.65pt] (12.8,5.15)--(12.8,4.75);

\node[anchor=west] at (0,4.35) {(c)\quad Resulting checkpoints and local hard functions};
\foreach \shift/\shade in {0/archblue,8.2/archorange}{
  \begin{scope}[xshift=\shift cm]
    \fill[\shade!6] (2.15,3.48) rectangle (7.2,3.98);
    \draw[archlight,line width=.5pt] (.8,3.70)--(7.2,3.70);
    \foreach \i in {0,...,28}{\fill[archlight] ({.8+.225*\i},3.70) circle (.8pt);}
    \archcheckpoint{.8}{3.70}{archink}
    \foreach \x in {2.15,7.2}{
      \draw[archlight,dashed,line width=.4pt] (\x,3.45)--(\x,4.03);
    }
  \end{scope}
}

\foreach \x/\name/\id in {2.375/Bridging/bridge,4.175/Bending/bendone,6.875/Bending/bendtwo}{
  \archcheckpoint{\x}{3.70}{archblue}
  \draw[archblue!45,densely dotted,line width=.5pt] (\x,3.48)--(\x,2.65);
  \node[component,draw=archblue,fill=archblue!5,text=archblue] (\id) at (\x,2.35) {\name};
}
\node[anchor=east,text=archgray] (previousbend) at ($(bridge.west)+(-.38,0)$) {$\cdots$};
\draw[case arrow] (previousbend.east)--(bridge.west);
\draw[case arrow] (bridge.east)--(bendone.west);
\draw[case arrow] (bendone.east)--(bendtwo.west);
\node[text=archblue,font=\scriptsize,align=center] at (2.375,1.71) {first selected\\checkpoint};
\node[text=archblue,font=\scriptsize,align=center] at (5.525,1.71) {remaining selected\\checkpoints};

\foreach \x/\id in {9.675/repairone,13.725/repairtwo}{
  \archcheckpoint{\x}{3.70}{archorange}
  \draw[archorange!45,densely dotted,line width=.5pt] (\x,3.48)--(\x,2.65);
  \node[component,draw=archorange,fill=archorange!5,text=archorange] (\id) at (\x,2.35) {Repairing};
}
\node[anchor=east,text=archgray] (previousrepair) at ($(repairone.west)+(-.38,0)$) {$\cdots$};
\draw[case arrow] (previousrepair.east)--(repairone.west);
\draw[case arrow] (repairone.east)--(repairtwo.west);
\node[text=archorange,font=\scriptsize] at (11.7,1.71) {added checkpoints};
\end{tikzpicture}
\endgroup
\caption[Checkpoint selection and local hard functions in one block.]{Selecting checkpoints and local hard functions. If the maximizing checkpoint set of a block is nonempty, we use a bridge at its first checkpoint and bending components at the remaining checkpoints. Otherwise, at the end of the block, we add one or two repair checkpoints with Huber repair components when the sum of stepsizes since the latest checkpoint reaches the repair limit. Repair checkpoints may lie in the incoming tail between the latest chosen checkpoint and the current block.}
\label{fig:architecture}
\end{figure}

\paragraph{From the Construction to the Lower Bound.}
In \cref{prop:fixed}, we combine the bound for nonempty maximizing sets with the bound for all repairs to obtain
\[
 D_{k+1}\ge\exp\!\left[-O\!\left((1+n/m)\log(m+1)\right)\right].
\]
Our checkpoint selection and repair rule ensure that the final coordinate $x_n^{(k+1)}$ remains at least a constant multiple of $D_{k+1}$, yielding a lower bound on squared distance to the minimizer.
\Cref{lem:repetition} then gives the following lower bounds on squared distance to the minimizer and function suboptimality:
\[
 \min\{\Derr_{n,\kappa}(H),\Eerr_{n,\kappa}(H)\}
 \ge\exp\!\left[-O\!\left(\frac{n\log(m+1)}{m}\right)\right].
\]

To strengthen the bound, we choose $m$ as large as possible while keeping $a_\beta m^{\psil+O(\beta)}$ a sufficiently small fraction of $\kappa$, as required in the earlier paragraph \emph{Choosing Checkpoints within Blocks}.
We therefore take $m$ of order $(\kappa/a_\beta)^{1/\psil-O(\beta)}$; see \eqref{eq:fixed-block-length} for the precise choice.
Using $\log a_\beta=O(1/\beta)$, we obtain
\[
 \log m\ge\frac{\log\kappa}{\psil}-O\!\left(\beta\log\kappa+\frac1\beta\right).
\]
Finally, we balance the two terms by choosing $\beta=(\log\kappa)^{-1/2}$ for sufficiently large $\kappa$. This gives $m\ge\kappa^{1/\psil}e^{-O(\sqrt{\log\kappa})}$ and proves \cref{thm:main}.

\Cref{fig:architecture} summarizes how we select checkpoints and local hard functions.

\paragraph{Organization.}
The rest of this paper is organized as follows.
\Cref{sec:local} constructs the three types of local hard functions and bounds their transfer ratios $D_{i+1}/D_i$ and relative thresholds $\ell_{i+1}/D_{i+1}$.
\Cref{sec:checkpoints} provides the checkpoint selection and repair rule and estimates for the selected checkpoints.
\Cref{sec:main-proof} constructs a hard function for each prescribed schedule and proves the main theorem.

\section{Local Hard Functions: Bending, Repairing, and Bridging}
\label{sec:local}

In this section, we present three types of local hard functions $\Phi_i(x^{(i)},x^{(i+1)})$ on adjacent coordinates used in \eqref{eq:global-function}. We first recall the local transfer setup from \cref{sec:existing-techniques} and fix the notation used below.

\paragraph{Notation and setup.}
\leavevmode
For the current local hard function $\Phi_i$, write $(X,Y)=(x^{(i)},x^{(i+1)})$.
Fix an input threshold $\ell_i\ge0$, an incoming excess $D_i>0$ above $\ell_i$, and gap stepsizes $0\le h_j<\kappa$ for $1\le j\le m_{\mathrm{gap}}$. Here the stepsize index $j$ starts at one in the current gap.
Starting from $(X_0,Y_0)=(\ell_i+D_i,0)$, let $(X_j,Y_j)$ be the state after $j$ updates.
The gap is followed by a checkpoint update with stepsize $h_{m_{\mathrm{gap}}+1}=b_i>0$, which evaluates the gradient at $(X_{m_{\mathrm{gap}}},Y_{m_{\mathrm{gap}}})$, the state just before this update.
We construct the current component $\Phi_i$ and choose an outgoing threshold $\ell_{i+1}\ge0$ so that
\begin{equation}\label{eq:local-transfer-states}
 \begin{gathered}
 X_j\ge\ell_i,\qquad 0\le Y_j\le\ell_{i+1}
       \quad(0\le j\le m_{\mathrm{gap}}),\\
 Y_{m_{\mathrm{gap}}+1}=\ell_{i+1}+D_{i+1},\qquad D_{i+1}>0.
 \end{gathered}
\end{equation}
For $i<k$, the pair $(\ell_{i+1},D_{i+1})$ becomes the incoming data for the next local hard function $\Phi_{i+1}$.

\paragraph{Local updates with the preceding component.}
\leavevmode
With later components inactive, \eqref{eq:global-function} gives the local updates, for $1\le j\le m_{\mathrm{gap}}+1$:

\begin{equation}\label{eq:local-dynamics}
 \binom{X_j}{Y_j}
 =\left(1-\frac{h_j}{\kappa}\right)\binom{X_{j-1}}{Y_{j-1}}
 -\frac{(1-1/\kappa)h_j}{2}\nabla\Phi_i(X_{j-1},Y_{j-1})
 +h_j\binom{f_j}{0}.
\end{equation}

For $i>1$, the term $h_jf_j$ is the contribution of the preceding component $\Phi_{i-1}$ to the update of the shared coordinate $x^{(i)}$, where $f_j=-(1-1/\kappa)\partial_{x^{(i)}}\Phi_{i-1}/2$. Here $\partial_{x^{(i)}}\Phi_{i-1}$ differentiates $\Phi_{i-1}(x^{(i-1)},x^{(i)})$ in the shared coordinate $x^{(i)}$ and is evaluated before step $j$. Each local hard function below satisfies $\partial_Y\Phi_i\le0$, so $f_j\ge0$. For $i=1$, there is no preceding component and $f_j=0$.

To bound the loss from quadratic contraction, recall $\chi_i$ and $\eta_i$ from \eqref{eq:overview-contraction}. If $s_i<\kappa$, then
\begin{equation}\label{eq:local-gap-estimates}
 1-\frac{s_i}{\kappa}\le\chi_i\le1,\qquad
 \eta_i\ge1-\frac{s_i}{\kappa}(1+\ell_i/D_i).
\end{equation}
For an empty gap, $s_i=0$ and $\chi_i=\eta_i=1$.

With these preparations in hand, we now give the constructions and transfer guarantees for the three types of local hard functions introduced in \cref{sec:sc-ingredients}.
For each $\Phi_i$, we bound the transfer ratio $D_{i+1}/D_i$ from below to limit the loss of excess.
A larger final excess $D_{k+1}$ gives a stronger lower bound.
We also bound the relative threshold $\ell_{i+1}/D_{i+1}$ from above to limit contraction loss in the next gap.
\subsection{Bending}
\label{sec:bendin-local}

For our basic local hard function, we adapt the bending construction from our previous work~\citep{ye2026silver} to account for quadratic contraction.
Specifically, we choose a parameter $0<\beta\le1/4$.
Along the constructed gap trajectory, the scaled gradient turns toward the negative $Y$-direction along a circular arc centered at $(\beta,\beta)$.
The construction yields constants $a_\beta\ge8$ and $1<c_\beta<2$, depending only on $\beta$, with

\begin{equation}\label{eq:bending-parameter-growth}
 c_\beta=1+O(\beta),\qquad
 \log a_\beta=O(1/\beta),
\end{equation}
uniformly when $0<\beta\le1/4$.
Their explicit choices are given in \eqref{eq:bending-c-beta} and \eqref{eq:bending-a-beta} in \cref{app:local-bending}.

\begin{lemma}[Bending]\label{lem:bending-local}
Let $\kappa\ge4$, $0<\beta\le1/4$, $s_i/\kappa\le1/16$, and $\eta_i>0$. There exist a nonnegative convex $1$-smooth local hard function $\Phi_i$ and a threshold $\ell_{i+1}\ge0$, independent of the checkpoint stepsize $b_i$, such that $\partial_Y\Phi_i\le0$ everywhere and
\begin{equation*}
 \Phi_i(X,Y)=0\quad(X\le\ell_i, Y\ge0),\qquad
 \partial_Y\Phi_i(X,Y)=0\quad(X\in\R,Y\ge\ell_{i+1}).
\end{equation*}
For every $b_i>8$, the iterations \eqref{eq:local-dynamics} with $f_j=0$ satisfy \eqref{eq:local-transfer-states}, and
\begin{equation}\label{eq:bending-local-bound}
 \frac{D_{i+1}}{D_i}\ge\frac{b_i-8}{c_\beta(a_\beta+s_i+8)}\,\eta_i(1-s_i/\kappa),\qquad \frac{\ell_{i+1}}{D_{i+1}}\le4\frac{s_i+2}{b_i-8}.
\end{equation}
\end{lemma}

For these explicit choices, taking $\beta$ smaller brings $c_\beta$ closer to one but increases $a_\beta$.
The construction and proof are given in \cref{app:local-bending}.
As in our smooth convex construction~\citep{ye2026silver}, the bending component $\Phi_i$ satisfies $\partial_Y\Phi_i(X,Y)=0$ for every $X$ whenever $Y\ge\ell_{i+1}$.
For $i<k$, $Y=x^{(i+1)}$ is the input coordinate of the next component $\Phi_{i+1}$.
Thus, $\Phi_i$ does not affect the update of this coordinate in this region.

\subsection{Repairing}
\label{sec:local-huber}

As discussed in \cref{sec:sc-ingredients}, we select checkpoint sets by maximizing a product based on the bending transfer bound.
This selection can leave many consecutive steps unselected, so the sum of gap stepsizes $s_i$ can become large.
If $s_i$ becomes large relative to $\kappa$, quadratic contraction can substantially reduce the incoming excess $D_i$.
We therefore split such gaps with suitably chosen additional checkpoints and use \textit{Huber repair components} to bound the relative threshold $\ell_{i+1}/D_{i+1}$ after the last added checkpoint $t_i$.
The construction uses the following one-sided Huber function of \citep{jung2026stronger}:
\[
 H_\delta(t)=\begin{cases}
 0,&t\le0,\\ t^2/2,&0\le t\le \delta,\\ \delta t-\delta^2/2,&t\ge \delta.
 \end{cases}
\]
We choose the Huber scale and local hard function as follows:
\begin{equation}\label{eq:huber-local-choice}
 \delta_{\mathrm H}=\frac{D_i\eta_i}{2+(\kappa-1)(1-\chi_i)},\qquad
 \Phi_i(X,Y)=\frac12H_{2\delta_{\mathrm H}}(X-Y-\ell_i).
\end{equation}
Set $\ell_{i+1}=(\kappa-1)(1-\chi_i)\delta_{\mathrm H}/2$.

\begin{lemma}[Repairing]\label{lem:huber-local}
Let $\kappa>1$ and $\eta_i>0$.
The function $\Phi_i$ defined in \eqref{eq:huber-local-choice} is nonnegative, convex, and $1$-smooth, with $\partial_Y\Phi_i\le0$ everywhere and $\Phi_i(X,Y)=0$ whenever $X\le\ell_i$, $Y\ge0$.

For every $b_i>0$ and every sequence $f_j\ge0$, the iterations \eqref{eq:local-dynamics} satisfy \eqref{eq:local-transfer-states}, and
\begin{equation}\label{eq:huber-local-output}
 \begin{aligned}
 \frac{D_{i+1}}{D_i}&\ge\frac{b_i\eta_i}{2(s_i+2)}
 \left(1-\frac1\kappa\right)\left(1-\frac{s_i}{\kappa}\right),\\
 \frac{\ell_{i+1}}{D_{i+1}}&=\frac{\kappa(1-\chi_i)}{b_i\chi_i}.
 \end{aligned}
\end{equation}
The outgoing excess $D_{i+1}$ is independent of the nonnegative contributions $f_j$.
\end{lemma}

For this Huber repair component, $\ell_{i+1}$ is the value of $Y$ just before the checkpoint update, while $D_{i+1}$ is the net increase in $Y$ produced by that update.
If $s_i<\kappa$, the relative-threshold identity in \eqref{eq:huber-local-output} gives
\begin{equation}\label{eq:huber-ratio-mass}
 \frac{s_i}{b_i}\le\frac{\ell_{i+1}}{D_{i+1}}
 \le\frac{s_i/b_i}{1-s_i/\kappa}.
\end{equation}
When $s_i$ is small relative to $\kappa$, \eqref{eq:huber-ratio-mass} shows that a checkpoint with $b_i$ comparable to $s_i$ gives a bounded relative threshold $\ell_{i+1}/D_{i+1}$.
The proof of the local guarantees is given in \cref{app:local-huber}.
One issue of the Huber repair component is that $\partial_Y\Phi_i(X,Y)=0$ only when $Y\ge X-\ell_i$. Thus, even after the input coordinate $x^{(i+1)}$ of the next component $\Phi_{i+1}$ crosses $\ell_{i+1}$, $\Phi_i$ may still affect its update.
The construction in the next section resolves this.

\subsection{Bridging}
\label{sec:local-bridge}

After a Huber repair, the preceding component $\Phi_{i-1}$ may still contribute $h_jf_j\ge0$ to the update of $X=x^{(i)}$, whereas \cref{lem:bending-local} requires $f_j=0$.
We therefore insert a \emph{bridging component} $\Phi_i$ between the repair component $\Phi_{i-1}$ and the next bending component $\Phi_{i+1}$.
We construct $\Phi_i$ so that its transfer guarantees hold for every sequence $f_j\ge0$ and $\partial_Y\Phi_i(X,Y)=0$ for every $X$ whenever $Y\ge\ell_{i+1}$.
The latter property ensures that $\Phi_i$ does not affect the update of $x^{(i+1)}$ during the next bending gap or its checkpoint update.

Using $\delta_{\mathrm H}$ from \eqref{eq:huber-local-choice}, we define the bridging component by
\begin{equation*}
 \ell_{i+1}=\frac{D_i\eta_i}{2},\qquad K_\triangle=\conv\{0,(\delta_{\mathrm H},-\delta_{\mathrm H}),(\delta_{\mathrm H},0)\},
\end{equation*}
and
\[
 \Phi_i(X,Y)=\max_{g\in K_\triangle}
 \left\{\big\langle g,(X-\ell_i-\ell_{i+1},Y-\ell_{i+1})\big\rangle-\frac12\|g\|^2\right\}.
\]

By the projection identity in \cref{fact:local-moreau}, the horizontal edge of $K_\triangle$ ensures that $\partial_Y\Phi_i(X,Y)=0$ whenever $Y\ge\ell_{i+1}$.

\begin{lemma}[Bridging]\label{lem:bridge}
Let $\kappa\ge4$, $s_i/\kappa\le1/16$, and $\eta_i>0$.
The function $\Phi_i$ defined above is nonnegative, convex, and $1$-smooth, with $\partial_Y\Phi_i\le0$ everywhere and $\Phi_i(X,Y)=0$ whenever $X\le\ell_i$, $Y\ge0$. Moreover, $\partial_Y\Phi_i(X,Y)=0$ for all $X\in\R$, $Y\ge\ell_{i+1}$.

For every $b_i>4$ and every sequence $f_j\ge0$, the iterations \eqref{eq:local-dynamics} satisfy \eqref{eq:local-transfer-states}, and
\begin{equation*}
 \frac{D_{i+1}}{D_i}\ge\frac14\eta_i\frac{b_i-4}{s_i+2},\qquad \frac{\ell_{i+1}}{D_{i+1}}\le2\frac{s_i+2}{b_i-4}.
\end{equation*}
\end{lemma}

With the same $\ell_i,D_i$, stepsizes, and sequence $f_j\ge0$, the bridge and the Huber repair component of \cref{sec:local-huber} follow the same trajectory through the checkpoint update.
The proof is provided in \cref{app:local-bridge}.

As a summarization, \cref{fig:local-gradient-sets} illustrates the gradients and GD trajectories of the three constructions.
\Cref{tab:local-guarantees} summarizes their guarantees and the conditions for combining the components.

\clearpage

\begin{figure}[H]
\centering
\resizebox{\linewidth}{!}{\input{figures/local-gradient-sets}}
\stepcounter{footnote}
\caption[Scaled gradients of $\Phi_i$ and schematic GD trajectories]{Schematic scaled gradients of $\Phi_i$ (top) and GD trajectories $(X_j,Y_j)$ with the quadratic term (bottom). Gray dashed arrows show the checkpoint updates from the same starting points without this term. Blue arrows show the contraction toward the origin.}
\label{fig:local-gradient-sets}
\end{figure}

\begin{table}[htbp]
\centering
\normalsize
\setlength{\tabcolsep}{6pt}
\renewcommand{\arraystretch}{1.5}
\begin{tabular*}{\linewidth}{@{\extracolsep{\fill}}lcccc@{}}
\toprule
 & \shortstack{Transfer ratio\\$D_{i+1}/D_i\ge$}
 & \shortstack{Relative threshold\\$\ell_{i+1}/D_{i+1}\le$}
 & \shortstack{$\partial_Y\Phi_i(X,Y)=0$\\$\forall X\in\R,\ Y\ge\ell_{i+1}$}
 & \shortstack{Requirement\\on $f_j$} \\
\midrule
\textbf{Bending}
 & $\dfrac{b_i-8}{c_\beta(a_\beta+s_i+8)}\eta_i\left(1-\dfrac{s_i}{\kappa}\right)$
 & $\dfrac{4(s_i+2)}{b_i-8}$
 & Yes & $f_j=0$ \\
\addlinespace[7pt]
\textbf{Repairing}
 & $\dfrac{b_i\eta_i}{2(s_i+2)}\left(1-\dfrac1\kappa\right)\left(1-\dfrac{s_i}{\kappa}\right)$
 & $\dfrac{s_i/b_i}{1-s_i/\kappa}$
 & No & $f_j\ge0$ \\
\addlinespace[7pt]
\textbf{Bridging}
 & $\dfrac{\eta_i(b_i-4)}{4(s_i+2)}$
 & $\dfrac{2(s_i+2)}{b_i-4}$
 & Yes & $f_j\ge0$ \\
\bottomrule
\end{tabular*}
\caption{Local guarantees from \cref{lem:bending-local,lem:huber-local,lem:bridge}. The repair threshold bound assumes $s_i<\kappa$.}
\label{tab:local-guarantees}
\end{table}

\clearpage

\section{Selecting Checkpoints for Every Schedule}
\label{sec:checkpoints}

In this section, we describe the checkpoint selection rule (illustrated in \cref{fig:architecture}).
As explained in \cref{sec:sc-ingredients}, selecting checkpoints solely by maximizing the bending score, as in the smooth convex setting, can leave gaps in which quadratic contraction exhausts the excess $D_i$.
We therefore divide the schedule into blocks of $m$ consecutive steps and maximize this score within each block sequentially.
If the sum of unselected stepsizes since the latest checkpoint becomes too large, we add repair checkpoints.

The rest of this section is organized as follows.
\Cref{sec:fixed-c-block} defines the score within each block, states the rule for choosing checkpoints by maximizing this score, and gives estimates for the selected checkpoints and gaps.
\Cref{sec:checkpoint-repair} states the rule for adding repair checkpoints and gives its guarantees.

\subsection{Choosing Checkpoints within a Block}\label{sec:fixed-c-block}

Fix a positive integer $m$. Divide the schedule $H=(h_1,\ldots,h_n)$ into consecutive blocks of $m$ steps, with possibly fewer steps in the last block.\footnote{Blocks are fixed before checkpoint selection. We defer the choice of $m$ to \cref{sec:mixed-block}; see \eqref{eq:fixed-block-length}. By contrast, a gap is determined by successive checkpoints and may cross block boundaries.}
We examine the blocks sequentially.
Write $\tilde{H}=(\tilde{h}_1,\ldots,\tilde{h}_m)$ for the current block, padding it with zeros if necessary.
The \emph{incoming tail} consists of all steps after the latest chosen checkpoint and before the current block begins, or all earlier steps if no checkpoint has been chosen.
Write $w\ge0$ for the sum of stepsizes in the incoming tail; $w=0$ at the start of the first block.
For checkpoints selected in the current block, we number the corresponding local components from $1$ in checkpoint order.
The first gap starts immediately after the preceding checkpoint (or at initialization if there is none) and includes the incoming tail.
The incoming pair $(\ell_1,D_1)$ specifies the first component's activation threshold and the excess above it at the start of this gap.

For the bending components, we take $a=a_\beta$ and $c=c_\beta$ as in \cref{sec:bendin-local}.
For general $a\ge8$ and $1\le c\le2$, define
\begin{equation}\label{eq:checkpoint-exponents}
 p=p(c):=\log_2(1+\sqrt{1+c}).
\end{equation}
An \emph{admissible checkpoint set} $T=\{t_1<\cdots<t_k\}\subseteq\{1,\ldots,m\}$ selects only steps with $b_i:=\tilde{h}_{t_i}>8$. For a nonempty checkpoint set, put
\begin{equation}
     s_1:=w+\sum_{j=1}^{t_1-1}\tilde{h}_j,\qquad
 s_i:=\sum_{j=t_{i-1}+1}^{t_i-1}\tilde{h}_j\quad(2\le i\le k),\qquad s_{\rm tail}:=\sum_{j=t_k+1}^m\tilde{h}_j.
 \label{eq:carried-gap-definitions}
\end{equation}
The lower bound on $D_{i+1}/D_i$ in \cref{lem:bending-local} is $(b_i-8)/(c(a+s_i+8))$ times the contraction factor $\eta_i(1-s_i/\kappa)$.
As explained in \cref{sec:sc-ingredients}, we omit the contraction factor when choosing checkpoints and account for it in \cref{sec:mixed-block}.
Following our smooth convex analysis~\citep[Section~3.2]{ye2026silver}, we define the score as the product of the remaining factors:
\begin{equation}\label{eq:carried-score}
 P_w(T;\tilde{H}):=\prod_{i=1}^k\frac{b_i-8}{c(a+s_i+8)},\qquad
 P_w(\varnothing;\tilde{H}):=1.
\end{equation}
For fixed $\tilde{H}$ and $w$, we maximize this score over the admissible checkpoint sets $T$. If a nonempty set attains the maximum, we choose one. Otherwise we choose $T=\varnothing$, selecting no checkpoint in this block. There are finitely many sets, so a maximizer exists and its score is at least one.

\begin{samepage}
The next lemma provides properties of a maximizing set that we use to bound the product of transfer ratios in \cref{lem:mixed-block}.
\begin{lemma}[Bounds for a maximizing checkpoint set]
\label{lem:fixed-c-block}
In the setting above, there is an absolute constant $C_b$ with the following properties.
Parts~(ii)--(iv) hold for every nonempty maximizing set $T=\{t_1<\cdots<t_k\}$.
\begin{enumerate}[label=(\roman*),leftmargin=*]
\item If the empty set maximizes \eqref{eq:carried-score}, then $w+\sum_{j=1}^m\tilde{h}_j\le C_b(w+a m^p)$.
\item $c\sum_{i=1}^k(a+s_i+8)\le C_b(w+a m^p)$ and $s_{\rm tail}\le C_b a m^p$.
\item For every $1\le i<k$,
\begin{equation*}
 \frac{1}{b_i-8}\le\frac{2}{a+s_i+8}+\frac{2}{a+s_{i+1}+8}.
\end{equation*}
\item The last checkpoint satisfies $b_k-8\ge c(a+s_k+8)$.
\end{enumerate}
\end{lemma}
\end{samepage}

The full proof is provided in \cref{app:sequence}.

\subsection{Repairing Accumulated Unselected Steps}
\label{sec:checkpoint-repair}

If the chosen checkpoint set in a block is empty, its steps extend the \emph{unselected tail}.\footnote{The unselected tail consists of the examined steps after the latest chosen checkpoint, or all examined steps if none has been chosen. At the start of a block, it is the incoming tail defined in \cref{sec:fixed-c-block}.}
Over consecutive blocks, the sum of these stepsizes can grow large enough for quadratic contraction to exhaust the excess $D_i$.
For $\kappa\ge4$, we therefore perform a repair at the end of such a block whenever the stepsizes in the unselected tail sum to at least the \emph{repair limit}
\begin{equation}\label{eq:repair-limit}
 S_0:=\frac{\kappa}{256C_b(64C_b+1)}.
\end{equation}
Here we take the absolute constant $C_b$ in \cref{lem:fixed-c-block} to be at least $1$ throughout the rest of the paper.
A repair selects one or two checkpoints in this tail and uses the Huber repair components of \cref{lem:huber-local} at these checkpoints.
By \eqref{eq:huber-ratio-mass}, a bounded ratio $s_i/b_i$ gives a bounded relative threshold $\ell_{i+1}/D_{i+1}$ when $s_i$ is small relative to $\kappa$.
This requires $b_i$ to be at least a constant fraction of $s_i$.
However, satisfying this condition alone does not restrict the sum of stepsizes remaining after the checkpoint in the current tail.
We therefore design the following repair checkpoint selection rule.
\paragraph{Repair checkpoint selection rule.}
\leavevmode
For a repair, let $h_1,\ldots,h_r\ge0$ denote the stepsizes in the unselected tail through the end of the current block, listed in their original order.
Suppose $S:=h_1+\cdots+h_r\ge S_0>0$.
We select $q\in\{1,2\}$ checkpoints from the shortest suffix whose stepsizes sum to more than $S_0/8$, and write $b_i:=h_{t_i}$ and $t_0=0$.
\begin{itemize}[leftmargin=*]
\item If this suffix contains a stepsize $h_t\ge S_0/16$, select one such index as the only checkpoint $t_1$.
\item Otherwise, select $t_1<t_2$ in the suffix with $b_1,b_2\ge S_0/(64r^2)$ and $s_2:=\sum_{t=t_1+1}^{t_2-1}h_t\le\min\{b_1,b_2\}$. Such a pair exists by \cref{lem:repair-pair}.
\end{itemize}
\begin{samepage}
The following lemma gives bounds for repair checkpoints selected by this rule.
For each repair, we use the same local numbering, with incoming pair $(\ell_1,D_1)$ at the start of its unselected tail.

\begin{lemma}[Bounds for repair checkpoints]
\label{lem:repair}
There is a constant $c_0>0$, depending only on $C_b$, with the following property. Let $\kappa\ge4$ and use $S_0$ from \eqref{eq:repair-limit}. Suppose that $1\le S_0\le S\le2C_bS_0$ and that the incoming pair $\ell_1\ge0$, $D_1>0$ satisfies
\[
 \frac{\ell_1}{D_1}\le64C_b.
\]
Then the Huber repair components of \cref{lem:huber-local} can be constructed successively at any checkpoints selected by the rule above.
They satisfy $\eta_i>0$ and $D_{i+1}>0$ for every $1\le i\le q$, with the following bounds.
\begin{enumerate}[label=(\roman*),leftmargin=*]
\item \textbf{Transfer ratio lower bound.}
\begin{equation*}
 \frac{D_{q+1}}{D_1}\ge c_0(r+1)^{-4}.
\end{equation*}
\item \textbf{Relative threshold upper bound.} $\ell_{q+1}/D_{q+1}\le64C_b$.
\item \textbf{Remaining tail bound.} $\displaystyle\sum_{t=t_q+1}^r h_t\le S_0/4$.
\end{enumerate}
\end{lemma}
\end{samepage}
The proof is given in \cref{app:repair}.

Noted that the repair selection rule leaves a tail whose stepsizes sum to at most $S_0/4$, as stated in \cref{lem:repair}~(iii).
Thus no second repair is needed at the end of the same block.
With the block length chosen later in \eqref{eq:fixed-block-length}, \cref{lem:fixed-c-block}~(ii) gives $s_{\rm tail}\le C_bam^p\le S_0/4$ for a nonempty maximizing set, so such a block needs no repair.
We next bound the total number of stepsize indices considered by repairs across all blocks.
\begin{lemma}[Counting stepsize indices considered by repairs]\label{lem:repair-cost}
For the sequential checkpoint selection above, let $B:=\lceil n/m\rceil$ and number the repairs in order as $j=1,\ldots,J$.
For each repair $j$, the indices considered are all original stepsize indices in its unselected tail; let $r_j$ denote their number.
Then $J\le B$, and each stepsize index is considered by at most two repairs.
Consequently,
\[
 \sum_{j=1}^J r_j\le2n,\qquad
 \sum_{j=1}^J\log(r_j+1)\le B\log(2m+1).
\]
\end{lemma}
\begin{proof}
For repair $j$, let $u_j$ and $v_j$ be the first and last indices of its unselected tail, and let $p_j$ be its last selected checkpoint.
Thus $r_j=v_j-u_j+1$ and $u_{j+1}\ge p_j+1$.
These indices are illustrated in \Cref{fig:repair-indices}.

We first show $p_{j+1}>v_j$ for $1\le j<J$.
Otherwise, $u_{j+1}\ge p_j+1$, the repair trigger, and the $S_0/4$ remaining-tail bound guaranteed by the suffix selection rule give
\[
 S_0\le\sum_{t=u_{j+1}}^{v_{j+1}}h_t
 \le\sum_{t=p_j+1}^{v_j}h_t
    +\sum_{t=p_{j+1}+1}^{v_{j+1}}h_t
 \le\frac{S_0}{2},
\]
a contradiction.
Hence $u_{j+2}\ge p_{j+1}+1>v_j$ for $1\le j\le J-2$, so repairs whose numbers differ by at least two share no stepsize indices.
This proves $\sum_{j=1}^J r_j\le2n$.

Pad $r_1,\ldots,r_J$ with zeros to obtain $B$ entries.
Concavity of the logarithm and $n\le Bm$ give
\[
 \sum_{j=1}^J\log(r_j+1)
 \le B\log\left(1+\frac{\sum_{j=1}^J r_j}{B}\right)
 \le B\log(2m+1).
\]
\end{proof}

\begin{figure}[!htbp]
\centering
\begingroup
\definecolor{repairblue}{HTML}{345D9D}
\definecolor{repairorange}{HTML}{BF620C}
\definecolor{repairink}{HTML}{242731}
\definecolor{repairgray}{HTML}{717783}
\definecolor{repairlight}{HTML}{CAD0D7}
\begin{tikzpicture}[x=1cm,y=1cm,>=Latex,font=\footnotesize,
                    text=repairink,inner sep=1pt]
\path[use as bounding box] (0,0) rectangle (16.2,4.15);

\foreach \i/\name in {0/{preceding block},1/{$\cdots$},
                         2/{preceding block},3/{current block}}{
  \pgfmathsetmacro{\leftedge}{.75+3.675*\i}
  \pgfmathsetmacro{\rightedge}{\leftedge+3.675}
  \filldraw[fill=repairlight!12,draw=repairlight,line width=.45pt]
    (\leftedge,3.40) rectangle (\rightedge,3.95);
  \node at ({(\leftedge+\rightedge)/2},3.675) {\name};
}
\foreach \x in {.75,4.425,8.1,11.775,15.45}{
  \draw[repairlight,dashed,line width=.4pt] (\x,3.39)--(\x,2.25);
}

\coordinate (previous) at (2.89375,2.60);
\coordinate (start) at (3.50625,2.60);
\coordinate (last) at (13.30625,2.60);
\coordinate (retained) at (13.91875,2.60);
\coordinate (end) at (15.14375,2.60);
\fill[repairblue!7] (3.20,2.38) rectangle (15.45,2.82);
\draw[repairlight,line width=.55pt] (.40,2.60)--(15.80,2.60);
\foreach \i in {1,...,24}{
  \fill[repairgray] ({.75+.6125*(\i-.5)},2.60) circle (1pt);
}

\draw[repairink,line width=.8pt] ($(previous)+(0,-.14)$)--($(previous)+(0,.19)$);
\filldraw[fill=repairink,draw=white,line width=.35pt] (previous) circle (2.1pt);
\node[font=\small] at ($(previous)+(0,-.39)$) {$u_j-1$};
\node[font=\scriptsize,align=center] at ($(previous)+(-.55,-.86)$)
  {previous checkpoint};

\draw[repairblue,line width=.45pt] ($(start)+(0,.07)$)--($(start)+(0,.26)$);
\node[text=repairblue,font=\small] at ($(start)+(0,.44)$) {$u_j$};
\draw[repairblue,line width=.45pt] ($(end)+(0,.07)$)--($(end)+(0,.26)$);
\node[text=repairblue,font=\small] at ($(end)+(0,.44)$) {$v_j$};

\draw[repairorange,line width=.8pt] ($(last)+(0,-.14)$)--($(last)+(0,.19)$);
\filldraw[fill=repairorange,draw=white,line width=.35pt] (last) circle (2.1pt);
\node[text=repairorange,font=\small] at ($(last)+(0,.44)$) {$p_j$};

\draw[repairgray,line width=.45pt]
  ($(retained)+(0,-.36)$)--($(retained)+(0,-.48)$)
  --($(end)+(0,-.48)$)--($(end)+(0,-.36)$);
\node[text=repairgray,font=\scriptsize,align=center] at (14.53125,1.78)
  {unselected tail\\$p_j+1,\ldots,v_j$};

\draw[repairblue,line width=.55pt]
  ($(start)+(0,-1.19)$)--($(start)+(0,-1.34)$)
  --($(end)+(0,-1.34)$)--($(end)+(0,-1.19)$);
\node[text=repairblue,align=center] at (9.325,.83)
  {repair $j$: $r_j=v_j-u_j+1$ stepsize indices};
\end{tikzpicture}
\endgroup
\caption{Repair $j$ considers the $r_j=v_j-u_j+1$ stepsize indices $u_j,\ldots,v_j$, from just after the previous checkpoint $u_j-1$ to the end of the current block.
The last selected repair checkpoint $p_j$ is highlighted, and the steps after it remain unselected.
Although shown in the current block, $p_j$ may lie in an earlier block.}
\label{fig:repair-indices}
\end{figure}

\section{Assembling the Hard Function and Proving \texorpdfstring{\cref{thm:main}}{Theorem \getrefnumber{thm:main}}}\label{sec:main-proof}
In this section, we assemble the local hard functions into the global objective \eqref{eq:global-function} and prove \cref{thm:main}.

\subsection{Constructing Components at a Maximizing Checkpoint Set}\label{sec:mixed-block}

We now construct the local components for a single block whose maximizing checkpoint set is nonempty.
Fix $\kappa\ge4$ and $0<\beta\le1/4$, and set $a=a_\beta$, $c=c_\beta$, and $p=p(c)$.
Recall that $C_b$ is the absolute constant in \cref{lem:fixed-c-block}.
With $S_0$ fixed by \eqref{eq:repair-limit}, choose the block length
\begin{equation}\label{eq:fixed-block-length}
 m:=\left\lfloor\left(\frac{S_0}{4C_ba_\beta}\right)^{1/p(c_\beta)}\right\rfloor.
\end{equation}
Assume $m\ge1$. This choice gives
\begin{equation}\label{eq:budget-choice}
 C_bam^p\le S_0/4.
\end{equation}
Consider a block $\tilde{H}=(\tilde{h}_1,\ldots,\tilde{h}_m)$ as in \cref{sec:fixed-c-block}, with an incoming tail whose stepsizes sum to $0\le w<S_0$.
Let $T=\{t_1<\cdots<t_k\}$ be a nonempty admissible checkpoint set.\footnote{Admissibility means $\tilde{h}_{t_i}>8$ at every selected checkpoint, as defined in \cref{sec:fixed-c-block}.}
Write $b_i=\tilde{h}_{t_i}$ and take the sums of gap stepsizes $s_i$ from \eqref{eq:carried-gap-definitions}.
We use the local component notation of \cref{sec:fixed-c-block}, with incoming pair $(\ell_1,D_1)$.

\paragraph{Bridging Back to Bending.}
\leavevmode
The component used at the latest checkpoint before this block, if any, may be a Huber repair component.
It can contribute a nonzero term $h_jf_j$ with $f_j\ge0$ to the update of the first component's input coordinate $X$ in \eqref{eq:local-dynamics}.
The bending construction requires $f_j=0$.
As explained in \cref{sec:local-bridge}, we therefore use a bridging component at $t_1$ and bending components at the remaining checkpoints (see \cref{fig:architecture}).

\paragraph{Bounding the Quadratic Contraction.}
\leavevmode
The score $P_w(T;\tilde{H})$ used for checkpoint selection in \cref{sec:fixed-c-block} omits quadratic contraction.
For a maximizing set $T$ with $\ell_1/D_1\le64C_b$, \cref{lem:mixed-block} shows that these bridging and bending components retain at least one eighth of the incoming excess despite this contraction.
It also bounds the outgoing relative threshold by $4\,(\le64C_b)$.
Together with the repair guarantee in \cref{lem:repair}~(ii), this preserves the same bound on the relative threshold passed to subsequent blocks.

\begin{lemma}\label{lem:mixed-block}
For the block $\tilde{H}$ and incoming tail specified above, let $T=\{t_1<\cdots<t_k\}$ be a nonempty admissible checkpoint set that maximizes the carried score $P_w(T;\tilde{H})$ in \eqref{eq:carried-score} over all admissible checkpoint sets, including the empty set. Let $\ell_1\ge0$ and $D_1>0$, and suppose that
\begin{equation}\label{eq:mixed-smallness}
 \frac{\ell_1}{D_1}\le64C_b.
\end{equation}
Then the bridge construction in \cref{lem:bridge} can be applied at $t_1$, followed by the bending construction in \cref{lem:bending-local} at each $t_i$, $2\le i\le k$. Their local trajectories satisfy \eqref{eq:local-transfer-states}, and the resulting pairs satisfy
\[
 \frac{D_{k+1}}{D_1}\ge\frac18,\qquad
 \frac{\ell_{k+1}}{D_{k+1}}\le4.
\]
\end{lemma}
\begin{proof}
Since $T$ maximizes the carried score over all admissible checkpoint sets, \cref{lem:fixed-c-block}~(ii) applies. Together with the incoming-tail bound $w<S_0$ and \eqref{eq:budget-choice}, it gives
\[
 c\sum_{i=1}^k(a+s_i+8)
 \le C_b(w+am^p)
 <\left(C_b+\frac14\right)S_0\le2C_bS_0.
\]
The definition of $S_0$ in \eqref{eq:repair-limit} therefore gives
\begin{equation}\label{eq:mixed-contraction-budget}
 \frac{c(64C_b+1)}{\kappa}\sum_{i=1}^k(a+s_i+8)
 \le2C_b(64C_b+1)\frac{S_0}{\kappa}=\frac1{128}.
\end{equation}
Since $s_i\le c(a+s_i+8)$, this implies $s_i/\kappa\le1/128<1/16$ for every $1\le i\le k$.
Whenever $\eta_i>0$, \cref{lem:bridge} for $i=1$ and \cref{lem:bending-local} for $i\ge2$ construct the $i$th component and give
\begin{align}
 \frac{D_2}{D_1}&\ge\frac14\frac{b_1-8}{c(a+s_1+8)}\eta_1
 &&\text{(bridge at $t_1$)},\label{eq:bridge-score}\\
 \frac{D_{i+1}}{D_i}&\ge\frac{b_i-8}{c(a+s_i+8)}\eta_i(1-s_i/\kappa)
 &&(2\le i\le k),\label{eq:bending-score}\\
 \frac{\ell_{i+1}}{D_{i+1}}&\le\frac{4(s_i+2)}{b_i-8}
 &&(1\le i\le k).\label{eq:ratios}
\end{align}
We verify $\eta_i>0$ in checkpoint order using \eqref{eq:local-gap-estimates}. For the first gap,
\[
 1-\eta_1\le\frac{s_1}{\kappa}\left(1+\frac{\ell_1}{D_1}\right)
 \le\frac1{128},
\]
so $\eta_1>0$ and the bridge can be constructed.
Now fix $2\le i\le k$ and suppose that $\Phi_1,\ldots,\Phi_{i-1}$ have been constructed.
Then \eqref{eq:ratios} at $i-1$ gives $\ell_i/D_i\le4(s_{i-1}+2)/(b_{i-1}-8)$.
We bound $1-\eta_i$ using the two adjacent terms in the sum in \eqref{eq:mixed-contraction-budget}:
\[
 \begin{aligned}
 1-\eta_i
 &\le\frac{a+s_i+8}{\kappa}
 +\frac{4(a+s_{i-1}+8)(a+s_i+8)}{\kappa(b_{i-1}-8)}\\
 &\le\frac{9(a+s_i+8)+8(a+s_{i-1}+8)}{\kappa}.
 \end{aligned}
\]
The second inequality follows from \cref{lem:fixed-c-block}~(iii).
By \eqref{eq:mixed-contraction-budget}, the last expression is at most $9/128<1$.
Thus $\eta_i>0$, so the next bending component can be constructed.

Summing these bounds and using $\ell_1/D_1\le64C_b$ and \eqref{eq:mixed-contraction-budget} gives
\[
 \sum_{i=1}^k(1-\eta_i)+\sum_{i=2}^k\frac{s_i}{\kappa}
 \le\frac{18(64C_b+1)}{\kappa}\sum_{i=1}^k(a+s_i+8)
 \le\frac{18}{128c}<\frac12.
\]
All factors $\eta_i$ and $1-s_i/\kappa$ lie in $(0,1]$. Since $T$ is maximizing and the empty set is admissible, $P_w(T;\tilde{H})\ge P_w(\varnothing;\tilde{H})=1$. Using $\prod_j(1-u_j)\ge1-\sum_j u_j$ for $0\le u_j\le1$, and multiplying \eqref{eq:bridge-score}--\eqref{eq:bending-score}, we obtain
\[
 \begin{aligned}
 \frac{D_{k+1}}{D_1}&\ge\frac14 P_w(T;\tilde{H})
       \prod_{i=1}^k\eta_i\prod_{i=2}^k(1-s_i/\kappa)\\
 &\ge\frac18 P_w(T;\tilde{H})\ge\frac18.
 \end{aligned}
\]
Finally, \cref{lem:fixed-c-block}~(iv) and \eqref{eq:ratios} at $i=k$ give
\[
 \frac{\ell_{k+1}}{D_{k+1}}\le\frac{4(s_k+2)}{b_k-8}
 \le\frac{4(s_k+2)}{c(a+s_k+8)}\le4.
\]
\end{proof}

\subsection{Constructing the Global Hard Function}\label{sec:all-horizons}

Combining \cref{lem:mixed-block} with the repair bounds in \cref{lem:repair,lem:repair-cost}, we construct the global hard function and obtain the following lower bound on squared distance to the minimizer.
\begin{proposition}[Global lower bound for fixed $\beta$]\label{prop:fixed}
With $S_0$ fixed by \eqref{eq:repair-limit}, there is an absolute constant $C>0$ with the following property.
For every $0<\beta\le1/4$ and $\kappa\ge4$, let $m$ be the block length in \eqref{eq:fixed-block-length}.
If $S_0\ge1$ and $m\ge1$, then every horizon $n\ge1$ and every nonnegative schedule $H$ satisfy
\begin{equation}\label{eq:fixed-prefactor}
 \Derr_{n,\kappa}(H)\ge
 \exp\!\left[-C\left(1+\frac nm\right)\log(m+1)\right].
\end{equation}
\end{proposition}
\begin{proof}
The choice of $m$ gives \eqref{eq:budget-choice}. Let $c_0$ be the constant in \cref{lem:repair}.

\textbf{Step 1. Selecting checkpoints and constructing components.}
Partition the schedule into $B:=\lceil n/m\rceil$ fixed blocks, each containing at most $m$ steps.
Following the procedure illustrated in \cref{fig:architecture}, we examine the blocks in order and construct the components in checkpoint order.
Before processing each block, let $w$ be the sum of stepsizes in its incoming tail, and let $\ell\ge0$ and $D>0$ specify the activation threshold and excess at the start of the next component's gap.
This pair is carried from the latest chosen checkpoint, while the incoming tail records the steps already examined since that checkpoint.
We maintain
\begin{equation}\label{eq:global-invariant}
w<S_0,\qquad\frac{\ell}{D}\le64C_b.
\end{equation}
Initially, $w=\ell=0$ and $D=1$. For the next block, maximize \eqref{eq:carried-score} over its checkpoint sets with this $w$, choosing a nonempty maximizing set whenever one exists.

\emph{Case 1: The maximizing checkpoint set is nonempty.}
By \eqref{eq:global-invariant}, \cref{lem:mixed-block} applies with incoming threshold $\ell$ and excess $D$.
It gives a bridge at the first selected checkpoint and bending components at the remaining checkpoints, with combined transfer ratio at least $1/8$.
At the last selected checkpoint, the new relative threshold $\ell/D$ is at most $4\le64C_b$.
By \cref{lem:fixed-c-block}~(ii) and \eqref{eq:budget-choice}, the remaining tail stepsizes sum to $w\le S_0/4$.
Thus \eqref{eq:global-invariant} is maintained.

\emph{Case 2: The maximizing checkpoint set is empty.} All steps in this block remain unselected, so they extend the tail after the latest checkpoint.
Part~(i) of \cref{lem:fixed-c-block} implies that the sum of stepsizes in this enlarged tail is at most $C_b(w+am^p)\le2C_bS_0$.
If this sum is less than $S_0$, retain this tail and leave $(\ell,D)$ unchanged.
Otherwise perform repair $j$ on the current unselected tail.
Its $r_j$ stepsizes sum to a value in $[S_0,2C_bS_0]$, and the incoming pair satisfies $\ell/D\le64C_b$.

Select checkpoints by the repair rule in \cref{sec:checkpoint-repair}.
Applying \cref{lem:repair} to this tail with incoming threshold $\ell$ and excess $D$ gives Huber repair components at these checkpoints.
Their transfer ratios have product at least $c_0(r_j+1)^{-4}$, the new relative threshold is at most $64C_b$, and the remaining tail stepsizes sum to at most $S_0/4$.
Thus \eqref{eq:global-invariant} is maintained.

By \cref{lem:huber-local,lem:bridge}, the Huber repair and bridge outputs $D_{i+1}$ do not depend on the nonnegative contributions $f_j$. Thus all incoming pairs and components can be determined in checkpoint order from the schedule before GD is run.

\textbf{Step 2. Defining the global function and verifying the trajectory.}
Write all selected checkpoint indices in the full schedule as $t_1<\dots<t_k$, and set $t_0=0$. Define $F$ on $\R^{k+1}$ by \eqref{eq:global-function}, with $x_0=e_1$. Each $\Phi_i$ is nonnegative, convex, $1$-smooth, and zero at $(0,0)$. Since each coordinate appears in at most two of the $\Phi_i$, $F$ is $1$-smooth and $1/\kappa$-strongly convex, with minimizer $x_*=0$.

We now verify the trajectory by induction over the checkpoints. The induction hypothesis is
\[
x^{(i)}_{t_{i-1}}=\ell_i+D_i,\qquad x^{(j)}_{t_{i-1}}=0\quad(i+1\le j\le k+1),
\]
which holds for $i=1$ by $x_0=e_1$.
Within gap $i$, we induct over the updates to verify \eqref{eq:local-transfer-states} and keep $x^{(j)}=0$ for $j\ge i+2$.
These bounds make all later components inactive.
If $\Phi_i$ is bending, its predecessor is a bridge or bending component and contributes zero to the shared coordinate since $x^{(i)}\ge\ell_i$.
For a repair component or bridge, $\partial_Y\Phi_{i-1}\le0$ gives $f_j\ge0$ when $i>1$; for $i=1$, $f_j=0$.
All other earlier components involve neither coordinate, so the update follows \eqref{eq:local-dynamics}.
The local lemmas preserve the gap bounds and give $x^{(i+1)}_{t_i}=\ell_{i+1}+D_{i+1}$ at the checkpoint, with all later coordinates still zero, completing the induction.

\textbf{Step 3. Bounding the final coordinate.} The procedure applies the repair rule at most once per block, so \cref{lem:repair-cost} applies. Each repair $j$ has a product of transfer ratios at least $c_0(r_j+1)^{-4}$. By \cref{lem:repair-cost} and $J\le B$, their combined factor is at least
\[
\prod_{j=1}^Jc_0(r_j+1)^{-4}\ge\exp[-CB\log(m+1)],
\]
where $C$ depends only on $c_0$. Each block with a nonempty maximizing checkpoint set contributes a factor at least $1/8$ by \cref{lem:mixed-block}. There are at most $B$ such blocks, so, after increasing $C$, the excess $D_{k+1}$ after the last checkpoint is at least $\exp[-CB\log(m+1)]$.

If $k\ge1$, then $x^{(k+1)}_{t_k}=\ell_{k+1}+D_{k+1}\ge D_{k+1}$. Since $\partial_{x^{(k+1)}}\Phi_k\le0$, every remaining update $t_k<t\le n$ satisfies $x^{(k+1)}_t\ge(1-h_t/\kappa)x^{(k+1)}_{t-1}$. By \eqref{eq:global-invariant}, these remaining stepsizes sum to less than $S_0<\kappa/2$, so
\[
x^{(k+1)}_n\ge\left(1-\frac{S_0}{\kappa}\right)D_{k+1}\ge\frac12\exp[-CB\log(m+1)].
\]
Since $\|x_0-x_*\|^2=1$ and $\|x_n-x_*\|^2\ge(x^{(k+1)}_n)^2$, increasing $C$ and using $B\le1+n/m$ proves \eqref{eq:fixed-prefactor}. If $k=0$, the entire schedule satisfies $\sum_{t=1}^nh_t<S_0$, and the pure quadratic $F(x)=x^2/(2\kappa)$ gives the same conclusion.
\end{proof}

The factor $\exp[-C\log(m+1)]$ in \eqref{eq:fixed-prefactor} is independent of the horizon $n$.
The next lemma shows that we can remove this factor using a repetition argument.
We also include the corresponding lower bound for function suboptimality.
\begin{lemma}\label{lem:repetition}
Under the hypotheses of \cref{prop:fixed}, let $m$ and $C$ be as there. For every integer $n\ge1$ and every schedule $H\in[0,\infty)^n$,
\begin{equation}\label{eq:fixed-no-prefactor}
 \min\{\Derr_{n,\kappa}(H),\Eerr_{n,\kappa}(H)\}
 \ge\exp\!\left[-\frac{Cn}{m}\log(m+1)\right].
\end{equation}
\end{lemma}
\begin{proof}
Fix $n$ and $H$. For each integer $K\ge1$, let $H^{[K]}$ denote $K$ consecutive copies of $H$, a schedule of length $Kn$.
By \eqref{eq:distance-risk}, the squared distance to the minimizer after each copy is at most $\Derr_{n,\kappa}(H)$ times its value before that copy. The same holds for function suboptimality with $\Eerr_{n,\kappa}(H)$. Applying these bounds on the same objective gives
\[
 \Derr_{Kn,\kappa}(H^{[K]})\le\Derr_{n,\kappa}(H)^K,\qquad
 \Eerr_{Kn,\kappa}(H^{[K]})\le\Eerr_{n,\kappa}(H)^K.
\]
Since $\nabla F(x_*)=0$, $1/\kappa$-strong convexity and $1$-smoothness give
\[
 \frac{1}{2\kappa}\|x-x_*\|^2
 \le F(x)-F(x_*)
 \le \frac12\|x-x_*\|^2.
\]
Applying these bounds at the initial and final iterates and taking suprema yields $\Derr_{Kn,\kappa}(H^{[K]})\le\kappa\Eerr_{Kn,\kappa}(H^{[K]})$.
Combining these inequalities with \cref{prop:fixed} at horizon $Kn$ yields
\[
 \begin{aligned}
 \exp\!\left[-C\left(1+\frac{Kn}{m}\right)\log(m+1)\right]
 &\le\Derr_{Kn,\kappa}(H^{[K]})\\
 &\le\min\bigl\{\Derr_{n,\kappa}(H)^K,\,
                  \kappa\Eerr_{n,\kappa}(H)^K\bigr\}.
 \end{aligned}
\]
Taking $K$th roots and letting $K\to\infty$ proves \eqref{eq:fixed-no-prefactor}, since $\kappa^{1/K}\to1$.
\end{proof}

\paragraph{Choosing the Bending Parameter.}
Note that the block length $m$ defined in \eqref{eq:fixed-block-length} depends on $\beta$, and increasing $m$ strengthens the bound in \eqref{eq:fixed-no-prefactor}.
We therefore choose $\beta$ as a function of $\kappa$ to optimize this bound.
The explicit constants in \cref{lem:bending-local} satisfy
\begin{equation*}
 c_\beta=1+O(\beta),\qquad
 \log a_\beta=O(1/\beta),\qquad
 \frac1{p(c_\beta)}=1/\psil-O(\beta).
\end{equation*}
The first two estimates are given in \eqref{eq:bending-parameter-growth}, and the last follows by differentiating \eqref{eq:checkpoint-exponents} near $c=1$.
For sufficiently large $\kappa$, choose
\[
 \beta:=(\log\kappa)^{-1/2}\le1/4.
\]
Since $\log a_\beta=O(\sqrt{\log\kappa})=o(\log\kappa)$, we have $\kappa/a_\beta\to\infty$.
Thus $S_0\ge1$ and $m\ge1$ for sufficiently large $\kappa$, as required by \cref{prop:fixed}.
The definition of $m$ and the estimates above give
\begin{align}
 \log m
 &\ge\frac1{p(c_\beta)}(\log\kappa-\log a_\beta)-O(1)\notag\\
 &\ge\frac{\log\kappa}{\psil}-C\bigl(\beta\log\kappa+1/\beta\bigr)
 \ge\frac{\log\kappa}{\psil}-C'\sqrt{\log\kappa}.\label{eq:optimized-block}
\end{align}
The loss $O(\beta\log\kappa)$ comes from the exponent $1/p(c_\beta)$, while $O(1/\beta)$ comes from $\log a_\beta$.
Our choice of $\beta$ balances these two losses.

\begin{proof}[Proof of \cref{thm:main,cor:complexity}]
For the chosen $\beta$, \cref{lem:repetition} gives \eqref{eq:fixed-no-prefactor}.
Since $m\le\kappa$ by its definition in \cref{prop:fixed}, we have $\log(m+1)\le C\log\kappa$.
Combining this estimate with \eqref{eq:fixed-no-prefactor} and \eqref{eq:optimized-block} yields \eqref{eq:main}.
To prove \cref{cor:complexity}, suppose that a schedule $H$ achieves $\Derr_{n,\kappa}(H)\le\varepsilon$ or $\Eerr_{n,\kappa}(H)\le\varepsilon$.
The lower bound in \eqref{eq:main} must then also be at most $\varepsilon$.
Taking logarithms and rearranging gives
\[
 n\ge
 \frac{\kappa^{1/\psil}}
 {C_1 e^{C_2\sqrt{\log\kappa}}\log\kappa}
 \log\frac1\varepsilon.
\]
This necessary condition holds for every schedule $H$ and for either error criterion, proving \eqref{eq:main-complexity}.
\end{proof}

\section{Concluding Remarks}\label{sec:conclusion}
In this paper, we proved an almost tight iteration lower bound for GD with arbitrary nonnegative predetermined stepsizes on smooth strongly convex objectives, matching the silver upper bound up to a subpolynomial factor in $\kappa$.
Three open questions remain:
\begin{enumerate}
\item \textit{Can we reduce the subpolynomial factor $e^{C\sqrt{\log\kappa}}\log\kappa$, which separates our lower bound in \eqref{eq:main-complexity} from the silver upper bound?}
This factor seems unavoidable in our current proof.
Reducing or removing this subpolynomial factor may require more refined local hard functions, or even a new framework for constructing hard instances.

\item \textit{Can we extend our lower bound to predetermined schedules that may contain negative stepsizes?}
Our previous work~\citep{YeLiu2026} gives an $\Omega(n^{-1.6342})$ lower bound for GD with possibly negative predetermined stepsizes on smooth convex objectives, leaving a polynomial gap to the silver upper bound.
We believe that a more refined construction and analysis can yield lower bounds matching the silver upper bounds up to subpolynomial factors in both the smooth convex and smooth strongly convex settings, even when negative stepsizes are allowed.

\item \textit{Can we connect the smooth convex and strongly convex lower bounds?}
Nesterov's analysis of accelerated first-order methods gives a single upper bound covering both settings~\citep[Theorem~2.2.2]{Nesterov2004}.
For GD, the smooth convex silver schedule arises as $\mu\to0$ after omitting the terminal step of the strongly convex schedule.
Conversely, repeating suitable smooth convex silver blocks recovers the strongly convex iteration bound up to constant factors~\citep[Remark~1 and Corollary~2]{AltschulerParrilo2025}.
It would be interesting to develop a unified lower bound construction that captures the transition between these settings.
\end{enumerate}

\section*{Acknowledgments}
We thank Benjamin Grimmer for helpful discussions during the \emph{Readings in Optimization} seminar at MIT and for reading and commenting on the manuscript.

\section*{AI Disclosure}
This work was driven by the human authors' conceptual direction. ChatGPT-6 Astra was utilized as an interactive assistant to execute instructions, work out technical details, and refine arguments through iterative feedback.
After digesting the arguments in detail, the authors reorganized and rewrote the presentation. They take full responsibility for the correctness and originality of all content.

\bibliographystyle{alpha}
\begingroup
\color{black}
\hypersetup{urlcolor=black}
\setlength{\bibsep}{6pt plus 1pt minus 1pt}
\bibliography{references}
\endgroup

\clearpage
\appendix
\crefalias{section}{appendix}
\section{Proofs for the Local Hard Functions}
\label{app:local}

In this section, we prove \cref{lem:bending-local,lem:huber-local,lem:bridge} by constructing $\Phi_i$ and the next threshold $\ell_{i+1}$ from the given $\ell_i,D_i$ and gap stepsizes.
The construction ensures \eqref{eq:local-transfer-states}, with $D_{i+1}$ determined by the checkpoint update.
We bound the transfer ratio $D_{i+1}/D_i$ from below to limit the loss of excess, and the relative threshold $\ell_{i+1}/D_{i+1}$ from above to limit contraction loss in the next gap.

All three constructions use the following property of Moreau envelopes of support functions.

\begin{fact}[Moreau envelopes of support functions]\label{fact:local-moreau}
Let $K\subset\R^2$ be compact and convex, with $0\in K$.
Write $\sigma_K(u)=\max_{g\in K}\langle g,u\rangle$ for its support function. Its Moreau envelope with parameter $1$ satisfies
\begin{equation*}
 \begin{aligned}
 (\env_1\sigma_K)(x)
 &=\min_{y\in\R^2}\left\{\sigma_K(y)+\frac12\|x-y\|^2\right\}\\
 &=\max_{g\in K}\left\{\langle g,x\rangle-\frac12\|g\|^2\right\}.
 \end{aligned}
\end{equation*}
This envelope is nonnegative, convex, and $1$-smooth, and
\begin{equation*}
 \nabla(\env_1\sigma_K)(x)=\Pi_K(x).
\end{equation*}
Here $\Pi_K$ denotes Euclidean projection onto $K$. For $g\in K$, the equality $g=\Pi_K(x)$ holds if and only if
\begin{equation}\label{eq:local-normal-certificate}
 \langle x-g,w-g\rangle\le0\qquad\text{for every }w\in K.
\end{equation}
\end{fact}

The envelope identities follow from~\citep[Propositions 5.b, 7.b, and 7.d]{moreau1965proximite}. For background, see~\citep[Sections 1.G and 2.D]{RockafellarWets1998} and~\citep[Sections 3.1 and 6.4.2]{ParikhBoyd2014}.

\subsection{\texorpdfstring{Proof of \cref{lem:bending-local}}{Proof of Lemma \getrefnumber{lem:bending-local}}}
\label{app:local-bending}

\paragraph{Constructing the gap trajectory.}
Fix $\kappa\ge4$, $0<\beta\le1/4$, $\ell_i\ge0$, and the nonnegative gap stepsizes with $s_i/\kappa\le1/16$.
As in our smooth convex construction~\citep[Section~3.1]{ye2026silver}, we first construct a two-dimensional gap trajectory.
After scaling and translation, this will give the GD trajectory generated by the bending component $\Phi_i$ during gap $i$.
We obtain the gradients by scaling vectors on a circular arc with center $C:=(\beta,\beta)$ and radius $R$, where
\begin{equation}\label{eq:bending-c-beta}
 R=1+\beta,\qquad c_\beta=\beta+\sqrt{1+2\beta}.
\end{equation}
Let $K$ be the convex hull of the origin and the shorter arc of this circle joining $(c_\beta,0)$ to $(\beta,-1)$.

The quadratic term in \eqref{eq:global-function} also contracts the iterates toward the origin.
Its effect in the new coordinates depends on the scaling and translation, so we specify these first.
For a scale $\varrho>0$ and a parameter $\tau\ge1$, both chosen later, define
\[
 \Phi_i(X,Y):=(\env_1\sigma_{\varrho K})\bigl((X,Y)-z^\circ\bigr),\qquad
 z^\circ:=\bigl(\ell_i+\ell_{i+1}/\beta,\ell_{i+1}\bigr),\qquad
 \ell_{i+1}:=\varrho\tau.
\]
By \cref{fact:local-moreau}, $\nabla\Phi_i(z^\circ+\varrho r)=\varrho\Pi_K(r)$ for every $r\in\R^2$.
We therefore write the gap trajectory as $z^\circ+\varrho r_j$ and use gradients $\varrho v_j$, where $v_j=\Pi_K(r_j)$.
Dividing \eqref{eq:local-dynamics} with $f_j=0$ by $\varrho$ gives
\begin{equation}\label{eq:bending-gap-update}
 r_j+\frac{z^\circ}{\varrho}
 =\left(1-\frac{h_j}{\kappa}\right)
       \left(r_{j-1}+\frac{z^\circ}{\varrho}\right)
       -\frac{(1-1/\kappa)h_j}{2}v_{j-1}
       \quad(1\le j\le m_{\mathrm{gap}}),
\end{equation}
where $z^\circ/\varrho=(\ell_i/\varrho+\tau/\beta,\tau)$.
The quadratic part contracts toward $-z^\circ/\varrho$, the point corresponding to the original origin $(X,Y)=(0,0)$.

\Cref{lem:bending-gap} constructs points $r_j$ backward from $r_{m_{\mathrm{gap}}}=(\beta,-1)$ so that \eqref{eq:bending-gap-update} holds with $v_j=\Pi_K(r_j)$.
Since the center of contraction depends on $\tau$, the lemma also chooses $\tau$ so that the trajectory in $(X,Y)$ starts at $Y_0=0$.

\begin{lemma}[Gap trajectory]\label{lem:bending-gap}
For every $\varrho>0$, there exist $\tau\ge1$, points $r_0,\ldots,r_{m_{\mathrm{gap}}}$, and vectors $v_0,\ldots,v_{m_{\mathrm{gap}}}$ on the arc satisfying \eqref{eq:bending-gap-update}, such that $r_{m_{\mathrm{gap}}}=v_{m_{\mathrm{gap}}}=(\beta,-1)$, $r_0$ has second coordinate $-\tau$, and
\[
 v_j=\Pi_K(r_j),\qquad
 r_j\in[\beta,\infty)\times[-\tau,-1]
 \quad(0\le j\le m_{\mathrm{gap}}).
\]
Let $P$ be the first coordinate of the weighted sum
\begin{equation}\label{eq:bending-weighted-sums}
 \frac{1-1/\kappa}{2}\sum_{j=1}^{m_{\mathrm{gap}}}
 h_j\left(\prod_{l=j+1}^{m_{\mathrm{gap}}}
       \left(1-\frac{h_l}{\kappa}\right)\right)v_{j-1}.
\end{equation}
Then
\[
 P\le\frac{1-1/\kappa}{2}c_\beta s_i,\qquad
 \tau\le1+\frac{1-1/\kappa}{2}s_i,\qquad
 \tau\le\beta\exp\!\left(\frac{\sqrt2R-\beta}{\beta}\right).
\]
Moreover, $\tau$ and $P$ can be chosen as continuous functions of $\varrho$.
\end{lemma}

The term $\varrho P$ is the contribution of the gradient steps to the decrease in $X$ over the gap, after accounting for later contractions.
Every arc vector has first coordinate at most $c_\beta$, giving the bound on $P$.
This accounts for the coefficient $c_\beta$ in \eqref{eq:bending-local-bound}.
The bound on $\tau$ in terms of $s_i$ controls $\ell_{i+1}=\varrho\tau$, and hence the relative threshold.
The last bound gives the exponential term in the constant
\begin{equation}\label{eq:bending-a-beta}
 a_\beta:=8+4\beta+4\exp\!\left(\frac{\sqrt2R-\beta}{\beta}\right).
\end{equation}
It satisfies $a_\beta\ge8$ and $\log a_\beta=O(1/\beta)$.
The lemma is proved in \cref{app:bending-gap-proof}, and \cref{fig:bending-backward} illustrates one backward step.

\paragraph{Choosing the scale and verifying the local hard function.}
We now choose $\varrho$ in the function $\Phi_i$ defined above so that $z^\circ+\varrho r_0=(\ell_i+D_i,0)$.
We then verify the zero conditions and compute the excess after the checkpoint update.

\begin{proof}[Proof of \cref{lem:bending-local}]
For each $\varrho>0$, choose the quantities in \cref{lem:bending-gap} with $\tau$ and $P$ continuous in $\varrho$, and write $(X_j,Y_j):=z^\circ+\varrho r_j$ with $\ell_{i+1}=\varrho\tau$.
Since $r_0$ has second coordinate $-\tau$, we have $Y_0=0$.
The bounds on $P,\tau$ and \eqref{eq:bending-a-beta} give
\begin{equation}\label{eq:bending-scale-coefficient}
 \beta+P+\frac{\tau}{\beta}
 \le\frac{a_\beta}{4}+\frac{1-1/\kappa}{2}c_\beta s_i.
\end{equation}

\textbf{Step 1. Choosing the scale $\varrho$.}
We first choose $\varrho$ so that the scaled trajectory starts at $(\ell_i+D_i,0)$.
Multiplying \eqref{eq:bending-gap-update} by $\varrho$ gives
\[
 (X_j,Y_j)=\left(1-\frac{h_j}{\kappa}\right)(X_{j-1},Y_{j-1})
 -\frac{(1-1/\kappa)h_j}{2}\varrho v_{j-1}
 \quad(1\le j\le m_{\mathrm{gap}}).
\]
Expanding the first coordinate and using \eqref{eq:bending-weighted-sums}, we obtain $X_{m_{\mathrm{gap}}}=\chi_iX_0-\varrho P$.
Since $X_{m_{\mathrm{gap}}}=\ell_i+\varrho(\beta+\tau/\beta)$ and $\chi_i(\ell_i+D_i)-\ell_i=D_i\eta_i$, the condition $X_0=\ell_i+D_i$ is equivalent to
\begin{equation}\label{eq:bending-scale-match}
 \varrho\left(\beta+P+\frac{\tau}{\beta}\right)=D_i\eta_i.
\end{equation}
Here $P,\tau$ depend on $\varrho$.
By \cref{lem:bending-gap}, the coefficient $\beta+P+\tau/\beta$ is continuous in $\varrho$; it is bounded above by \eqref{eq:bending-scale-coefficient} and below by $\beta+1/\beta>0$.
Thus the left side tends to zero as $\varrho\downarrow0$ and to infinity as $\varrho\to\infty$.
Since $D_i\eta_i>0$, choose a solution and use this $\varrho$ in the function $\Phi_i$ defined above.
By \cref{fact:local-moreau,lem:bending-gap},
\[
 \nabla\Phi_i(X_j,Y_j)
 =\Pi_{\varrho K}(\varrho r_j)
 =\varrho\Pi_K(r_j)=\varrho v_j.
\]
The displayed recurrence therefore coincides with \eqref{eq:local-dynamics} with $f_j=0$.
Starting from the matched initial point, GD visits $(X_j,Y_j)$ by induction on $j$.
Since $r_j\in[\beta,\infty)\times[-\tau,-1]$, we have $X_j\ge\ell_i+\varrho(\beta+\tau/\beta)>\ell_i$ and $0\le Y_j\le\varrho(\tau-1)<\ell_{i+1}$ throughout the gap.
Moreover, \eqref{eq:bending-scale-coefficient} and \eqref{eq:bending-scale-match} imply
\begin{equation}\label{eq:bending-rho}
 \frac{\varrho}{D_i}\ge\frac{2\eta_i}{a_\beta/2+(1-1/\kappa)c_\beta s_i}.
\end{equation}
The function and threshold do not depend on $b_i$.

\paragraph{Step 2. Verifying $\Phi_i=0$ and $\partial_Y\Phi_i=0$.}
On the arc, $p\ge\beta$ and $0\le q\le1$, so $p,q\ge0$ and $q\le p/\beta$ throughout $K$ by convexity.
Moreover, $p\le c_\beta$ throughout $K$, so $(p,0)\in[0,c_\beta]\times\{0\}\subset K$ whenever $(p,-q)\in K$.
Apply \cref{lem:local-projection}, proved in \cref{app:local-projection}, to $\varrho K$ with $\theta=\beta$ and translation $z^\circ$.
It follows that $\Phi_i$ is nonnegative, convex, and $1$-smooth, with $\partial_Y\Phi_i\le0$ everywhere and
\[
 \Phi_i(X,Y)=0\quad(X\le\ell_i,Y\ge0),\qquad
 \partial_Y\Phi_i(X,Y)=0\quad(X\in\R,Y\ge\ell_{i+1}).
\]

\paragraph{Step 3. The checkpoint update.}
Just before the checkpoint update, $r_{m_{\mathrm{gap}}}=v_{m_{\mathrm{gap}}}=(\beta,-1)$ gives $Y_{m_{\mathrm{gap}}}=\varrho(\tau-1)$ and $\partial_Y\Phi_i=-\varrho$.
For any $b_i>8$, the step of size $b_i$ yields
\[
 D_{i+1}=Y_{m_{\mathrm{gap}}+1}-\ell_{i+1}
 =\varrho\left(\frac{1-1/\kappa}{2}b_i-\frac{b_i}{\kappa}(\tau-1)-1\right).
\]
Using the bound on $\tau$ in \cref{lem:bending-gap} and $4(1-1/\kappa)(1-s_i/\kappa)>1$, we obtain
\[
 D_{i+1}\ge\frac{1-1/\kappa}{2}\varrho(1-s_i/\kappa)(b_i-8)>0.
\]
This proves \eqref{eq:local-transfer-states}.
Dividing by $D_i$, using \eqref{eq:bending-rho} and $(1-1/\kappa)c_\beta\ge3/4$, gives the first bound in \eqref{eq:bending-local-bound}:
\[
 \frac{D_{i+1}}{D_i}
 \ge\frac{(1-1/\kappa)(1-s_i/\kappa)(b_i-8)\eta_i}
              {a_\beta/2+(1-1/\kappa)c_\beta s_i}
 \ge\frac{b_i-8}{c_\beta(a_\beta+s_i+8)}\eta_i(1-s_i/\kappa).
\]
For the second bound, $\ell_{i+1}=\varrho\tau$, the bound on $\tau$, and $(1-1/\kappa)(1-s_i/\kappa)\ge45/64$ give
\[
 \frac{\ell_{i+1}}{D_{i+1}}
 \le\frac{2+s_i}{(1-1/\kappa)(1-s_i/\kappa)(b_i-8)}
 \le4\frac{s_i+2}{b_i-8}.
\]
Finally, \eqref{eq:bending-c-beta} and \eqref{eq:bending-a-beta} give the parameter estimates in \eqref{eq:bending-parameter-growth}.
\end{proof}

\subsection{\texorpdfstring{Proof of \cref{lem:huber-local}}{Proof of Lemma \getrefnumber{lem:huber-local}}}
\label{app:local-huber}

\begin{proof}
\textbf{Step 1. Choosing the Huber scale.}
For a positive trial scale $\delta_{\mathrm H}$, define points $(X_j,Y_j)$ from $(X_0,Y_0)=(\ell_i+D_i,0)$ by \eqref{eq:local-dynamics}, using the constant vector $(\delta_{\mathrm H},-\delta_{\mathrm H})$ in place of $\nabla\Phi_i$.
For every sequence $f_j\ge0$, solving this linear recurrence gives, for $0\le j\le m_{\mathrm{gap}}$,
\[
 \begin{aligned}
 Y_j&=\frac{\kappa-1}{2}\delta_{\mathrm H}
       \left(1-\prod_{l=1}^{j}(1-h_l/\kappa)\right),\\
 X_j-Y_j-\ell_i
 &\ge\bigl(\ell_i+D_i+(\kappa-1)\delta_{\mathrm H}\bigr)
        \prod_{l=1}^{j}(1-h_l/\kappa)
        -\ell_i-(\kappa-1)\delta_{\mathrm H}\\
 &\ge D_i\eta_i-(\kappa-1)(1-\chi_i)\delta_{\mathrm H}.
 \end{aligned}
\]
The first inequality uses $f_j\ge0$ and $0\le h_j<\kappa$, and the second uses $\prod_{l=1}^{j}(1-h_l/\kappa)\ge\chi_i$.
The Huber gradient equals $(\delta_{\mathrm H},-\delta_{\mathrm H})$ when $X-Y-\ell_i\ge2\delta_{\mathrm H}$.
To keep every gap point in this region, choose the scale so that the last lower bound equals $2\delta_{\mathrm H}$:
\[
 \delta_{\mathrm H}=\frac{D_i\eta_i}{2+(\kappa-1)(1-\chi_i)}>0.
\]
This is the choice in \eqref{eq:huber-local-choice}, and it gives $X_j-Y_j-\ell_i\ge2\delta_{\mathrm H}$ throughout the gap.
Since $Y_j$ is nondecreasing in $j$, set
\[
 \ell_{i+1}=Y_{m_{\mathrm{gap}}}
 =\frac{\kappa-1}{2}(1-\chi_i)\delta_{\mathrm H}.
\]
The resulting function $\Phi_i$ and threshold $\ell_{i+1}$ depend only on the incoming data and gap stepsizes, so neither changes with the sequence $f_j$ or the checkpoint stepsize $b_i$.
The margin bound and monotonicity of $Y_j$ give $X_j\ge\ell_i$ and $0\le Y_j\le\ell_{i+1}$ for $0\le j\le m_{\mathrm{gap}}$.

\paragraph{Step 2. Verifying the local hard function.}
For $K=\conv\{0,(\delta_{\mathrm H},-\delta_{\mathrm H})\}$, the Huber repair function in \eqref{eq:huber-local-choice} has the Moreau representation
\[
 \Phi_i(X,Y)=(\env_1\sigma_K)(X-\ell_i,Y)
 =\tfrac12H_{2\delta_{\mathrm H}}(X-Y-\ell_i).
\]
By \cref{lem:local-projection} with $\theta=1$ and translation $(\ell_i,0)$, this function is nonnegative, convex, and $1$-smooth, with $\partial_Y\Phi_i\le0$ everywhere and $\Phi_i(X,Y)=0$ when $X\le\ell_i$, $Y\ge0$.
At each constructed point, the condition $X_j-Y_j-\ell_i\ge2\delta_{\mathrm H}$ makes the projection onto $K$ equal to $(\delta_{\mathrm H},-\delta_{\mathrm H})$.
Indeed, the left side of \eqref{eq:local-normal-certificate} at the other endpoint $0$ is
\[
 -\delta_{\mathrm H}(X_j-Y_j-\ell_i-2\delta_{\mathrm H})\le0.
\]
Thus $\nabla\Phi_i(X_j,Y_j)=(\delta_{\mathrm H},-\delta_{\mathrm H})$ for $0\le j\le m_{\mathrm{gap}}$.
Induction in \eqref{eq:local-dynamics} confirms that the constructed points are the GD iterates for every nonnegative sequence $f_j$.

\paragraph{Step 3. The checkpoint update.}
For any $b_i>0$ and $f_{m_{\mathrm{gap}}+1}\ge0$, this update uses $Y_{m_{\mathrm{gap}}}=\ell_{i+1}$ and $\partial_Y\Phi_i=-\delta_{\mathrm H}$, giving
\[
 D_{i+1}
 =b_i\left(\frac{1-1/\kappa}{2}\delta_{\mathrm H}
                  -\frac{\ell_{i+1}}{\kappa}\right)
 =\frac{1-1/\kappa}{2}b_i\chi_i\delta_{\mathrm H}>0.
\]
Here $\chi_i>0$ because every gap stepsize is strictly less than $\kappa$.
This proves \eqref{eq:local-transfer-states}. Substituting the chosen $\delta_{\mathrm H}$ and $\ell_{i+1}$ gives the exact ratios
\begin{equation}\label{eq:huber-exact-output}
 \frac{D_{i+1}}{D_i}
 =\frac{(1-1/\kappa)b_i\chi_i\eta_i}
 {2\bigl(2+(\kappa-1)(1-\chi_i)\bigr)},\qquad
 \frac{\ell_{i+1}}{D_{i+1}}
 =\frac{\kappa(1-\chi_i)}{b_i\chi_i}.
\end{equation}
In particular, $D_{i+1}$ is independent of the contributions $f_j$.
Finally, telescoping gives
\[
 \kappa(1-\chi_i)
 =\sum_{j=1}^{m_{\mathrm{gap}}}h_j
       \prod_{l=j+1}^{m_{\mathrm{gap}}}(1-h_l/\kappa),\qquad
 \chi_i s_i\le\kappa(1-\chi_i)\le s_i.
\]
Together with $\chi_i\ge1-s_i/\kappa$, these bounds prove \eqref{eq:huber-ratio-mass} when $s_i<\kappa$.
They also give $2+(\kappa-1)(1-\chi_i)\le2+s_i$. When $s_i<\kappa$, combining this denominator bound and $\chi_i\ge1-s_i/\kappa$ with \eqref{eq:huber-exact-output} proves the transfer lower bound in \eqref{eq:huber-local-output}. When $s_i\ge\kappa$, its right-hand side is nonpositive, so the same lower bound follows from $D_{i+1}>0$.
\end{proof}

\subsection{\texorpdfstring{Proof of \cref{lem:bridge}}{Proof of Lemma \getrefnumber{lem:bridge}}}
\label{app:local-bridge}

\begin{proof}
\textbf{Step 1. Raising the threshold.}
Use the Huber repair trajectory from \cref{app:local-huber}, with the same $\ell_i,D_i$ and gap stepsizes.
For every nonnegative sequence $f_j$, it satisfies
\[
 X_j-Y_j-\ell_i\ge2\delta_{\mathrm H},\qquad
 0\le Y_j\le Y_{m_{\mathrm{gap}}}
 \quad(0\le j\le m_{\mathrm{gap}}).
\]
The repair threshold is $Y_{m_{\mathrm{gap}}}$.
For the bridge, raise it by $\delta_{\mathrm H}$ and denote the new threshold by $\ell_{i+1}$:
\[
 \ell_{i+1}=Y_{m_{\mathrm{gap}}}+\delta_{\mathrm H}
 =\left(1+\frac{\kappa-1}{2}(1-\chi_i)\right)\delta_{\mathrm H}
 =\frac{D_i\eta_i}{2}.
\]
Every gap point is now at least $\delta_{\mathrm H}$ below this threshold, so $Y_j-\ell_{i+1}\le-\delta_{\mathrm H}$.
Since $Y_{m_{\mathrm{gap}}}$ and $\delta_{\mathrm H}$ do not depend on $f_j$ or $b_i$, neither does this threshold.

\paragraph{Step 2. Verifying the local hard function.}
The bridge construction in \cref{sec:local-bridge} enlarges the segment used for Huber repair to the triangle
\[
 K_\triangle=\conv\{0,(\delta_{\mathrm H},-\delta_{\mathrm H}),(\delta_{\mathrm H},0)\}.
\]
By \cref{fact:local-moreau}, the bridging function $\Phi_i$ can be written as
\[
 \Phi_i(X,Y)=(\env_1\sigma_{K_\triangle})(X-\ell_i-\ell_{i+1},Y-\ell_{i+1}).
\]
Every $(p,-q)\in K_\triangle$ satisfies $0\le q\le p$, and $(p,0)$ lies on its horizontal edge.
Applying \cref{lem:local-projection} with $\theta=1$ and translation $(\ell_i+\ell_{i+1},\ell_{i+1})$ shows that $\Phi_i$ is nonnegative, convex, and $1$-smooth, with $\partial_Y\Phi_i\le0$ everywhere and
\[
 \Phi_i(X,Y)=0\quad(X\le\ell_i,Y\ge0),\qquad
 \partial_Y\Phi_i(X,Y)=0\quad(X\in\R,Y\ge\ell_{i+1}).
\]
At each repair iterate, the translated point $(X_j-\ell_i-\ell_{i+1},Y_j-\ell_{i+1})$ projects onto $(\delta_{\mathrm H},-\delta_{\mathrm H})$.
The expression in \eqref{eq:local-normal-certificate} is affine in $w$, so it suffices to check the other vertices $0$ and $(\delta_{\mathrm H},0)$.
The corresponding inner products are
\[
 -\delta_{\mathrm H}(X_j-Y_j-\ell_i-2\delta_{\mathrm H})\le0,\qquad
 \delta_{\mathrm H}(Y_j-\ell_{i+1}+\delta_{\mathrm H})\le0.
\]
Both inequalities follow from Step 1, so the criterion holds throughout $K_\triangle$.
Thus the bridge and repair gradients agree for $0\le j\le m_{\mathrm{gap}}$.
Induction in \eqref{eq:local-dynamics}, using the same initial point and sequence $f_j$, shows that their trajectories agree through the checkpoint update.
In particular, $X_j\ge\ell_i$ and $0\le Y_j<\ell_{i+1}$ throughout the gap.

\paragraph{Step 3. The checkpoint update.}
For every $b_i>4$ and every nonnegative sequence $f_j$, the bridge has the same value of $Y$ after this update as the repair, but its threshold is $\delta_{\mathrm H}$ larger.
Subtracting this amount from the repair excess gives
\begin{equation}\label{eq:bridge-output}
 D_{i+1}=Y_{m_{\mathrm{gap}}+1}-\ell_{i+1}
 =\delta_{\mathrm H}
       \left(\frac{(1-1/\kappa)b_i\chi_i}{2}-1\right).
\end{equation}
Since $s_i/\kappa\le1/16$ and $\kappa\ge4$, we have $(1-1/\kappa)\chi_i\ge45/64>1/2$, and hence
\[
 (1-1/\kappa)b_i\chi_i-2
 \ge(1-1/\kappa)\chi_i(b_i-4)>0.
\]
Thus $D_{i+1}>0$, completing \eqref{eq:local-transfer-states}.
Substituting \eqref{eq:huber-local-choice} into \eqref{eq:bridge-output} and using $(\kappa-1)(1-\chi_i)\le s_i$ gives
\[
 \begin{aligned}
 \frac{D_{i+1}}{D_i}
 &\ge\frac{\eta_i(1-1/\kappa)\chi_i(b_i-4)}{2(s_i+2)}
 \ge\frac14\eta_i\frac{b_i-4}{s_i+2},\\
 \frac{\ell_{i+1}}{D_{i+1}}
 &\le\frac{s_i+2}{(1-1/\kappa)\chi_i(b_i-4)}
 \le2\frac{s_i+2}{b_i-4}.
 \end{aligned}
\]
\end{proof}

\clearpage
\subsectionnotoc{\texorpdfstring{Proof of \cref{lem:bending-gap}}{Proof of Lemma \getrefnumber{lem:bending-gap}}}
\label{app:bending-gap-proof}

We use the notation of \cref{app:local-bending}.

\begin{proof}
We construct the points backward for a trial value of $\tau$, then choose $\tau$ so that $r_0$ has second coordinate $-\tau$.
We verify $v_j=\Pi_K(r_j)$ and finally bound $\tau$.
Write $r_j=(u_j,-y_j)$ and $v_j=(p_j,-q_j)$, where $p_j=p(q_j)$ and
\[
 p(q):=\beta+\sqrt{R^2-(q+\beta)^2}\quad(0\le q\le1).
\]
We must show $y_0=\tau$, $u_j\ge\beta$, and $1\le y_j\le\tau$.
Let $Q$ be the weighted sum of the $q_{j-1}$ with the weights in \eqref{eq:bending-weighted-sums}, so that the vector in that display is $(P,-Q)$.

To choose the arc vectors, we measure directions from the center $C$.
For a point $z=(z_1,z_2)$ with $z_1>\beta$, let $\omega(z):=(\beta-z_2)/(z_1-\beta)$ be the downward slope from $C$ to $z$, and set $\omega(z)=+\infty$ if $z_1=\beta>z_2$.
Along the arc, $\omega(p(q),-q)=(q+\beta)/(p(q)-\beta)$ increases from $\omega_c:=\beta/(c_\beta-\beta)$ at $(c_\beta,0)$ to $+\infty$ at $(\beta,-1)$.
Thus each slope in $[\omega_c,+\infty]$ determines exactly one arc point.

For an empty gap, take $r_0=v_0=(\beta,-1)$, so $P=Q=0$ and $\tau=1$.
Each zero stepsize leaves the point and gradient unchanged.
We therefore set $r_{j-1}=r_j$, $v_{j-1}=v_j$ when $h_j=0$ and apply the following rule only when $h_j>0$.

Fix $\varrho>0$ and a trial value $\tau\ge1$, and start from $r_{m_{\mathrm{gap}}}=v_{m_{\mathrm{gap}}}=(\beta,-1)$.
Given $r_j$ with $u_j\ge\beta$, rearrange \eqref{eq:bending-gap-update} as $r_{j-1}=\widehat r_j+\alpha_jv_{j-1}$, where
\[
 \widehat r_j:=r_j+\frac{h_j/\kappa}{1-h_j/\kappa}
       \left(r_j+\frac{z^\circ}{\varrho}\right),\qquad
 \alpha_j:=\frac{(1-1/\kappa)h_j}{2(1-h_j/\kappa)}.
\]
Here $\widehat r_j$ undoes the quadratic contraction in step $j$.
Choose $v_{j-1}$ where the ray from $C$ through $\widehat r_j+\alpha_jC$ meets the arc, or choose $(c_\beta,0)$ if this ray passes above the corner.
Equivalently,
\begin{equation}\label{eq:bending-backward-rule}
 \omega(v_{j-1})=\max\{\omega_c,\omega(\widehat r_j+\alpha_jC)\}.
\end{equation}
Then set $r_{j-1}:=\widehat r_j+\alpha_jv_{j-1}$.
The first coordinate of $r_j+z^\circ/\varrho$ is positive, so $\widehat r_j$ has first coordinate at least $u_j\ge\beta$, and $\widehat r_j+\alpha_jC$ has first coordinate greater than $\beta$.
Thus the rule is well defined and depends continuously on $\widehat r_j$.
It also gives $u_{j-1}\ge\beta$, so it applies at every index $j=m_{\mathrm{gap}},\ldots,1$.
This is the rule of our smooth convex construction~\citep[Appendix~A]{ye2026silver}, with center $C$ and with $\widehat r_j$ in place of $r_j$.

\Cref{fig:bending-backward} shows the geometry of \eqref{eq:bending-backward-rule}.
When $q_{j-1}>0$, the center $(\beta,\beta)$, $v_{j-1}$, and $\widehat r_j+\alpha_j(\beta,\beta)$ are collinear, and
\[
 r_{j-1}-\bigl(\widehat r_j+\alpha_j(\beta,\beta)\bigr)
 =\alpha_j\bigl(v_{j-1}-(\beta,\beta)\bigr).
\]
Thus the two shaded triangles are similar, and $r_{j-1}$ lies on the same ray from $C$.
When $v_{j-1}=(c_\beta,0)$, the vector $r_{j-1}-C$ is the sum of $\widehat r_j+\alpha_jC-C$ and $\alpha_j((c_\beta,0)-C)$.
Both have positive first coordinates and downward slopes at most $\omega_c$, so their sum does too.
Applying this observation at each backward step, with $r_{m_{\mathrm{gap}}}=v_{m_{\mathrm{gap}}}$ as the base case, proves
\begin{equation}\label{eq:bending-ray}
 \omega(v_j)=\max\{\omega_c,\omega(r_j)\}
 \quad(0\le j\le m_{\mathrm{gap}}).
\end{equation}
Thus each $v_j$ is where the ray from $C$ through $r_j$ meets the arc, or the corner if this ray passes above it.
To prove that $v_j=\Pi_K(r_j)$, we will also need the lower bound on $y_j$ established after choosing $\tau$.
First, we show that the gradients turn toward the corner as we work backward:
\begin{equation}\label{eq:bending-monotone}
 q_{j-1}\le q_j\quad(1\le j\le m_{\mathrm{gap}}).
\end{equation}
For $h_j>0$, the definition of $\widehat r_j$ gives
\[
 \left(1-\frac{h_j}{\kappa}\right)
       (\widehat r_j+\alpha_jC-C)
 =(r_j-C)+\frac{h_j}{\kappa}
       \left(C+\frac{z^\circ}{\varrho}\right)
       +\frac{(1-1/\kappa)h_j}{2}C.
\]
The last two vectors have nonnegative coordinates.
Adding such a vector cannot increase $\max\{\omega_c,\omega(\cdot)\}$.
Indeed, if $a=(a_1,a_2)$ has nonnegative coordinates and $\omega(z+a)>\omega_c$, then $0<\beta-z_2-a_2\le\beta-z_2$ and $z_1+a_1-\beta\ge z_1-\beta$, so $\omega(z+a)\le\omega(z)$.
Since rescaling about $C$ leaves the slope unchanged, \eqref{eq:bending-backward-rule} and \eqref{eq:bending-ray} give
\[
 \omega(v_{j-1})\le\max\{\omega_c,\omega(r_j)\}=\omega(v_j).
\]
This proves \eqref{eq:bending-monotone}; zero steps give equality.

\begin{figure}[t]
\centering
{\color{black}
\begingroup
\definecolor{backpurple}{HTML}{70418D}
\definecolor{backorange}{HTML}{BD720C}
\definecolor{backgray}{HTML}{555E68}
\definecolor{backgreen}{HTML}{267B71}
\begin{tikzpicture}[x=3.1cm,y=3.1cm,>=Latex,font=\small,text=black,
  backaxis/.style={->,draw=black!60,line width=.5pt},
  backedge/.style={draw=backpurple,line width=.85pt},
  backprojection/.style={draw=backgreen,dashed,line width=.8pt,-{Latex[length=4pt,width=3pt]}},
  backguide/.style={draw=black!35,dashed,line width=.5pt}]
\coordinate (backO) at (0,0);
\coordinate (backCenter) at (.25,.25);
\coordinate (backV) at (.848511017749858,-.847398998373895);
\coordinate (backVcurrent) at (.379103967613579,-.993314990477647);
\coordinate (backCurrent) at (.404359343047925,-1.236532820634866);
\coordinate (backStar) at (.649484579187710,-1.196483617257013);
\coordinate (backShiftedCenter) at (.901468706171837,-.944499490272886);
\coordinate (backPrevious) at (1.504729811364154,-2.050608004506892);
\coordinate (backCorner) at (1.504729811364154,-1.196483617257013);
\coordinate (backC) at (1.474744871391589,0);
\coordinate (backEnd) at (.25,-1);
\coordinate (backRayEnd) at ($(backCenter)!1.055!(backPrevious)$);
\coordinate (backCurrentRayEnd) at ($(backCenter)!1.075!(backCurrent)$);

\fill[backgreen!13] (backO)--(backC)
  arc[start angle=-11.536959033,end angle=-90,radius=1.25]--cycle;
\draw[backgreen!65,line width=.55pt] (backO)--(backEnd);
\node[text=backgreen] at (.90,-.31) {$K$};
\node[anchor=west,text=backgreen,inner sep=0pt] at (1.69,-.61)
  {$\begin{aligned}v_j&=\Pi_K(r_j)\\v_{j-1}&=\Pi_K(r_{j-1})\end{aligned}$};

\node[anchor=west,align=left,text=backgray,font=\small] at (1.90,-1.20)
  {$\widehat r_j$ undoes the quadratic\\contraction in step $j$.};

\fill[backpurple!12] (backO)--(backCenter)--(backV)--cycle;
\fill[backpurple!12] (backStar)--(backShiftedCenter)--(backPrevious)--cycle;
\draw[backguide] (backCenter)--(backC)
  node[pos=.61,above=2pt,inner sep=1pt] {$R$};
\draw[backguide] (backCenter)--(backCurrentRayEnd);
\draw[backguide] (backCenter)--(backRayEnd);
\draw[backpurple,dashed,line width=.65pt] (backCenter)--(backV);
\draw[backpurple,dashed,line width=.65pt] (backShiftedCenter)--(backPrevious);
\draw[backguide] (backStar)--(backCorner)--(backPrevious);
\draw[backprojection] (backCurrent)--(backVcurrent);
\draw[backprojection] (backPrevious)--(backV);

\draw[backguide] (backStar)--(backStar |- 0,-2.30);
\draw[backguide] (backPrevious)--(backPrevious |- 0,-2.30);
\draw[draw=black!55,decorate,decoration={brace,amplitude=4pt}]
  (backPrevious |- 0,-2.30)--(backStar |- 0,-2.30)
  node[midway,below=7pt] {$\alpha_jp_{j-1}$};
\draw[backguide] (backCorner)--($(backCorner)+(.15,0)$);
\draw[backguide] (backPrevious)--($(backPrevious)+(.15,0)$);
\draw[draw=black!55,decorate,decoration={brace,amplitude=4pt}]
  ($(backCorner)+(.15,0)$)--($(backPrevious)+(.15,0)$)
  node[midway,right=7pt] {$\alpha_jq_{j-1}$};
\draw[draw=black!45,line width=.45pt]
  ($(backCorner)+(-.065,0)$)--($(backCorner)+(-.065,-.065)$)--($(backCorner)+(0,-.065)$);

\draw[backedge] (backO)--(backCenter);
\draw[backedge] (backStar)--(backShiftedCenter);
\draw[backedge] (backO)--(backV);
\draw[backgray,line width=1.05pt,-{Latex[length=4pt,width=3.7pt]}]
  (backCurrent)--(backStar);
\draw[backorange,line width=1.15pt,-{Latex[length=5pt,width=4pt]}]
  (backStar)--(backPrevious)
  node[pos=.53,sloped,below=4pt,inner sep=1pt,text=backorange] {$\alpha_jv_{j-1}$};
\draw[backpurple,line width=.95pt]
  plot[domain=11.536959033:90,samples=90,variable=\t]
  ({.25+1.25*cos(\t)},{.25-1.25*sin(\t)});
\begin{scope}[shift={(backCenter)}]
  \draw[backpurple,line width=.55pt] (225:.13)
    arc[start angle=225,end angle=298.607632658,radius=.13];
\end{scope}
\begin{scope}[shift={(backShiftedCenter)}]
  \draw[backpurple,line width=.55pt] (225:.131031746)
    arc[start angle=225,end angle=298.607632658,radius=.131031746];
\end{scope}

\draw[backaxis] (-.13,0)--(2.0,0) node[right] {$u$};
\draw[backaxis] (0,-2.46)--(0,.56) node[above] {$-y$};
\node[below left=1pt] at (backO) {$0$};
\fill[black!65] (backCenter) circle (1.5pt)
  node[above right=2pt,inner sep=1pt] {$(\beta,\beta)$};
\fill[backpurple] (backC) circle (1.3pt)
  node[above right=2pt,inner sep=1pt] {$(c_\beta,0)$};
\fill[backpurple] (backEnd) circle (1.3pt)
  node[below=4pt,xshift=-7pt,inner sep=.8pt] {$(\beta,-1)$};
\fill[backpurple] (backV) circle (1.6pt)
  node[right=-6pt,yshift=10pt,inner sep=.8pt] {$v_{j-1}$};
\fill[backpurple] (backVcurrent) circle (1.5pt)
  node[below right=3pt,xshift=1pt,inner sep=.8pt] {$v_j$};
\filldraw[fill=white,draw=backpurple,line width=.65pt]
  (backShiftedCenter) circle (1.8pt)
  node[right=6pt,yshift=-1pt,inner sep=1pt,text=backpurple] {$\widehat r_j+\alpha_j(\beta,\beta)$};
\fill[backgray] (backCurrent) circle (1.5pt)
  node[below left=3pt,inner sep=1pt] {$r_j$};
\fill[backgray] (backStar) circle (1.5pt)
  node[anchor=north east,xshift=0pt,yshift=-3pt,inner sep=.8pt] {$\widehat r_j$};
\fill[backpurple] (backPrevious) circle (1.6pt)
  node[below right=3pt,inner sep=1pt] {$r_{j-1}$};
\end{tikzpicture}
\endgroup}
\caption{A backward step with $q_{j-1}>0$. Pale green denotes $K$. Green dashed arrows show projection onto $K$, giving the scaled gradients $v_j=\Pi_K(r_j)$ and $v_{j-1}=\Pi_K(r_{j-1})$.}
\label{fig:bending-backward}
\end{figure}

\paragraph{Choosing the threshold.}
The second coordinate of \eqref{eq:bending-gap-update} gives
\[
 \tau-y_j=\left(1-\frac{h_j}{\kappa}\right)(\tau-y_{j-1})
       +\frac{h_j}{\kappa}\frac{\kappa-1}{2}q_{j-1}.
\]
Expanding over the gap and using $y_{m_{\mathrm{gap}}}=1$ gives $\tau-1=\chi_i(\tau-y_0)+Q$.
Thus $y_0=\tau$ if and only if $\tau=1+Q$.

As the trial value $\tau$ increases, both coordinates of $z^\circ/\varrho=(\ell_i/\varrho+\tau/\beta,\tau)$ increase.
Backward induction from the common terminal point shows that each $r_j$ is coordinatewise nondecreasing in $\tau$.
Indeed, if $r_j$ is, then so is $\widehat r_j$.
The slope comparison above shows that $q_{j-1}$ is nonincreasing in $\tau$, so $v_{j-1}=(p(q_{j-1}),-q_{j-1})$ and $r_{j-1}=\widehat r_j+\alpha_jv_{j-1}$ are coordinatewise nondecreasing.
Since $Q$ is a nonnegative combination of the $q_{j-1}$, it is continuous and nonincreasing in $\tau$, with values in $[0,(1-1/\kappa)s_i/2]$.
On the interval $[1,1+(1-1/\kappa)s_i/2]$, $\tau-1-Q$ is continuous and strictly increasing, nonpositive at the left endpoint and nonnegative at the right endpoint.
It therefore has exactly one zero; fix this value of $\tau$.
The recursion is jointly continuous in $\tau$ and $\varrho>0$, and all these roots lie in the same compact interval.
When $\varrho$ converges to a positive value, any limit point of $\tau$ solves the limiting equation, whose solution is unique.
Thus $\tau$, the points, and $P$ are continuous in $\varrho$.

The second-coordinate recurrence is a convex combination of $\tau-y_{j-1}$ and $(\kappa-1)q_{j-1}/2$.
Starting from $\tau-y_0=0$, induction using \eqref{eq:bending-monotone} gives
\[
 \tau-y_{j-1}\le\tau-y_j
 \le\frac{\kappa-1}{2}q_{j-1}
 \le\frac{\kappa-1}{2}q_j.
\]
Since $y_{m_{\mathrm{gap}}}=1$, this proves $1\le y_j\le\tau$.

\paragraph{The projection identity.}
Since $\|C\|=\sqrt2\beta<R$, the disk centered at $C$ with radius $R$ contains $0$ and hence $K$.
If $0<q_j<1$, then \eqref{eq:bending-ray} puts $r_j$ on the ray from $C$ through $v_j$, and $y_j\ge1\ge q_j$ puts $r_j$ at or beyond $v_j$.
Thus $r_j-v_j$ is an outward normal to $K$ at $v_j$.
If $q_j=0$, then \eqref{eq:bending-ray} gives $y_j\le\beta(u_j-c_\beta)/(c_\beta-\beta)$, and $y_j\ge1$ forces $u_j>c_\beta$.
So $r_j-(c_\beta,0)$ is a nonnegative combination of the outward normals $(c_\beta-\beta,-\beta)$ and $(0,1)$.
If $q_j=1$, only zero steps have been undone, so $r_j=v_j=(\beta,-1)$.
The projection criterion \eqref{eq:local-normal-certificate} gives $\Pi_K(r_j)=v_j$ in every case.

\paragraph{Bounding $\tau$.}
Since $p_j\le c_\beta$, \eqref{eq:bending-weighted-sums} gives the bound on $P$.
The bound $\tau\le1+(1-1/\kappa)s_i/2$ was proved above, so it remains to bound $\tau$ in terms of $\beta$ alone.
Set $\lambda_j:=(y_j+\beta)/(q_j+\beta)$.
When $q_j>0$, \eqref{eq:bending-ray} gives $r_j-C=\lambda_j(v_j-C)$, so $\lambda_j=\|r_j-C\|/R$.
Write $[r,v]:=r_1v_2-r_2v_1$ for the two-dimensional determinant.
For $v=(p,-q)$ with $p,q\ge0$, this determinant does not change when a multiple of $v$ is added to $r$, and does not increase when a vector with nonnegative coordinates is added to $r$.
Since $y_j\le\tau$, the vector $r_j+z^\circ/\varrho$, and hence $\widehat r_j-r_j$, has nonnegative coordinates.
Therefore
\[
 [r_{j-1},v_{j-1}]=[\widehat r_j,v_{j-1}]\le[r_j,v_{j-1}].
\]
If $q_{j-1}>0$, then also $q_j>0$.
Substituting $r_{j-1}=C+\lambda_{j-1}(v_{j-1}-C)$ and $r_j=C+\lambda_j(v_j-C)$ gives
\[
 (1-\lambda_{j-1})[C,v_{j-1}]
 \le (1-\lambda_j)[C,v_{j-1}]+\lambda_j[v_j,v_{j-1}].
\]
Using $[C,v]=-\beta(p+q)$ and $[v_j,v_{j-1}]=p_{j-1}q_j-p_jq_{j-1}$, we obtain
\begin{equation}\label{eq:bending-height-step}
 \begin{aligned}
 \lambda_{j-1}
 &\le\lambda_j\left(1+
 \frac{p_{j-1}(q_j-q_{j-1})+q_{j-1}(p_{j-1}-p_j)}
      {\beta(p_{j-1}+q_{j-1})}\right)\\
 &\le\lambda_j\exp\!\left(
 \frac{\max\{q_j-q_{j-1},p_{j-1}-p_j\}}{\beta}\right).
 \end{aligned}
\end{equation}
When $q_{j-1}=0$, the first inequality follows from $y_{j-1}\le y_j$, because
\[
 \lambda_{j-1}=\frac{y_{j-1}+\beta}{\beta}
 \le\frac{y_j+\beta}{\beta}
 =\lambda_j\left(1+\frac{q_j}{\beta}\right).
\]
The second bounds a weighted average by its larger entry and uses $1+x\le e^x$.

Since $q_j$ is nondecreasing and $p(q)+q$ is concave, the sequence $p_j+q_j$ first increases and then decreases.
Moreover, $q_j-q_{j-1}\ge p_{j-1}-p_j$ exactly when $p_j+q_j\ge p_{j-1}+q_{j-1}$.
Thus the larger coordinate change is $q_j-q_{j-1}$ before the maximum and $p_{j-1}-p_j$ afterward.
These changes telescope:
\[
 \sum_{j=1}^{m_{\mathrm{gap}}}
 \max\{q_j-q_{j-1},p_{j-1}-p_j\}
 =\max_{0\le j\le m_{\mathrm{gap}}}(p_j+q_j)-\beta-q_0
 \le\sqrt2R-\beta-q_0.
\]
The last inequality follows from $(p_j-\beta)^2+(q_j+\beta)^2=R^2$.
Multiplying \eqref{eq:bending-height-step} and using $\lambda_{m_{\mathrm{gap}}}=1$ gives
\[
 y_0+\beta
 \le(q_0+\beta)\exp\!\left(\frac{\sqrt2R-\beta-q_0}{\beta}\right)
 \le\beta\exp\!\left(\frac{\sqrt2R-\beta}{\beta}\right),
\]
where the last inequality uses $1+q_0/\beta\le e^{q_0/\beta}$.
Since $y_0=\tau$, this proves the last bound in \cref{lem:bending-gap}.
\end{proof}

\subsectionnotoc{Projection properties}
\label{app:local-projection}

The following lemma gives sufficient conditions for $\Phi=0$ and for $\partial_Y\Phi=0$ in the regions used above.

\begin{lemma}[Projection properties]\label{lem:local-projection}
Let $K\subset\R^2$ be compact and convex, with $0\in K$, and let $z^\circ=(z_1,z_2)\in\R^2$.
The function
\[
 \Phi(X,Y)=(\env_1\sigma_K)((X,Y)-z^\circ)
 =\max_{g\in K}\left\{\langle g,(X,Y)-z^\circ\rangle-\tfrac12\|g\|^2\right\}
\]
is nonnegative, convex, and $1$-smooth, with $\nabla\Phi(X,Y)=\Pi_K((X,Y)-z^\circ)$.
Suppose $K\subset\{(p,-q):p,q\ge0\}$. Then:
\begin{enumerate}[label=(\roman*)]
 \item $\partial_Y\Phi(X,Y)\le0$ for all $(X,Y)$.
 \item If, for some $\theta>0$ and $\ell_i\ge0$, every $(p,-q)\in K$ satisfies $q\le p/\theta$, and $z_2\ge0$, $z_1\ge\ell_i+z_2/\theta$, then
 \[
  \Phi(X,Y)=0\qquad\text{whenever }X\le\ell_i,\ Y\ge0.
 \]
 \item If $(p,0)\in K$ whenever $(p,-q)\in K$, then
 \[
  \partial_Y\Phi(X,Y)=0\qquad\text{whenever }Y\ge z_2.
 \]
\end{enumerate}
\end{lemma}

\begin{proof}
The representation and gradient formula follow from \cref{fact:local-moreau}.
Since the gradient belongs to $K$, its second coordinate is nonpositive, proving (i).
For (ii), if $X\le\ell_i$ and $Y\ge0$, then every $g=(p,-q)\in K$ satisfies
\[
 \langle g,(X,Y)-z^\circ\rangle
 =p(X-z_1)-q(Y-z_2)
 \le p(\ell_i-z_1)+qz_2
 \le z_2(q-p/\theta)\le0.
\]
Thus the maximum defining $\Phi$ is zero, attained at $g=0$.
For (iii), suppose $Y\ge z_2$ and $g=(p,-q)\in K$ has $q>0$.
Replacing $g$ by $g'=(p,0)\in K$ reduces the squared distance to $(X,Y)-z^\circ$ by
\[
 \|(X,Y)-z^\circ-g\|^2-\|(X,Y)-z^\circ-g'\|^2
 =2q(Y-z_2)+q^2>0.
\]
The projection therefore has second coordinate zero, proving (iii).
\end{proof}

\section{\texorpdfstring{Proof of \Cref{lem:fixed-c-block}}{Proof of \getrefnumber{lem:fixed-c-block}}}
\label{app:sequence}

We prove \cref{lem:fixed-c-block} here.
Throughout, fix $1\le c\le2$ and $a\ge8$, and let $p=p(c)$ be as in \eqref{eq:checkpoint-exponents}.

\subsectionnotoc{Bounding the minimum checkpoint cost}
\label{app:sequence-auxiliary}

In the score \eqref{eq:carried-score}, only the part of a selected stepsize above $8$ appears in the numerator. We therefore separate each stepsize into a fixed part $g_j=\min\{\tilde{h}_j,8\}$ and a selectable part $B_j=(\tilde{h}_j-8)_+$. The entry $B_j$ is positive exactly when index $j$ can be selected.
For the proof, allow arbitrary nonnegative vectors $g=(g_1,\ldots,g_r)$ and $B=(B_1,\ldots,B_r)$ and arrange their entries as
\[
 g_1,B_1,g_2,B_2,\ldots,g_r,B_r.
\]
Only entries $B_j>0$ may be selected as checkpoints. The entries $g_j$ are included in gap sums and are never selected.
For a selected $B_j$, let $\widetilde s_j$ be the sum of all entries after the previous selected $B$ and before $B_j$.
For the first selected position, the sum starts at the beginning of the list.
Let $s_{\rm tail}$ be the sum of the entries after the final selected $B$.
Define the checkpoint score $P(T)$ and \emph{cost} $Q_\Lambda(T;g,B)$ by
\[
 P(T)=\prod_{j\in T}\frac{B_j}{c(a+\widetilde s_j)},\qquad
 Q_\Lambda(T;g,B)=\frac{c(\Lambda+s_{\rm tail})}{P(T)},\qquad \Lambda\ge a.
\]
Here the \emph{terminal constant} $\Lambda$ appears in the factor $c(\Lambda+s_{\rm tail})$ for the gap after the last checkpoint. For the empty set, $P(\varnothing)=1$ and $Q_\Lambda(\varnothing;g,B)=c[\Lambda+\sum_j(g_j+B_j)]$. Write
\[
 V_\Lambda(g,B):=\min_{T\text{ admissible}}Q_\Lambda(T;g,B),\qquad
 U_\Lambda(g):=\sup_{B\in[0,\infty)^r}V_\Lambda(g,B).
\]
The entries $g$ remain fixed in this supremum.
Taking the supremum over $B$ lets us bound the minimum cost using only the fixed entries $g$.
A checkpoint set attaining $V_\Lambda(g,B)$ is called \emph{tight}.
We abbreviate $Q_\Lambda(T;g,B)$ to $Q_\Lambda(T)$ when $g,B$ are fixed.

To see why the cost includes a terminal factor, fix an endpoint checkpoint $t$.
The score contributed by checkpoints before and including $t$ is $B_t/Q_{a+g_t}$, with the cost evaluated on the positions before $t$.
Thus optimizing additional checkpoints before this endpoint amounts to minimizing that cost.

For the original stepsizes, we add the incoming sum $w$ to the first entry, using the list $g_1+w,B_1,g_2,B_2,\ldots,g_r,B_r$.
Each selected step contributes its fixed part $8$ to the preceding sum, so $\widetilde s_{t_i}=s_i+8$ in the notation of \eqref{eq:carried-gap-definitions}.
This gives $P(T)=P_w(T;\tilde{H})$.

As in~\citep[Lemma~3.3]{ye2026silver}, we decompose the minimum cost at a selected checkpoint. The entry $g_t$ immediately before $B_t$ stays on the left, so the constant in the left subproblem is $a+g_t$. This is why we allow any $\Lambda\ge a$.
\begin{lemma}\label{lem:fixed-c-value-decomposition}
For every $g,B\in[0,\infty)^r$, $r\ge1$, and $\Lambda\ge a$,
\begin{equation}\label{eq:fixed-c-value-decomposition}
 V_\Lambda(g,B)=\min\Biggl\{
 c\left(\Lambda+\sum_{j=1}^r(g_j+B_j)\right),\min_{\substack{1\le t\le r\\B_t>0}}
 \frac{V_{a+g_t}(g_{1:t-1},B_{1:t-1})
       V_\Lambda(g_{t+1:r},B_{t+1:r})}{B_t}
 \Biggr\}.
\end{equation}
Here $g_{u:v}:=(g_u,\ldots,g_v)$ and likewise for $B_{u:v}$. For empty vectors, $V_\mu(\varnothing,\varnothing)=U_\mu(\varnothing)=c\mu$. The inner minimum is $+\infty$ if no $B_t$ is positive.
\end{lemma}

\begin{proof}
For a checkpoint set $T$ containing $t$, let $T_L,T_R$ denote its selected indices before and after $t$, reindexed within the two shorter lists. Then
\[
 Q_\Lambda(T;g,B)=
 \frac{Q_{a+g_t}(T_L;g_{1:t-1},B_{1:t-1})
       Q_\Lambda(T_R;g_{t+1:r},B_{t+1:r})}{B_t}.
\]
Minimizing over $T_L,T_R$ and $t$ accounts for every nonempty checkpoint set. Including the empty set gives \eqref{eq:fixed-c-value-decomposition}.
\end{proof}

To bound $U_\Lambda(g)$ using this recursion, we need a maximizing $B$ at which both the empty set and a singleton are tight. We first show that $\log Q_\Lambda(T)$ is submodular in the checkpoint set $T$, following~\citep[Lemma~B.1]{ye2026silver}.
\begin{lemma}[Log-submodularity]\label{lem:fixed-c-log-submodular}
For any two admissible checkpoint sets $S,T$,
\begin{equation}\label{eq:fixed-c-log-submodular}
 Q_\Lambda(S)Q_\Lambda(T)
 \ge Q_\Lambda(S\cap T)Q_\Lambda(S\cup T).
\end{equation}
If $S$ and $T$ are disjoint and both nonempty, the inequality is strict.
Consequently, tight sets are closed under unions and intersections, and two nonempty tight sets cannot be disjoint.
\end{lemma}

\begin{proof}
\emph{Inserting a checkpoint.}
Insert a checkpoint of size $B>0$ into a gap. Let $\ell\ge0$ be the sum of entries before it in this gap, and let $K$ be the sum after it plus the constant $a$ if the gap ends at an existing checkpoint, or $\Lambda$ if it is the final gap. Thus $K\ge a$.
Before insertion, this gap contributes $c(\ell+B+K)$ to $Q_\Lambda$. After insertion, it contributes $c(a+\ell)cK/B$, so the cost is multiplied by
\[
 \frac{c(a+\ell)K}{B(\ell+B+K)}.
\]
Its logarithmic derivatives with respect to $\ell$ and $K$ are $1/(a+\ell)-1/(\ell+B+K)$ and $1/K-1/(\ell+B+K)$, respectively. Both are positive because $B>0$ and $K\ge a$, so the ratio is strictly increasing in both variables.

\smallskip
\noindent\emph{Comparing checkpoint sets.}
Insert the checkpoints in $S\setminus T$, one at a time in the same order, into $S\cap T$ and into $T$. At each insertion, the set built from $T$ contains the set built from $S\cap T$. Its gap containing the new checkpoint therefore starts no earlier and ends no later, so its values of $\ell$ and $K$ are no larger. If this gap ends at a checkpoint while the corresponding gap for the smaller set is final, the constant in $K$ also decreases from $\Lambda$ to $a\le\Lambda$. Multiplying the insertion ratios gives
\[
 \frac{Q_\Lambda(S)}{Q_\Lambda(S\cap T)}
 \ge
 \frac{Q_\Lambda(S\cup T)}{Q_\Lambda(T)},
\]
which proves \eqref{eq:fixed-c-log-submodular}. If $S$ and $T$ are disjoint and nonempty, the first insertion compares the empty set with $T$. Every selected entry in $T$ is positive, so at least one of $\ell,K$ is strictly smaller when inserting into $T$. The first comparison, and hence the final inequality, is strict.

Finally, if $S,T$ are tight and $V=V_\Lambda(g,B)$, then
\[
 V^2\ge Q_\Lambda(S\cap T)Q_\Lambda(S\cup T)\ge V^2.
\]
Each factor in the middle is at least $V>0$, so both equal $V$. Thus $S\cap T$ and $S\cup T$ are tight. If $S,T$ were also nonempty and disjoint, the first inequality would be strict, a contradiction.
\end{proof}

We now prove the analogue of~\citep[Lemma~3.4]{ye2026silver}, keeping the fixed entries $g$ unchanged.
\begin{lemma}\label{lem:fixed-c-attainment}
For fixed $g\in[0,\infty)^r$, $r\ge1$, and $\Lambda\ge a$, the supremum defining $U_\Lambda(g)$ is finite and attained. At every maximizing $B$, the empty set and a singleton are tight.
\end{lemma}

\begin{proof}\leavevmode
\paragraph{Step 1. Finiteness of $U_\Lambda(g)$.}
We induct on the length $r$ of $g$, simultaneously for all $g\ge0$ and $\Lambda\ge a$.
The case $r=0$ is $U_\Lambda(\varnothing)=c\Lambda$.
For an arbitrary $B$, let $B_{\max}=\max_jB_j$.
If $B_{\max}\le1$, the empty set gives
\[
 V_\Lambda(g,B)\le c\left(\Lambda+\sum_jg_j+r\right).
\]
Otherwise, choose $j$ with $B_j=B_{\max}$ and split at $j$ in \eqref{eq:fixed-c-value-decomposition}.
Using the empty set on the left gives
\[
 V_\Lambda(g,B)
 \le\frac{c\left(a+g_j+\sum_{i<j}(g_i+B_i)\right)}{B_j}
       U_\Lambda(g_{j+1:r})
 \le c\left(r+a+\sum_i g_i\right)U_\Lambda(g_{j+1:r}).
\]
The finitely many possible suffixes have finite $U_\Lambda$ by induction, so these bounds are uniform over $B$.

\paragraph{Step 2. Existence of a maximizer.}
For fixed $g$, the function $V_\Lambda(g,B)$ is continuous at every finite $B\ge0$.
Indeed, costs selecting an entry that tends to zero diverge, since all numerator factors stay positive.
Such sets cannot minimize the cost nearby because the empty-set cost stays bounded.
The remaining costs are continuous, and their finite minimum is therefore continuous.

Choose $B^{(k)}$ with $V_\Lambda(g,B^{(k)})\to U_\Lambda(g)$.
If this sequence were unbounded, pass to a subsequence along which each coordinate has a limit in $[0,\infty]$ and at least one tends to infinity.
Let $j$ be the first such coordinate, so the preceding entries stay bounded.
Splitting at $j$ and using the empty set on the left yields
\[
 V_\Lambda(g,B^{(k)})
 \le\frac{c\left(a+g_j+\sum_{i<j}(g_i+B_i^{(k)})\right)}{B_j^{(k)}}
       U_\Lambda(g_{j+1:r})\longrightarrow0.
\]
This contradicts $U_\Lambda(g)\ge V_\Lambda(g,0)=c(\Lambda+\sum_jg_j)>0$.
Thus a subsequence converges to a finite $B$, and continuity gives $V_\Lambda(g,B)=U_\Lambda(g)$.

\paragraph{Step 3. Tightness of the empty set.}
Fix a maximizing $B$ and write $V=V_\Lambda(g,B)$.
If a perturbation strictly increased every tight cost, all other costs would remain above $V$ for sufficiently small perturbations, by Step~2.
Since there are finitely many checkpoint sets, their new minimum would exceed $U_\Lambda(g)$, a contradiction.

Suppose that the empty set were not tight.
By \cref{lem:fixed-c-log-submodular}, the intersection of all tight sets is tight and hence nonempty.
Choose $j$ in this intersection and slightly decrease $B_j$, keeping it positive.
Every tight set selects $j$, so its cost depends on $B_j$ only through $1/B_j$ and strictly increases.
The observation above gives a contradiction.
Thus the empty set is tight.

\paragraph{Step 4. Existence of a tight singleton.}
There must also be a nonempty tight set.
Otherwise, increasing any $B_j$ slightly would increase the only tight cost, that of the empty set, contradicting Step~3.

Let $T_*$ be the intersection of all nonempty tight sets.
By \cref{lem:fixed-c-log-submodular}, $T_*$ is nonempty and tight.
Suppose it contains distinct indices $i,j$, and put $u=B_i$, $v=B_j$, with $u\le v$.
If $u<v$, choose $1<\theta<v/u$ and replace $(u,v)$ by $(u-t,v+\theta t)$ for sufficiently small $t>0$.
If $u=v$, replace them by $(u-t,u+t+t^2/(2u))$.
In both cases the sum increases and the product decreases.
The empty-set cost therefore increases.
Every nonempty tight set selects both entries, so its cost depends on them only through $1/(uv)$ and also increases.
This again contradicts Step~3.
Hence $T_*$ is a tight singleton.\par
\end{proof}

Splitting at this singleton gives two shorter lists whose empty sets are both tight. As in~\citep[Lemma~3.5 and Appendix~B.3]{ye2026silver}, we apply induction to these lists to bound $U_\Lambda(g)$. The induction step uses the following scalar inequality, proved as in~\citep[Lemma~B.2]{ye2026silver}. The case $c=2$ also follows by rescaling~\citep[Lemma~C.4]{jung2026stronger}.

\begin{lemma}\label{lem:fixed-c-split}
If $x,y\ge a$, $B>0$, and
\[
 W=x+y+B-a,\qquad WB=cxy,
\]
then $W^{1/p}\le x^{1/p}+y^{1/p}$.
\end{lemma}
\begin{proof}
Set $\overline W=(x^{1/p}+y^{1/p})^p$. Our goal is to show $W\le\overline W$. Define $\psi(X)=X^2-(x+y)X-cxy$. Eliminating $B$ gives
\[
 \psi(W)=-aW<0.
\]
Since $\psi$ has exactly one positive root and $W,\overline W>0$, it suffices to prove $\psi(\overline W)\ge0$. Set $z=(x/\overline W)^{1/p}$, so $1-z=(y/\overline W)^{1/p}$ and $z\in(0,1)$. Then
\[
 \frac{\psi(\overline W)}{\overline W^2}=1-f(z),\qquad
 f(z):=z^p+(1-z)^p+c[z(1-z)]^p.
\]
Thus it remains to show $f(z)\le1$ on $[0,1]$.
The choice $p=\log_2(1+\sqrt{1+c})$ gives $c=2^{2p}-2^{p+1}$, so $f(0)=f(1/2)=1$. On $(0,1/2)$,
\[
 f'(z)=p[z(1-z)]^{p-1}g(z),\qquad
 g(z)=(1-z)^{1-p}-z^{1-p}+c(1-2z).
\]
Moreover,
\[
 g''(z)=p(p-1)[(1-z)^{-p-1}-z^{-p-1}]<0,\qquad
 g(0+)=-\infty,\qquad g(1/2)=0.
\]
Concavity implies that $g$ either stays negative or changes sign once from negative to positive. Thus $f$ either decreases or first decreases and then increases, so its maximum on $[0,1/2]$ is at an endpoint. By symmetry, $f(z)\le1$ on $[0,1]$, so $\psi(\overline W)\ge0$ and $W\le\overline W$, as required.
\end{proof}

We now use this inequality to bound $U_\Lambda(g)$ by induction.
\begin{lemma}\label{lem:fixed-c-terminal}
For every $g\in[0,\infty)^r$ and $\Lambda\ge a$,
\[
 \left[\frac{U_\Lambda(g)}c\right]^{1/p}
 \le\Lambda^{1/p}+\sum_{j=1}^r(a+g_j)^{1/p}.
\]
\end{lemma}

\begin{proof}
We induct on $r$, simultaneously for all $g\ge0$ and $\Lambda\ge a$.
The case $r=0$ is $U_\Lambda(\varnothing)=c\Lambda$.
For $r\ge1$, choose a maximizing $B$.
By \cref{lem:fixed-c-attainment}, the empty set and a singleton $\{t_*\}$ are tight.
Set
\[
 W:=U_\Lambda(g)/c,\qquad
 x:=a+g_{t_*}+\sum_{j<t_*}(g_j+B_j),\qquad
 y:=\Lambda+\sum_{j>t_*}(g_j+B_j).
\]
Then $x,y\ge a$ and $B_{t_*}>0$.
Tightness of the empty set and singleton gives
\[
 W=x+y+B_{t_*}-a,\qquad WB_{t_*}=cxy.
\]
Thus \cref{lem:fixed-c-split} gives $W^{1/p}\le x^{1/p}+y^{1/p}$.

The recursive decomposition \eqref{eq:fixed-c-value-decomposition} and the empty-set costs on the two sides give
\[
 cW\le
 \frac{V_{a+g_{t_*}}(g_{1:t_*-1},B_{1:t_*-1})
       V_\Lambda(g_{t_*+1:r},B_{t_*+1:r})}{B_{t_*}}
 \le\frac{c^2xy}{B_{t_*}}=cW.
\]
Both numerator factors therefore equal their empty-set bounds $cx$ and $cy$, respectively.
Applying the induction hypothesis to the two shorter lists yields
\[
 \begin{aligned}
 x^{1/p}&\le(a+g_{t_*})^{1/p}+\sum_{j<t_*}(a+g_j)^{1/p},\\
 y^{1/p}&\le\Lambda^{1/p}+\sum_{j>t_*}(a+g_j)^{1/p}.
 \end{aligned}
\]
Adding these bounds proves the claim.\par
\end{proof}

\subsectionnotoc{Bounding the sum of entries when the maximum checkpoint score is one}
\label{app:fixed-c-mass}

Since $V_\Lambda(g,B)\le U_\Lambda(g)$, \cref{lem:fixed-c-terminal} bounds the minimum cost for every $B$. We use it to bound $\sum_j(g_j+B_j)$ when $\max_TP(T)=1$. For the original stepsizes, this gives part~(i) of \cref{lem:fixed-c-block} and the bound on $s_{\rm tail}$ in part~(ii). First we bound the number of entries $B_j$ above each level $q>0$, then integrate this bound to obtain their sum.
For this subsection, write $M=\sum_{j=1}^r(a+g_j)^{1/p}$ for the sum over the fixed entries in \cref{lem:fixed-c-terminal}.
\begin{lemma}[Level-set bound]\label{lem:fixed-c-level-set}
If $\max_TP(T)=1$, then for every \(q>0\),
\[
    \#\{j:B_j>q\}\le c^{1/p} Mq^{-1/p}.
\]
\end{lemma}

\begin{proof}
Fix $q>0$ and list the positions with $B_j>q$ as $j_1<\cdots<j_k$. The claim is immediate if $k=0$. Otherwise set $j_0=0$ and
\[
 M_i=\sum_{t=j_{i-1}+1}^{j_i}(a+g_t)^{1/p}
 \quad(1\le i\le k),\qquad
 \sum_{i=1}^kM_i\le M.
\]
Select every $j_i$. Between $j_{i-1}$ and $j_i$, choose additional checkpoints to minimize the cost with terminal constant $a+g_{j_i}$. By \cref{lem:fixed-c-terminal}, this minimum is at most $cM_i^p$, so this part of the score is at least $B_{j_i}/(cM_i^p)$. Multiplying these factors and using the arithmetic--geometric mean inequality gives
\[
 1=\max_TP(T)
 \ge\prod_{i=1}^k\frac{q}{cM_i^p}
 \ge\left(\frac qc\right)^k\left(\frac{k}{M}\right)^{pk}.
\]
Taking $k$th roots yields $k\le c^{1/p} Mq^{-1/p}$, as required.
\end{proof}

We now convert this counting estimate into a bound on the sum of entries.
\begin{lemma}\label{lem:fixed-c-mass}
If $\max_TP(T)=1$, then
\[
 \sum_j(g_j+B_j)\le
 \left(1+\frac{cp}{p-1}\right)M^p.
\]
\end{lemma}

\begin{proof}
If the list is empty, both sides of the claimed bound are zero. Otherwise, \cref{lem:fixed-c-level-set} gives
\[
    N(q):=\#\{j:B_j>q\}
      \le c^{1/p} Mq^{-1/p}.
\]
If $\max_jB_j>0$, take $0<q<\max_jB_j$. Then $N(q)\ge1$, and hence
\[
    1\le c^{1/p} Mq^{-1/p}.
\]
Letting \(q\uparrow\max_jB_j\) gives
\[
    \max_jB_j\le cM^p.
\]
This bound also holds when all $B_j=0$. Therefore,
\[
\sum_jB_j=\int_0^{cM^p}N(q)\,\d q \le c^{1/p} M\int_0^{cM^p}q^{-1/p}\,\d q=\frac{cp}{p-1}M^p.
\]

It remains to control the sum of $g_j$. Since \(p>1\), subadditivity gives
\[
    \left(\sum_jg_j\right)^{1/p}
      \le \sum_jg_j^{1/p}
      \le \sum_j(a+g_j)^{1/p}
      =M.
\]
Combining the two estimates gives
\[
    \sum_j(g_j+B_j)
      \le
    M^p+\frac{cp}{p-1}M^p
      =
    \left(1+\frac{cp}{p-1}\right)M^p.
\]
\end{proof}

\subsectionnotoc{Deriving the checkpoint bounds}
\label{app:fixed-c-block}

We now return to the original stepsizes and the checkpoint score $P_w(T;\tilde{H})$ to prove \cref{lem:fixed-c-block}.

\begin{proof}[Proof of \cref{lem:fixed-c-block}]
Let $\tilde{h}_1,\ldots,\tilde{h}_m$ be the stepsizes in the current block, padded with zeros if necessary, and let $w$ be the sum of stepsizes in its incoming tail.
Take $g_j=\min\{\tilde{h}_j,8\}$ and $B_j=(\tilde{h}_j-8)_+$. As explained in \cref{app:sequence-auxiliary}, the list $g_1+w,B_1,g_2,B_2,\ldots,g_m,B_m$ has score $P(T)=P_w(T;\tilde{H})$.

For this list, write
\[
    M=(a+g_1+w)^{1/p}+\sum_{j=2}^m(a+g_j)^{1/p}.
\]
Since \(p>1\), subadditivity of \(t^{1/p}\) gives
\[
    (a+g_1+w)^{1/p}
      \le w^{1/p}+(a+g_1)^{1/p}.
\]
Since \(g_j\le8\) for every \(j\), it follows that
\[
    M\le w^{1/p}+m(a+8)^{1/p}\le w^{1/p}+2^{1/p} m a^{1/p}.
\]
Using $(u+v)^p\le2^{p-1}(u^p+v^p)$, we obtain
\begin{equation}
    M^p\le C\bigl(w+a m^p\bigr),
\label{eq:block-M-bound}
\end{equation}
where \(C\) is an absolute constant as the range of $p$ is bounded.

\paragraph{Proof of (i).}
If the empty set maximizes $P_w$, its maximum is one. Applying \cref{lem:fixed-c-mass} to this list and using \eqref{eq:block-M-bound} gives
\[
 w+\sum_{j=1}^m\tilde{h}_j
 \le\left(1+\frac{cp}{p-1}\right)M^p
 \le C_b(w+a m^p).
\]
The constant is uniform because $c$ and $\frac{p}{p-1}$ are bounded.

\paragraph{Proof of (ii).}
Now let \(T=\{t_1<\cdots<t_k\}\) be a nonempty maximizing checkpoint set, and put \(t_0=0\).
Group the positions from $t_{i-1}+1$ through $t_i$, including the selected endpoint, and define
\[
    M_1=(a+g_1+w)^{1/p}+\sum_{j=2}^{t_1}(a+g_j)^{1/p},\qquad
    M_i=\sum_{j=t_{i-1}+1}^{t_i}(a+g_j)^{1/p}\quad(2\le i\le k).
\]
We claim that
\begin{equation}\label{eq:block-di-Mi}
    a+s_i+8\le M_i^p.
\end{equation}
Keep $t_i$ fixed. Inserting checkpoints before it cannot increase the score, so the empty set minimizes the cost of this gap. Its cost is
\[
    c(a+s_i+8).
\]
If $i=1$ and $t_1=1$, the incoming sum $w$ is attached to the selected endpoint itself, and $a+8+w=M_1^p$ directly.
Otherwise, apply \cref{lem:fixed-c-terminal} to the unselected positions before $t_i$, with terminal constant $\Lambda=a+g_{t_i}=a+8$.
The term $\Lambda^{1/p}$ is exactly the endpoint contribution in $M_i$.
In the first gap, use $g_1+w$ as the first fixed entry.
This gives \eqref{eq:block-di-Mi} in every case.
The groups defining the \(M_i\)'s are disjoint, so
\[
    \sum_{i=1}^k M_i\le M.
\]
Since \(p\ge1\),
\[
 c\sum_i(a+s_i+8)\le c\sum_i M_i^p\le c\left(\sum_iM_i\right)^p
 \le cM^p\le C_b(w+a m^p).
\]
For the tail after $t_k$, the incoming sum is zero. Any checkpoint set in this tail with score greater than one could be appended to $T$ to increase its score. Hence the maximum score in the tail is one, and \cref{lem:fixed-c-mass} gives
\[
 s_{\rm tail}
 \le\left(1+\frac{cp}{p-1}\right)
 \left[\sum_{j=t_k+1}^m(a+g_j)^{1/p}\right]^p
 \le C_b a m^p.
\]
Here we used $g_j\le8$, $a\ge8$, and the fact that the tail contains at most $m$ steps. The same bound holds for an empty tail.

\paragraph{Proof of (iii).}
For consecutive checkpoints $t_i<t_{i+1}$, write $x=a+s_i+8$, $y=a+s_{i+1}+8$, and $B=b_i-8$.
Delete the checkpoint at \(t_i\), while retaining all the other checkpoints. The two gaps adjacent to \(t_i\) then merge, and the step at \(t_i\), whose full stepsize is \(\tilde{h}_{t_i}=B+8\), becomes part of the merged gap.
The denominator at the next checkpoint is therefore
\[
\begin{aligned}
    c\bigl(a+s_i+(B+8)+s_{i+1}+8\bigr)
      &=c(x+y+B-a).
\end{aligned}
\]
All other factors remain unchanged. Comparing the score before and after deletion gives
\[
 \frac{B}{cx}\frac{b_{i+1}-8}{cy}
 \ge\frac{b_{i+1}-8}{c(x+y+B-a)}.
\]
Cancelling the positive common factors gives
\begin{equation}\label{eq:block-deletion-basic}
    B(x+y+B-a)\ge cxy.
\end{equation}
We next prove
\begin{equation}\label{eq:block-xyB}
    \frac{xy}{B}\le2(x+y).
\end{equation}
If \(B\le x+y\), then \eqref{eq:block-deletion-basic} gives
\[
    \frac{xy}{B}
      \le
    \frac{x+y+B-a}{c}
      \le x+y+B
      \le2(x+y),
\]
where we used \(c\ge1\) and \(a\ge0\).
If \(B>x+y\), then by the arithmetic--geometric mean inequality,
\[
    \frac{xy}{B}
      <
    \frac{xy}{x+y}
      \le\frac{x+y}{4}
      \le2(x+y).
\]
Dividing \eqref{eq:block-xyB} by $xy$ gives part~(iii).

\paragraph{Proof of (iv).}
Deleting the final checkpoint cannot increase the score. Since
\[
 P_w(T;\tilde{H})=P_w(T\setminus\{t_k\};\tilde{H})
 \frac{b_k-8}{c(a+s_k+8)},
\]
the last factor is at least one, proving part~(iv).
\end{proof}

\section{\texorpdfstring{Proof of \Cref{lem:repair}}{Proof of the \getrefnumber{lem:repair}}}\label{app:repair}

We first prove the auxiliary lemma used to justify the second case of the repair rule in \cref{sec:checkpoint-repair}.

\begin{lemma}\label{lem:repair-pair}
Let $h_1,\ldots,h_r\ge0$, where $r\ge1$, and let
$S':=\sum_{t=1}^r h_t>0$. At least one of the following alternatives holds:
\begin{enumerate}
\item There exist $1\le i<j\le r$ such that
\[
 h_i,h_j\ge\frac{S'}{8r^2},
 \qquad
 \sum_{t=i+1}^{j-1}h_t\le\min\{h_i,h_j\}.
\]
\item There exists $1\le i\le r$ such that $h_i\ge3S'/4$.
\end{enumerate}
\end{lemma}

\medskip
\begin{proof}[\normalfont\bfseries Proof of \cref{lem:repair-pair}]
\mbox{}\par\nobreak
Suppose the first alternative fails and set $\tau=S'/(8r^2)$. Then
\begin{equation}\label{eq:separation}
\sum_{t=i+1}^{j-1}h_t>\min\{h_i,h_j\}
\qquad\text{whenever }i<j\text{ and }h_i,h_j\ge\tau.
\end{equation}
We will show that the sum of all terms other than the largest is at most $S'/4$.

\paragraph{Two inductive estimates.}
Consider a nonnegative sequence $a_1,\ldots,a_k$ satisfying \eqref{eq:separation} with $a_i$ in place of $h_i$. We prove the following two estimates together by induction on $k$:
\begin{enumerate}[label=(\roman*),leftmargin=*]
\item If $a_i<\sum_{t<i}a_t$ whenever $a_i\ge\tau$, then $\sum_{i=1}^k a_i\le2\tau k^2$.
\item If, in addition, $a_i<\sum_{t>i}a_t$ whenever $a_i\ge\tau$, then $\sum_{i=1}^k a_i\le\tau k^2$.
\end{enumerate}
The stronger estimate (ii) is needed to apply the induction to the terms before the largest element.

Both estimates hold for $k=0$, and for $k\ge1$ when every $a_i<\tau$, since then $\sum_i a_i<k\tau\le\tau k^2$. Otherwise, choose $a_z=\max_i a_i\ge\tau$ and write $S_L:=\sum_{i<z}a_i$ and $S_R:=\sum_{i>z}a_i$. Applying \eqref{eq:separation} between $a_z$ and each $a_i\ge\tau$ gives
\[
 a_i<\sum_{t=i+1}^{z-1}a_t\quad(i<z),
 \qquad
 a_i<\sum_{t=z+1}^{i-1}a_t\quad(i>z).
\]
Under either hypothesis, the left subsequence inherits the bound by preceding terms from the full sequence, so it satisfies both conditions in (ii). The right subsequence satisfies the condition in (i). If the full sequence satisfies (ii), the right subsequence also inherits the bound by following terms. Both subsequences retain \eqref{eq:separation}.

Under the hypothesis of (i), we have $a_z<S_L$. Applying (ii) to the left subsequence and (i) to the right gives
\[
 \sum_{i=1}^k a_i<2S_L+S_R
 \le2\tau\bigl((z-1)^2+(k-z)^2\bigr)
 \le2\tau k^2.
\]
Under the hypothesis of (ii), we have $a_z<\min\{S_L,S_R\}$, and both subsequences satisfy (ii). Thus
\[
 \begin{aligned}
 \sum_{i=1}^k a_i
 &<S_L+\min\{S_L,S_R\}+S_R\\
 &\le\tau\bigl((z-1)^2+\min\{(z-1)^2,(k-z)^2\}+(k-z)^2\bigr)\\
 &\le\tau(k-1)^2\le\tau k^2,
 \end{aligned}
\]
where the last line uses $\min\{(z-1)^2,(k-z)^2\}\le2(z-1)(k-z)$. This proves both estimates.

\paragraph{Applying the estimates.}
Choose $h_z=\max_t h_t\ge S'/r>\tau$. By \eqref{eq:separation}, the reversed left subsequence $(h_{z-1},\ldots,h_1)$ and the right subsequence $(h_{z+1},\ldots,h_r)$ satisfy the hypothesis of (i). Applying that estimate gives
\[
 S'-h_z\le2\tau\bigl((z-1)^2+(r-z)^2\bigr)
 \le2\tau r^2=\frac{S'}4.
\]
Hence $h_z\ge3S'/4$, as required by the second alternative.
\end{proof}

We now prove \cref{lem:repair}, first verifying the repair rule and then establishing its guarantees.

\medskip
\begin{proof}[\normalfont\bfseries Proof of \cref{lem:repair}]
\mbox{}\par\nobreak
Let the data, including the incoming pair $(\ell_1,D_1)$, satisfy the hypotheses of the lemma.
Choose $c_0=[4096(1+2C_b)]^{-2}$.
Since $\ell_1/D_1\le64C_b$ and $S\le2C_bS_0$, the definition of $S_0$ in \eqref{eq:repair-limit} gives
\[
 \left(1+\frac{\ell_1}{D_1}\right)\frac S\kappa
 \le2C_b(64C_b+1)\frac{S_0}{\kappa}=\frac1{128}\le\frac18.
\]
In particular, $S/\kappa\le1/8$.
For candidate checkpoint indices chosen by the rule, set $t_0=0$, write $b_i:=h_{t_i}$, and define
\begin{equation}\label{eq:repair-indexed-gaps}
 s_i:=\sum_{t=t_{i-1}+1}^{t_i-1}h_t,\qquad
 \chi_i:=\prod_{t=t_{i-1}+1}^{t_i-1}(1-h_t/\kappa)
 \qquad(1\le i\le q).
\end{equation}
Thus every repair gap satisfies $s_i\le S<\kappa$, and every stepsize lies in $[0,\kappa)$.

\paragraph{Step 1. Verifying the repair rule.}
\leavevmode
Define
\begin{equation*}
 j_*:=\max\left\{j\in\{1,\ldots,r\}:\sum_{t=j}^r h_t>S_0/8\right\}.
\end{equation*}
This index exists because $S\ge S_0$. Set $S':=\sum_{t=j_*}^r h_t$. By maximality of $j_*$,
\begin{equation}\label{eq:repair-suffix-tail}
 S'>\frac{S_0}{8},\qquad
 \sum_{t=j_*+1}^r h_t\le\frac{S_0}{8}.
\end{equation}
The set $\{j_*,\ldots,r\}$ is the shortest suffix used in the repair rule. If it contains a stepsize at least $S_0/16$, the rule selects $q=1$ and gives $b_1\ge S_0/16$.
Otherwise, every stepsize in this suffix is less than $S_0/16$. Apply \cref{lem:repair-pair} to this list. Its second alternative is impossible, since it would give a stepsize at least $3S'/4>3S_0/32>S_0/16$. Since the list has at most $r$ entries, the first alternative supplies two stepsizes at least $S'/(8r^2)>S_0/(64r^2)$, with the intervening stepsizes summing to at most the smaller of the two. Thus the repair rule can always be carried out.

Fix any checkpoint selection permitted by the rule. If $q=2$, then
\begin{equation}\label{eq:repair-pair-data}
 b_1,b_2\ge\frac{S_0}{64r^2},
 \qquad
 s_2=\sum_{t=t_1+1}^{t_2-1}h_t\le\min\{b_1,b_2\}.
\end{equation}
In both cases, the selected stepsizes are positive and all selected checkpoints belong to this suffix. Moreover, $t_q\ge j_*$ and \eqref{eq:repair-suffix-tail} give
\[
 \sum_{t=t_q+1}^r h_t\le\frac{S_0}{8}\le\frac{S_0}{4},
\]
which proves (iii).

\paragraph{Step 2. Positivity and bounds on $\ell_{i+1}/D_{i+1}$.}
\leavevmode
With these checkpoint indices fixed, we check that the components can be constructed from $(\ell_1,D_1)$, then bound the final relative threshold.
For $s_i$ and $\chi_i$ from \eqref{eq:repair-indexed-gaps}, the estimate \eqref{eq:local-gap-estimates} gives
\begin{equation*}
 \chi_i\ge1-\frac{s_i}{\kappa}\ge1-\frac S\kappa\ge\frac12,
 \qquad
 0\le\kappa(1-\chi_i)\le s_i
 \qquad(1\le i\le q).
\end{equation*}
We construct the components successively from $(\ell_1,D_1)$. Whenever $\ell_i\ge0$, $D_i>0$, and $\eta_i>0$, \cref{lem:huber-local} applies at $t_i$ and gives $\ell_{i+1}\ge0$, $D_{i+1}>0$. The exact formulas in \eqref{eq:huber-exact-output} give
\begin{equation}\label{eq:repair-exact-transfer}
 \begin{aligned}
 \frac{D_{i+1}}{D_i}
 &=\frac{(1-1/\kappa)b_i\chi_i\eta_i}
 {2[2+(\kappa-1)(1-\chi_i)]},\\
 \frac{\ell_{i+1}}{D_{i+1}}
 &=\frac{\kappa(1-\chi_i)}{b_i\chi_i}
 \le\frac{2s_i}{b_i}.
 \end{aligned}
\end{equation}
For the first checkpoint, the bound above gives
\begin{equation*}
 \eta_1
 =1-(1-\chi_1)\left(1+\frac{\ell_1}{D_1}\right)
 \ge1-\left(1+\frac{\ell_1}{D_1}\right)\frac S\kappa
 \ge\frac78.
\end{equation*}
Thus $D_2>0$, and \eqref{eq:repair-exact-transfer} applies to the first transfer.

If $q=2$, the relative threshold $\ell_2/D_2$ may be large, but $s_2\le b_1$ in \eqref{eq:repair-pair-data} controls the contraction loss in the second gap. Using also $s_1+s_2\le S$, we obtain
\begin{equation*}
 \begin{aligned}
 1-\eta_2
 &=(1-\chi_2)\left(1+\frac{\ell_2}{D_2}\right)
 \le\frac{s_2}{\kappa}\left(1+\frac{2s_1}{b_1}\right)\\
 &\le\frac{s_2+2s_1}{\kappa}
 \le\frac{2S}{\kappa}\le\frac14.
 \end{aligned}
\end{equation*}
Hence $\eta_2\ge3/4$ and $D_3>0$. All the repair components can therefore be constructed, with $\eta_i\ge1/2$ for every $1\le i\le q$.

Finally, \eqref{eq:repair-exact-transfer} and the checkpoint choices give
\[
 \begin{aligned}
 \frac{\ell_2}{D_2}&\le\frac{2s_1}{b_1}
 \le\frac{4C_bS_0}{S_0/16}=64C_b
 &&\text{if }q=1,\\
 \frac{\ell_3}{D_3}&\le\frac{2s_2}{b_2}\le2\le64C_b
 &&\text{if }q=2,
 \end{aligned}
\]
where we used the bound $s_2\le b_2$ in \eqref{eq:repair-pair-data}.
This proves (ii). 

\paragraph{Step 3. Bounding $D_{q+1}/D_1$ from below.}
Since $\kappa\ge4$, $1-s_i/\kappa\ge1/2$, and $\eta_i\ge1/2$, the transfer lower bound in \eqref{eq:huber-local-output} gives
\begin{equation}\label{eq:repair-coarse-transfer}
 \begin{aligned}
 \frac{D_{i+1}}{D_i}
 &\ge\frac{3b_i}{32(s_i+2)}
 \ge\frac{3b_i}{64(1+S)}\\
 &\ge\frac{b_i}{64(1+2C_b)S_0}
 \qquad(1\le i\le q).
 \end{aligned}
\end{equation}
We used $s_i+2\le2(1+S)$, followed by $S\le2C_bS_0$ and $S_0\ge1$.

If $q=1$, then $b_1\ge S_0/16$, so
\[
 \frac{D_2}{D_1}\ge\frac{1}{1024(1+2C_b)}
 \ge c_0(r+1)^{-4}.
\]
If $q=2$, then \eqref{eq:repair-pair-data} and \eqref{eq:repair-coarse-transfer} give
\[
 \frac{D_3}{D_1}
 =\frac{D_2}{D_1}\frac{D_3}{D_2}
 \ge\frac{1}{[4096(1+2C_b)r^2]^2}
 =c_0r^{-4}\ge c_0(r+1)^{-4}.
\]
Together with the positivity of $\eta_i$ and $D_{i+1}$ established in Step~2, these bounds prove (i).
\end{proof}

\end{document}